\documentclass[12pt,twoside]{amsart}
\usepackage[margin=3cm]{geometry}
\usepackage[colorlinks=false,pdfborder={0 0 0}]{hyperref}
\usepackage[english]{babel}
\usepackage{graphicx,titling}
\usepackage{float}
\usepackage{amsmath,amsfonts,amssymb,amsthm}
\usepackage{lipsum}
\usepackage[T1]{fontenc}
\usepackage{fourier}
\usepackage{color}
\usepackage[utf8]{inputenc}
\usepackage{esint}
\usepackage{caption}
\usepackage{ccicons}

\makeatletter
\def\blfootnote{\xdef\@thefnmark{}\@footnotetext}
\makeatother

\newcommand\ccnote{
    \blfootnote{\copyright\,\, Alessandro Carlotto and Mario B. Schulz}
    \blfootnote{\ccLogo\, \ccAttribution\,\, Licensed under a \href{https://creativecommons.org/licenses/by/4.0/}{Creative Commons Attribution License (CC-BY)}.}
}

\usepackage[export]{adjustbox}
\numberwithin{equation}{section}
\usepackage{setspace}

\renewcommand{\leq}{\leqslant}

\renewcommand{\geq}{\geqslant}
\renewcommand{\mathbb}{\varmathbb}
\usepackage{fancyhdr}
\newtheorem{theorem}{Theorem}[section]
\newtheorem{lemma}[theorem]{Lemma}

\theoremstyle{remark}
\newtheorem{remark}[theorem]{Remark}
\usepackage[babel=true]{microtype}
\usepackage{mathtools}
\usepackage{enumitem}
\usepackage{tikz} 
\usepackage{tikz-3dplot} 
\usetikzlibrary{plotmarks}
\usepackage[initials]{amsrefs}

\newtheorem{conj}[theorem]{Conjecture} 

\providecommand{\R}{\mathbb{R}}
\providecommand{\N}{\mathbb{N}}
\providecommand{\Sp}{\mathbb{S}}
\providecommand{\st}{\, :\ }
  
\DeclarePairedDelimiter\abs{\lvert}{\rvert}

\DeclarePairedDelimiter\interval{]}{[}
\DeclarePairedDelimiter\Interval{[}{[}
\DeclarePairedDelimiter\intervaL{]}{]}
\DeclarePairedDelimiter\IntervaL{[}{]}

\DeclareMathOperator{\Ogroup}{O}
\DeclareMathOperator{\SOgroup}{SO}
\DeclareMathOperator{\SUgroup}{SU}

\makeatletter
\tikzset{
    scale plot marks/.is choice,
    scale plot marks/true/.style={},	
    scale plot marks/false/.code={
        \def\pgfuseplotmark##1{\pgftransformresetnontranslations\csname pgf@plot@mark@##1\endcsname}
    },
every mark/.append style={scale plot marks=false},
plus/.style={mark=+,mark size=2.25pt},
vdash/.style={mark=|,mark size=2.25pt},
hdash/.style={mark=-,mark size=2.25pt},
bullet/.style={mark=*,mark size=1.125pt},
trajectory/.style={semithick},
torustrajectory/.style={semithick,color={cmyk,1:magenta,0.5;cyan,1}},
spheretrajectory/.style={semithick,color={cmyk,1:magenta,0.5;yellow,1}},
}
\makeatother

\address{Alessandro Carlotto, Universit\`{a} di Trento,
Dipartimento di Matematica,
via Sommarive 14,
38123 Povo di Trento,
Italy}
\email{alessandro.carlotto@unitn.it}
\medskip
\address{Mario B. Schulz, University of M\"unster, 
Mathematisches Institut, 
Einsteinstrasse 62,
48149 M\"unster,
Germany} 
\email{mario.schulz@uni-muenster.de} 

\begin{document}

\thispagestyle{empty}

\begin{minipage}{0.28\textwidth}
\begin{figure}[H]
\includegraphics[width=2.5cm,height=2.5cm,left]{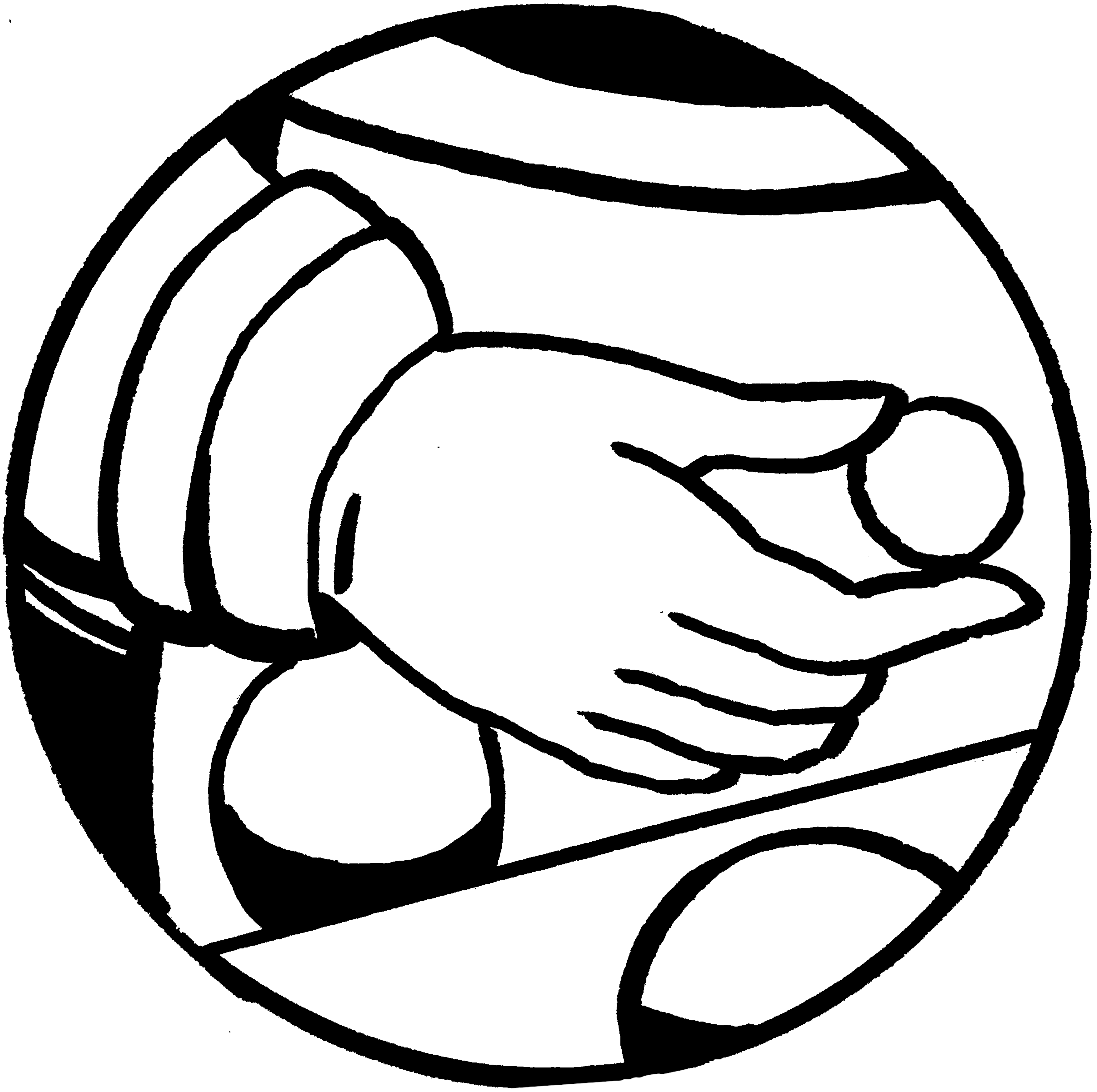}
\end{figure}
\end{minipage}
\begin{minipage}{0.7\textwidth} 
\begin{flushright}
Ars Inveniendi Analytica (2023), Paper No. 8, 33 pp.
\\
DOI 10.15781/gtrt-v523
\\
ISSN: 2769-8505
\end{flushright}
\end{minipage}

\ccnote

\vspace{1cm}


\begin{center}
\begin{huge}
\textit{Minimal hypertori in the four-dimensional sphere}
\end{huge}
\end{center}

\vspace{1cm}


\begin{minipage}[t]{.45\textwidth}
\begin{center}
{\large{\bf{Alessandro Carlotto}}} \\
\vskip0.15cm
\footnotesize{Universit\`{a} di Trento}
\end{center}
\end{minipage}
\hfill
\noindent
\begin{minipage}[t]{.45\textwidth}
\begin{center}
{\large{\bf{Mario B. Schulz}}} \\
\vskip0.15cm
\footnotesize{University of M\"unster}
\end{center}
\end{minipage} 

\vspace{1cm}


\begin{center}
\noindent \em{Communicated by Sigurd Angenent}
\end{center}
\vspace{1cm}


\noindent \textbf{Abstract.} \textit{We prove that the four-dimensional round sphere contains a minimally embedded hypertorus, as well as infinitely many, pairwise non-isometric, immersed ones. Our analysis also yields infinitely many, pairwise non-isometric, minimally embedded hyperspheres and thus provides a self-contained solution to Chern's spherical Bernstein conjecture in dimensions four and six.}
\vskip0.3cm

\vspace{0.5cm}


\section{Introduction}\label{sec:intro}

Lawson \cite{Lawson1970} proved in 1970 that closed surfaces of any orientable topological type can be minimally embedded in the three-dimensional round sphere, and that, for non-orientable types, there always exists a minimal immersion with the sole exception of $\R\mathbb{P}^2$. For what concerns the former case, such minimal embeddings are known to be non-unique for surfaces of genus $g$ that is not prime, while uniqueness theorems are only at disposal when $g=0$ (indeed by an Hopf differential type argument due to Almgren \cite{Alm66}, and Calabi \cite{Calabi1967}) and when $g=1$ (by much more recent work of Brendle, see \cite{Brendle2013}).
It is therefore quite surprising that, if we look at higher-dimensional spheres, remarkably little is known: we simply lack any sort of characterisation of those $m$-dimensional manifolds that can be minimally embedded, or minimally immersed, in $\Sp^{m+1}$ (a notation we shall henceforth adopt to denote the standard, round sphere in $\R^{m+2}$).

A particularly interesting case is that of the four-dimensional sphere. To the best of our knowledge, till about a decade ago only \emph{three} topological types were known to enjoy such a property: the three-dimensional sphere (which obviously admits a totally geodesic realisation), the product $S^2\times S^1$ (for which one can consider the corresponding Clifford-type hypersurface) and a quotient of $\SOgroup(3)$ by a suitable action of the group $\mathbb{Z}_2\times \mathbb{Z}_2$ (the minimal embedding corresponds to one of the so-called Cartan isoparametric hypersurfaces, cf. \cite{KiNak87} and \cite{Sol90}). In fact, as we shall explain in the sequel of this introduction, the first two out of these three topological examples are actually known to be minimally embeddable in $\Sp^4$ in infinitely many, pairwise non-isometric, different ways (see \cites{Hsiang1983, Hsiang1987}). 

In recent years, we have witnessed the impressive development of the min-max theory for the area functional, which has led (among other things) to the positive solution of Yau's conjecture asserting the existence of infinitely many closed minimal hypersurfaces in any compact manifold without boundary, possibly accounting for a singular set in ambient dimension larger than seven, see in particular \cite{MarNev17}, \cite{IriMarNev18}, \cite{MarNevSon19}, \cite{LioMarNev18} and \cite{Son18}. One should note that it follows from work of Zhou \cite{Zhou20} on the so-called multiplicity one conjecture, see Theorem B therein, that every Riemannian manifold, of dimension less than eight and positive Ricci curvature, contains a sequence of closed minimal hypersurfaces with areas diverging as prescribed by the Weyl law for the corresponding min-max spectrum. This conclusion obviously applies, for instance, to the case of $\Sp^4$, however it is well-known that these variational methods \emph{do not}, except for very few special cases, provide topological control of the minimal hypersurfaces they produce, and thus it is still unknown whether such aforementioned hypersurfaces do indeed exhibit increasing topological complexity.

All that said, we wish to present here some advances on the question we stated above, specified to the most fundamental case of higher-dimensional tori (henceforth simply referred to as \emph{hypertori}), i.\,e. products of the form $T^m=S^1\times\ldots \times S^1$ of $m\geq 2$ copies of $S^1$. 
In particular, we shall prove that there exists a minimal embedding of $T^3$ into the round four-dimensional sphere, as well as infinitely many, pairwise non-isometric, immersed ones. 
As we will better explain below, when contextualising our results and methodology, this seems to be the first construction of a minimal embedding of $T^m$ into $\Sp^{m+1}$ for any $m\geq 3$: in other words, prior to our work the only known minimal embedding of an $m$-dimensional torus into a round $(m+1)$-dimensional sphere was the standard Clifford torus in $\Sp^3$. In particular, we hereby provide an affirmative answer to the (somewhat more general) question posed by Choe and Fraser in \cite{ChoeFra18}.
To avoid misunderstandings, we wish to stress that all immersions we study in the present article happen in codimension equal to one: if one allows for higher codimension, then on the one hand cheap examples exist in abundance and, on the other hand, one faces a much richer theory (we refer the reader to e.\,g. the seminal paper \cite{Bryant1982} by Bryant).
Actually, the results claimed above follow, as special cases (for $n=2$), from the combination of two somewhat more general statements. The first one reads as follows:

\begin{theorem}\label{thm:main}
For any $2\leq n\in\N$ there exists a minimal embedding of $S^{n-1}\times S^{n-1}\times S^1$ in $\Sp^{2n}$.  
\end{theorem}

We will later elaborate about why we expect the minimal embedding of $T^3$ in the round four-dimensional sphere (the case $n=2$ in the statement) to be unique up to ambient isometries, see Conjecture \ref{conj:unique}; however it is a matter of fact that if one allows for \emph{immersions} then the landscape changes quite a lot, as witnessed by this second main result:

\begin{theorem}\label{thm:mainImmersed}
Let $n\in\{2,3\}$. Then, there exist infinitely many, pairwise non-isometric, minimal immersions of $S^{n-1}\times S^{n-1}\times S^1$ in $\Sp^{2n}$.
\end{theorem}

We remark that, in the previous statement and throughout this paper, two (possibly immersed) hypersurfaces are said isometric if they are related by an ambient isometry; some authors would rather employ the word \emph{congruent} for this purpose, although the terminology is not completely standard.

Before adding some significant, additional remarks on these two statements we need to briefly digress on the methodology employed in their proofs. In short, we employ equivariant techniques (along the lines pioneered in \cite{HsiLaw71}), i.\,e. we consider a suitable isometric group action on $\R^{2n+1}$ and study the reduced equation one obtains in the quotient space $\Sp^{2n}/G$; more specifically, the group $G$ is the product $\Ogroup(n)\times\Ogroup(n)$ and we regard $\R^{2n+1}=\R^n\times\R^n\times \R$ (namely: we consider the representation $\rho_n\oplus\rho'_n\oplus\textbf{1}$ that is the outer direct sum of the standard representations of $\Ogroup(n)$ as a group acting on $\R^n$). 
The idea of considering this group action, and the corresponding reduced minimal surface equation, is very natural and was indeed successfully employed in 1983 by Hsiang \cite{Hsiang1983} to disprove (for spheres of dimension $4$ and $6$) Chern's spherical Bernstein conjecture, namely to answer the question whether a minimal embedding of the $m$-dimensional sphere into $\Sp^{m+1}$ should necessarily be totally geodesic (hence a standard equator). 
As recalled at the beginning of this introduction, a classical result obtained independently in \cite{Alm66} and \cite{Calabi1967} provides an affirmative answer for $m=2$, so it was quite notable, at least in certain respects, that \emph{infinitely many counterexamples} (i.\,e. infinitely many non-equatorial embeddings) can be obtained in higher-dimension. 

At a technical level, Hsiang's work amounts to proving the existence of certain trajectories for a planar dynamical system of second order (which we will describe in Section \ref{sec:setup}). 
One of the reasons why this is not a trivial task is that the ODE in question becomes singular at the boundary of the quotient space, and indeed Hsiang's construction relies on earlier work \cite{HsiHsi82} where the authors prove a local existence and uniqueness theorem which applies to the problem in question, thereby allowing for a \emph{singular shooting method}. Secondly, Hsiang crucially exploits a blow-up trick, which allows to bypass some analytic challenges of the problem, by ultimately appealing to a pre-existing, important result by Bombieri--De Giorgi--Giusti \cite{BdGG69} concerning the corresponding limit problem in $\R^m$, which happens to be reducible to a first-order planar ODE due to its scaling invariance.

Now, the conclusion of Theorem \ref{thm:main} comes from proving the existence of \emph{closed periodic orbits} for the reduced minimal surface equation and, as such, is patently \emph{not tractable} with Hsiang's methods. Such orbits, if any, keep away from the boundary of the quotient space and thus one cannot shoot from there, and also our problem cannot be tackled through blow-up shortcuts. That being said, we took a different path: we recasted the original equation as a $3 \times 3$ ODE system (thus to a non-planar problem) and developed a careful study of various classes of solutions, which ultimately led to the result by means of a suitable continuity argument. Such an analysis requires a priori estimates that turn out to be rather subtle, and yet the resulting outcome is a rather direct, self-contained proof. In fact, the proof of Theorem \ref{thm:mainImmersed} requires a significant amount of additional work, but forced us to derive some ancillary lemmata that immediately allow to gain, as a byproduct, a somewhat simpler proof of Hsiang's result:

\begin{theorem}\label{thm:mainHyperspheres}
Let $n\in\{2,3\}$. Then, there exist infinitely many, pairwise non-isometric, minimal, nonequatorial embeddings of $S^{2n-1}$ in $\Sp^{2n}$. 
\end{theorem}

Concerning both Theorem \ref{thm:mainImmersed} and \ref{thm:mainHyperspheres} we wish to add a comment about the dimensional assumptions: ample numerical evidence seems to suggest that the infinitely many trajectories corresponding to such minimal immersions of $S^{n-1}\times S^{n-1}\times S^1$ and minimal embeddings of hyperspheres, respectively, cease to exist when $n\geq 4$, and thus lead us to believe that our results should be sharp in that respect as well. These phenomena are described at the end of Section \ref{sec:immersed}, see also Conjecture \ref{conj:n>3} therein.

Getting back to hypertori, we would like to mention what follows: in the concluding remarks at page 550 of \cite{HsiHsi80} it is claimed that the methods developed there would also allow to prove the existence of minimal embeddings of $S^{n-1}\times S^{n-1}\times S^{n-1}\times S^1$ in the round sphere $\Sp^{3n-1}$ in $\R^{3n}$ for any $n\geq 2$ and, thus, that such an argument would also provide (in the sole case $n=2$) a minimal embedding of $T^4$ in $\Sp^5$. Although there are indeed some formal similarities between such a problem and those explicitly studied in \cite{HsiHsi80} (i.\,e. in the symmetric spaces $\SUgroup(3)/\SOgroup(3)$, $\SUgroup(3)$,  $\SUgroup(6)/\operatorname{Sp}(3)$ and $E_6/F_4$), and the structure of the orbit space is analogous (the quotient space is a `triangular domain' which, by virtue of symmetry arguments, can be split into six pairwise isometric subdomains, where one needs to construct a free boundary arc of prescribed geodesic curvature) we note that the ODE analysis is actually \emph{not} carried through there, nor in any later article we are aware of. For this reason, and for the sake of completeness, we show in Appendix \ref{sec:HsiangClaim} how some of the arguments we present in the first part of Section \ref{sec:embedded} can be modified, at modest additional cost, so to confirm the existence of a minimal embedding of $T^4$ in $\Sp^5$; as the reader will see, this result is arguably much simpler to prove than Theorem \ref{thm:main}. 
In any event, the authors of \cite{HsiHsi80} also wrote `\emph{One does not know whether there are minimal imbeddings of codimension-one torus in $\Sp^m$ for $m\neq 3,5$}' and, to our knowledge, our result is the first advance on this matter since 1980.

\bigskip\noindent\textbf{Acknowledgements.} The authors would like to thank Lucas Ambrozio, Renato Bettiol and Ben Sharp for a number of stimulating conversations, and to the anonymous referees for their valuable suggestions. This project has received funding from the European Research Council (ERC) under the European Union’s Horizon 2020 research and innovation programme (grant agreement No. 947923). 
The research of M.\,S. was partly funded by the EPSRC grant EP/S012907/1 
and by the Deutsche Forschungsgemeinschaft (DFG, German Research Foundation) under Germany's Excellence Strategy EXC 2044 -- 390685587, Mathematics M\"unster: Dynamics--Geometry--Structure, and the Collaborative Research Centre CRC 1442, Geometry: Deformations and Rigidity.

\section{Setup and preliminaries}\label{sec:setup}

Let $2\leq n\in\N$. 
The group $G=\Ogroup(n)\times\Ogroup(n)$ acts on the sphere $\Sp^{2n}$ via the representation $\rho_n\oplus\rho_n'\oplus \textbf{1}$. 
The quotient $\Sp^{2n}/G$ is equipped with spherical coordinates 
$(r,\theta)\in[0,\pi]\times[0,\frac{\pi}{2}]$, and the orbital distance metric 
\begin{align*}
g=dr^2+(\sin r)^2d\theta^2.
\end{align*}
We recall e.\,g. from \cite{HsiLaw71} that a curve $s\mapsto\gamma(s)=(r(s),\theta(s))$ in $\Sp^{2n}/G$ parametrised by arc-length lifts to a minimal hypersurface in $\Sp^{2n}$ if and only if it has vanishing geodesic curvature in the conformal metric $V^2g$ where $V=V(r,\theta)$ stands for the volume of the pre-image of the point $(r,\theta)$, so in our case $V(r,\theta)=\omega^{2}_{n-1} \sin^{2n-2}(r)\sin^{n-1}(2\theta)/2^{n-1}$ where $\omega_{n-1}$ denotes the volume of the unit sphere in $\R^n$. As a result, since with respect to the unit normal
\[
\eta=-\sin(r)\frac{d\theta}{ds}\frac{\partial}{\partial r}+\frac{1}{\sin(r)}\frac{dr}{ds}\frac{\partial}{\partial\theta}
\]
the geodesic curvature in metric $g$ of the curve in question equals
\[
\kappa_g=\frac{d\alpha}{ds}+\cos(r)\frac{d\theta}{ds}
\]
the well-known equation $\kappa_{V^2g}=\kappa_g-\nabla_{\eta}\ln V$ implies that the lift is a $G$-equivariant minimal hypersurface if and only if
\begin{align}\label{eqn:ode}
\frac{d\alpha}{ds}-\frac{(2n-2)\cot(2\theta)}{\sin(r)}\frac{dr}{ds}+(2n-1)\cos(r)\frac{d\theta}{ds}=0, 
\end{align}
where $\alpha$ is the (signed) angle between the vectors $d\gamma/ds$ and $\partial/\partial r$; to avoid ambiguities we agree that $\alpha=\pi/2$ (respectively $\alpha=-\pi/2$) for curves of the form $s\mapsto (r_0, s)$ (respectively $s\mapsto (r_0, -s)$) for any fixed $r_0\in\interval{0,\pi}$.

Parametrisation by arc-length implies    
\begin{align}\label{eqn:constraint}
1&=g\biggl(\frac{d\gamma}{ds},\frac{d\gamma}{ds}\biggr)=\biggl(\frac{dr}{ds}\biggr)^2
+(\sin r)^2\biggl(\frac{d\theta}{ds}\biggr)^2, &
\cos(\alpha)&=\frac{dr}{ds}, & 
\sin(\alpha)&=(\sin r)\frac{d\theta}{ds}. 
\end{align}
We wish to rewrite the differential equation above as an autonomous $3\times 3$ ODE system, i.\,e. to have it in the form  $dU/ds=f(U)$ where $U$ varies in a suitable subset of $\R^3$. Indeed, we can take $U=(r,\theta,\alpha)$ and thus consider the system of equations 
\begin{align}
\label{eqn:dr/ds}
\frac{dr}{ds}&=\cos(\alpha), \\[1ex]
\label{eqn:dt/ds}
\smash{\left\{\vphantom{
\begin{aligned}
\frac{dr}{ds}&=\cos(\alpha), \\[1ex]
\frac{d\theta}{ds}&=\frac{\sin(\alpha)}{\sin(r)}, \\[1ex]
\frac{d\alpha}{ds}&=\frac{(2n-2)\cot(2\theta)}{\sin(r)}\cos(\alpha)-(2n-1)\cot(r)\sin(\alpha).
\end{aligned}}\right.}\,
\frac{d\theta}{ds}&=\frac{\sin(\alpha)}{\sin(r)}, \\[1ex]
\label{eqn:da/ds}
\frac{d\alpha}{ds}&=\frac{(2n-2)\cot(2\theta)}{\sin(r)}\cos(\alpha)-(2n-1)\cot(r)\sin(\alpha).
\end{align}

In Section \ref{sec:embedded} we will design a suitable shooting method to construct solutions of the system \eqref{eqn:dr/ds},\eqref{eqn:dt/ds},\eqref{eqn:da/ds} such that the $(r,\theta)$-trajectory is simple and periodic. Instead, in Section \ref{sec:immersed} we will obtain on the one hand periodic orbits (with suitable symmetries) and any odd number of self-intersections
and on the other hand central-symmetric trajectories, without self-intersections, reaching the boundary of the quotient space orthogonally. 
When lifted back to $\Sp^{2n}$ these curves provide the minimally immersed hypersurfaces and the minimally embedded hyperspheres claimed in the statements of Theorem \ref{thm:mainImmersed} and \ref{thm:mainHyperspheres}, respectively.

\section{Existence of a minimally embedded hypertorus}\label{sec:embedded}

This section is devoted to the proof of Theorem \ref{thm:main}, but also serves as a preparation for some of the results we will present in the following sections. 

\begin{lemma}[Well-posedness]\label{lem:well-posed}
Given any $r_0\in\interval{0,\pi/2}$ there exists $0<s_*<\infty$ such that the system \eqref{eqn:dr/ds},\eqref{eqn:dt/ds},\eqref{eqn:da/ds} has a unique solution 
\[
(r,\theta,\alpha)\colon[0,s_*]\to B\vcentcolon=\intervaL{0,\tfrac{\pi}{2}}\times\intervaL{0,\tfrac{\pi}{4}}\times[-\tfrac{\pi}{2},0]
\] with initial data $(r,\theta,\alpha)(0)=(r_0,\pi/4,-\pi/2)$ on which the solution depends continuously such that $r(s_*)=\pi/2$ or $\alpha(s_*)=0$; in either case $\theta(s_*)>0$. 
Moreover,  $dr/ds\geq0$, $d\theta/ds\leq0$ and $d\alpha/ds\geq0$.
\end{lemma}

\begin{proof}
As long as \((r,\theta,\alpha)\in B \), equations 
\eqref{eqn:dr/ds}--\eqref{eqn:da/ds} imply  
$dr/ds\geq0$, $d\theta/ds\leq0$ and $d\alpha/ds\geq0$ (i.\,e. the motion is weakly monotone in each of the three coordinates; note also that there are no stationary points and \eqref{eqn:constraint} holds).
Therefore, solutions with initial data $(r,\theta,\alpha)(0)=(r_0,\pi/4,-\pi/2)$ stay in $B$ at least for a short time and can leave $B$ only via $r=\frac{\pi}{2}$, $\theta=0$ or $\alpha=0$ (see Figure \ref{fig:shooting}). 

Equation \eqref{eqn:da/ds} implies $d\alpha/ds(0)=(2n-1)\cot(r_0)>0$ since we assume $r_0\in\interval{0,\pi/2}$. 
Then, by monotonicity of $\alpha$ along the motion, $\cot(\alpha)<0$ for any $s>0$. 
If the trajectory in question intersects $\{\theta=\pi/8\}$ at some $s_0\in\interval{0,s_*}$ we can find $\delta_0>0$ such that $\cot(\alpha)\leq-\delta_0$ for all $s\in\Interval{s_0,s_*}$. 
Equations \eqref{eqn:dt/ds} and \eqref{eqn:da/ds} then imply 
\begin{align*}
\frac{d\alpha}{ds}
&=\Bigl((2n-2)\cot(2\theta)\cot(\alpha)-(2n-1)\cos(r) \Bigr)\frac{d\theta}{ds}
\geq-(2n-2)\cot(2\theta)\delta_0\frac{d\theta}{ds} 
\end{align*} 
for all $s\in\Interval{s_0,s_*}$. 
Integrating this differential inequality over an interval $[s_0,s]$ gives 
    \[
    \frac{\pi}{2\delta_0}\geq \frac{\alpha(s)-\alpha(s_0)}{\delta_0}\geq (n-1) \log\left(\frac{\sin(\pi/4)}{\sin(2\theta(s))}\right).
    \]
This rough bound shows that the trajectory in question must stay at positive distance away from $\{\theta=0\}$. 
Hence, again appealing to its monotonicity, the trajectory will either exit $B$ through the roof (namely: reaching $\alpha=0$) or through the side (namely: reaching $r=\pi/2$). Of course, if the trajectory does not even reach $\theta=\pi/8$ then the same conclusions still hold true.
  
In particular, for any $r_0\in\interval{0,\pi/2}$ and any sufficiently small $\eta>0$, the image through the flow map of $[r_0-\eta,r_0+\eta]$ will stay at positive distance from the plane $\{\theta=0\}$: in that region, the vector field generating the flow (cf. right-hand side of \eqref{eqn:dr/ds},\eqref{eqn:dt/ds},\eqref{eqn:da/ds}) is smooth and uniformly bounded. Hence given any $\varepsilon>0$ one can find $\delta\in (0,\eta)$ such that trajectories emanating from initial conditions less than $\delta$ apart will leave $B$ at times less than $\varepsilon$ apart and will be $\varepsilon$-close in $C^{\infty}$ for all times of their motion (in $B$). Thereby the proof is complete.
\end{proof}

\begin{lemma}\label{lem:dipping_down}
Given $r_0\in\interval{0,\pi/8}$ let $(r,\theta,\alpha)\colon[0,s_*]\to B$ be the solution constructed in Lemma \ref{lem:well-posed} and let $s_1\in\intervaL{0,s_*}$ be arbitrary. 
If $r(s_1)\geq2r_0$ then 
\begin{align*}
\theta(s_1)<\frac{\pi}{4}-\frac{1}{6n}.
\end{align*}
\end{lemma}

\begin{proof}
Let $\delta=1/(6n)$ be fixed and assume there exists $s_1\in\intervaL{0,s_*}$ with $r(s_1)\geq 2r_0$. (Otherwise there is nothing to prove.)
Towards a contradiction, suppose $\theta(s_1)\geq\pi/4-\delta$. 
By monotonicity of $\theta$, we then have $\theta(s)\geq\pi/4-\delta$ for all $s\in[0,s_1]$. Hence, 
\begin{align*}
0\leq\cot(2\theta)&\leq\tan(2\delta)=\vcentcolon b, &
\cot(\alpha)&\leq0, &
0&\leq\cos(r)\leq1, & 
\frac{d\theta}{ds}&\leq0 
\end{align*}
for all $s\in[0,s_1]$. 
Equations \eqref{eqn:dt/ds} and \eqref{eqn:da/ds} then imply  
\begin{align}
\frac{d\alpha}{ds}
&=\Bigl((2n-2)\cot(2\theta)\cot(\alpha)-(2n-1)\cos(r) \Bigr)\frac{d\theta}{ds}
\label{eqn:da/dt}
\leq(2n-1)\Bigl(b\cot(\alpha)-1\Bigr)\frac{d\theta}{ds}.
\end{align}
We observe that the function $f\colon\interval{-\pi/2,0}\to\R$ given by $f(x)=(b\cot(x)-1)^{-1}$ is the derivative of
\begin{align*}
F(x)
&=-\frac{b\log\bigl(b\cos(x)-\sin(x)\bigr)+x}{b^2+1}.
\end{align*}
Multiplying estimate \eqref{eqn:da/dt} by $f(\alpha)<0$ and integrating from $s=0$ to $s_1$ we obtain 
\begin{align}\label{eqn:20210415}
F\bigl(\alpha(s_1)\bigr)-F\bigl(-\tfrac{\pi}{2}\bigr)
\geq(2n-1)\bigl(\theta(s_1)-\tfrac{\pi}{4}\bigr)\geq-(2n-1)\delta.
\end{align}
Now, the function $[-\pi/2,0]\ni x\mapsto b\cos(x)-\sin(x)$ is concave and 
is checked to attain its minimum $b$ at $x=0$. 
Hence,  
\begin{align*}
F(\alpha)&\leq\frac{-b\log(b)-\alpha}{b^2+1}, &
F(-\tfrac{\pi}{2})
&=\frac{\pi}{2(b^2+1)}
\end{align*}
so estimate \eqref{eqn:20210415} implies 
\(
-b\log(b)-\alpha(s_1)-\frac{\pi}{2}
\geq-(b^2+1)(2n-1)\delta
\)
which is equivalent to
\begin{align}\label{eqn:alpha(s1)}
\alpha(s_1)&\leq-b\log(b)+(b^2+1)(2n-1)\delta-\frac{\pi}{2}.
\end{align}
Recalling that we had set $b=\tan(2\delta)$, it is evident that if $\delta>0$ is sufficiently small, then \eqref{eqn:alpha(s1)} implies $\alpha(s_1)<-\pi/4$; in fact, one can check $\delta=1/(6n)$ is sufficient to that aim.  
By monotonicity of $\alpha$, we have  
$\alpha(s)<-\pi/4$ for all $s\in[0,s_1]$ and equations \eqref{eqn:dr/ds} and \eqref{eqn:dt/ds} then imply 
\begin{align}\label{eqn:dt/dr}
\frac{d\theta}{ds}
&=\frac{\tan(\alpha)}{\sin(r)}\frac{dr}{ds}
\leq \frac{-1}{\sin(r)}\frac{dr}{ds}.
\end{align}
A primitive of $\interval{0,\pi/2}\ni r\mapsto-1/\sin(r)$
is given by the function $-\log(\tan(r/2))$. 
Since by assumption $r(s_1)\geq 2r_0=2r(0)$, integrating estimate \eqref{eqn:dt/dr} from $s=0$ to $s_1$ yields
\begin{align}\label{eqn:20210918-ts1}
\theta(s_1)-\frac{\pi}{4}
\leq\log\biggl(\frac{\tan(r_0/2)}{\tan(r_0)}\biggr)\leq\log\biggl(\frac{1}{2}\biggl)<-\delta
\end{align}
which gives a contradiction to our initial assumption.
\end{proof}

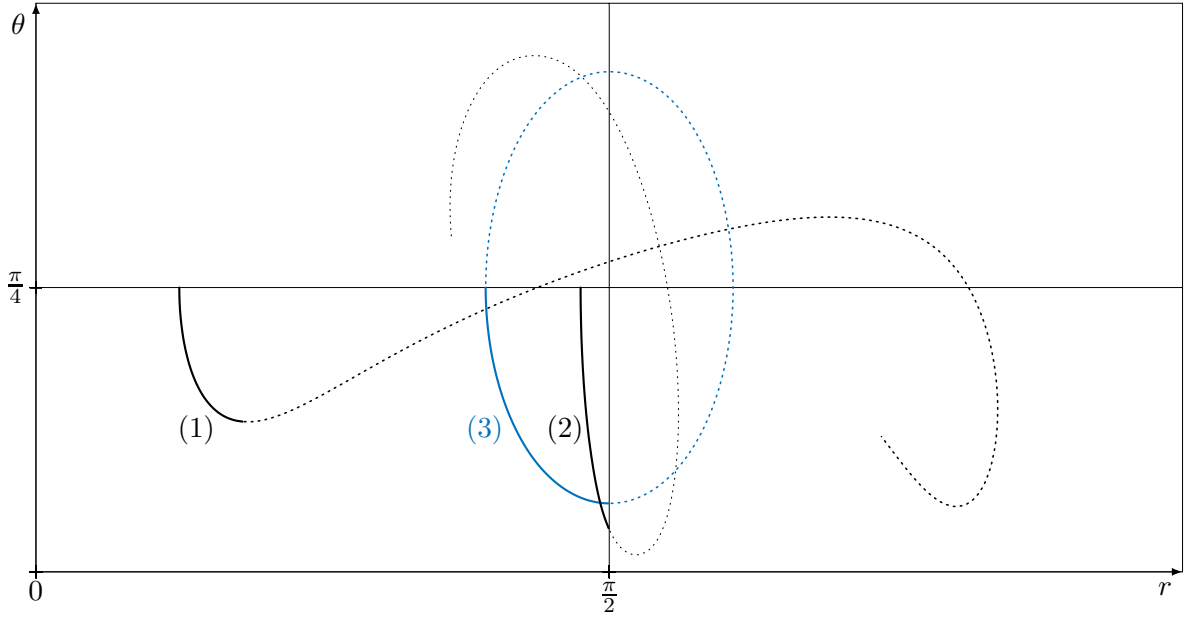
\begin{figure}\centering
\begin{tikzpicture}[line cap=round,line join=round,baseline={(0,0)},scale=\textwidth/3.5cm]
\draw[very thin](0,0)grid [xstep=pi/2,ystep=pi/4](pi,pi/2);
\draw[-latex](0,0)--(pi,0)node[below left]{$r$};
\draw[-latex](0,0)--(0,pi/2)node[below left]{$\theta$};
\draw plot[plus](0,0)node[below]{$0$};
\draw plot[plus](pi/2,0)node[below]{$\frac{\pi}{2}$};
\draw plot[plus](0,pi/4)node[left]{$\frac{\pi}{4}$};
\draw (0.44,pi/8)node{(1)};
\draw (1.45,pi/8)node{(2)};
\draw[torustrajectory](1.23,pi/8)node{(3)};
\draw[smooth,dotted,semithick]plot coordinates {
(0.5521,0.4168) (0.5824,0.4143) (0.6127,0.4173) (0.6428,0.4244) (0.6724,0.4345) (0.7017,0.4467) (0.7306,0.4604) (0.7593,0.4751) (0.7876,0.4904) (0.8157,0.5062) (0.8437,0.5220) (0.8715,0.5380) (0.8992,0.5539) (0.9268,0.5696) (0.9544,0.5852) (0.9820,0.6005) (1.0095,0.6155) (1.0371,0.6303) (1.0647,0.6447) (1.0923,0.6589) (1.1200,0.6727) (1.1477,0.6863) (1.1754,0.6995) (1.2032,0.7125) (1.2311,0.7252) (1.2590,0.7376) (1.2870,0.7497) (1.3151,0.7615) (1.3432,0.7731) (1.3715,0.7845) (1.3997,0.7955) (1.4281,0.8064) (1.4565,0.8169) (1.4850,0.8273) (1.5136,0.8374) (1.5423,0.8472) (1.5710,0.8569) (1.5998,0.8663) (1.6287,0.8754) (1.6577,0.8843) (1.6867,0.8930) (1.7159,0.9014) (1.7451,0.9095) (1.7744,0.9173) (1.8038,0.9249) (1.8333,0.9321) (1.8628,0.9390) (1.8925,0.9455) (1.9222,0.9516) (1.9521,0.9573) (1.9820,0.9625) (2.0120,0.9671) (2.0421,0.9712) (2.0722,0.9746) (2.1025,0.9773) (2.1327,0.9791) (2.1630,0.9800) (2.1933,0.9798) (2.2236,0.9784) (2.2538,0.9757) (2.2840,0.9714) (2.3139,0.9653) (2.3436,0.9572) (2.3730,0.9469) (2.4020,0.9339) (2.4304,0.9181) (2.4581,0.8991) (2.4849,0.8764) (2.5105,0.8494) (2.5348,0.8179) (2.5574,0.7816) (2.5779,0.7403) (2.5960,0.6942) (2.6112,0.6433) (2.6231,0.5878) (2.6312,0.5285) (2.6350,0.4668) (2.6339,0.4045) (2.6275,0.3439) (2.6153,0.2881) (2.5973,0.2405) (2.5735,0.2051) (2.5454,0.1850) (2.5153,0.1807) (2.4854,0.1897) (2.4573,0.2078) (2.4311,0.2315) (2.4065,0.2583) (2.3831,0.2867) (2.3605,0.3156) (2.3384,0.3448) (2.3166,0.3737) 
};
\draw[smooth,thick]plot coordinates {
(0.3927,0.7854) (0.3935,0.7461) (0.3960,0.7074) (0.4001,0.6700) (0.4056,0.6343) (0.4127,0.6008) (0.4210,0.5698) (0.4306,0.5417) (0.4413,0.5166) (0.4529,0.4945) (0.4654,0.4755) (0.4786,0.4593) (0.4924,0.4459) (0.5066,0.4352) (0.5213,0.4270) (0.5361,0.4211) (0.5511,0.4172) (0.5661,0.4151) 
};
\draw[smooth,torustrajectory,dotted,semithick]plot coordinates {
(1.5763,0.1891) (1.6054,0.1919) (1.6340,0.1990) (1.6614,0.2098) 
(1.6873,0.2241) (1.7114,0.2411) (1.7336,0.2604) (1.7542,0.2816) (1.7731,0.3044) (1.7905,0.3286) (1.8064,0.3538) (1.8210,0.3801) (1.8342,0.4071) (1.8463,0.4348) (1.8571,0.4632) (1.8669,0.4920) (1.8756,0.5214) (1.8833,0.5511) (1.8899,0.5811) (1.8956,0.6114) (1.9004,0.6420) (1.9042,0.6727) (1.9070,0.7036) (1.9090,0.7346) (1.9100,0.7657) (1.9101,0.7968) (1.9093,0.8279) (1.9076,0.8589) (1.9050,0.8898) (1.9014,0.9206) (1.8969,0.9512) (1.8915,0.9816) (1.8851,1.0117) (1.8777,1.0415) (1.8693,1.0710) (1.8598,1.1000) (1.8492,1.1285) (1.8375,1.1564) (1.8245,1.1836) (1.8103,1.2100) (1.7949,1.2356) (1.7779,1.2602) (1.7594,1.2835) (1.7392,1.3051) (1.7173,1.3249) (1.6936,1.3424) (1.6684,1.3573) (1.6416,1.3690) (1.6136,1.3772) (1.5847,1.3812) (1.5556,1.3811) (1.5267,1.3767) (1.4985,1.3682) (1.4713,1.3559) (1.4456,1.3403) (1.4219,1.3220) (1.4003,1.3019) (1.3804,1.2800) (1.3622,1.2566) (1.3454,1.2319) (1.3301,1.2062) (1.3161,1.1797) (1.3033,1.1523) (1.2918,1.1243) (1.2813,1.0958) (1.2720,1.0667) (1.2638,1.0373) (1.2565,1.0074) (1.2502,0.9773) (1.2449,0.9468) (1.2405,0.9162) (1.2371,0.8854) (1.2346,0.8545) (1.2330,0.8235) (1.2323,0.7924) 
};
\draw[smooth,torustrajectory,thick]plot coordinates {
(1.2321,0.7854) (1.2323,0.7699) (1.2326,0.7545) (1.2332,0.7390) (1.2340,0.7236) (1.2350,0.7082) (1.2363,0.6928) (1.2378,0.6774) (1.2395,0.6621) (1.2414,0.6468) (1.2436,0.6316) (1.2460,0.6164) (1.2486,0.6013) (1.2515,0.5862) (1.2546,0.5712) (1.2580,0.5563) (1.2616,0.5414) (1.2655,0.5267) (1.2696,0.5120) (1.2740,0.4974) (1.2786,0.4830) (1.2835,0.4686) (1.2887,0.4544) (1.2941,0.4403) (1.2998,0.4264) (1.3059,0.4126) (1.3122,0.3990) (1.3188,0.3856) (1.3257,0.3723) (1.3329,0.3593) (1.3405,0.3465) (1.3484,0.3339) (1.3567,0.3216) (1.3654,0.3096) (1.3744,0.2979) (1.3839,0.2865) (1.3938,0.2756) (1.4041,0.2651) (1.4148,0.2550) (1.4259,0.2454) (1.4375,0.2364) (1.4494,0.2280) (1.4618,0.2203) (1.4746,0.2133) (1.4877,0.2070) (1.5012,0.2016) (1.5151,0.1971) (1.5292,0.1936) (1.5436,0.1910) (1.5581,0.1895)
(pi/2,0.1891) 
};
\draw[smooth,dotted]plot coordinates {
(1.5599,0.1448) (1.5706,0.1175) (1.5839,0.0913) (1.6014,0.0679) (1.6250,0.0508) (1.6537,0.0476) (1.6798,0.0608) (1.6990,0.0830) (1.7132,0.1089) (1.7240,0.1365) (1.7325,0.1649) (1.7393,0.1938) (1.7448,0.2230) (1.7492,0.2525) (1.7528,0.2821) (1.7556,0.3119) (1.7577,0.3416) (1.7592,0.3714) (1.7602,0.4013) (1.7605,0.4311) (1.7604,0.4610) (1.7598,0.4908) (1.7588,0.5207) (1.7573,0.5505) (1.7554,0.5802) (1.7530,0.6100) (1.7502,0.6396) (1.7471,0.6693) (1.7435,0.6988) (1.7395,0.7283) (1.7351,0.7577) (1.7303,0.7870) (1.7251,0.8162) (1.7195,0.8453) (1.7135,0.8743) (1.7070,0.9032) (1.7002,0.9320) (1.6929,0.9606) (1.6851,0.9891) (1.6769,1.0174) (1.6682,1.0456) (1.6589,1.0735) (1.6492,1.1013) (1.6388,1.1287) (1.6278,1.1559) (1.6161,1.1828) (1.6037,1.2094) (1.5904,1.2356) (1.5762,1.2613) (1.5609,1.2864) (1.5444,1.3106) (1.5265,1.3339) (1.5070,1.3558) (1.4857,1.3761) (1.4624,1.3942) (1.4370,1.4090) (1.4095,1.4197) (1.3806,1.4253) (1.3510,1.4255) (1.3218,1.4200) (1.2945,1.4088) (1.2698,1.3929) (1.2475,1.3731) (1.2277,1.3502) (1.2101,1.3246) (1.1946,1.2971) (1.1813,1.2682) (1.1700,1.2386) (1.1605,1.2084) (1.1526,1.1776) (1.1463,1.1462) (1.1414,1.1145) (1.1378,1.0824) (1.1356,1.0502) (1.1346,1.0179) (1.1349,0.9855) (1.1363,0.9532) (1.1388,0.9210) 
};
\draw[smooth,thick]plot coordinates {
(1.4923,0.7854) (1.4923,0.7708) (1.4924,0.7561) (1.4925,0.7415) (1.4927,0.7269) (1.4930,0.7122) (1.4933,0.6976) (1.4936,0.6830) (1.4940,0.6684) (1.4944,0.6537) (1.4949,0.6391) (1.4955,0.6245) (1.4961,0.6099) (1.4968,0.5953) (1.4975,0.5807) (1.4982,0.5660) (1.4991,0.5514) (1.5000,0.5368) (1.5009,0.5223) (1.5019,0.5077) (1.5030,0.4931) (1.5041,0.4785) (1.5053,0.4639) (1.5066,0.4494) (1.5079,0.4348) (1.5094,0.4203) (1.5109,0.4057) (1.5124,0.3912) (1.5141,0.3767) (1.5159,0.3623) (1.5178,0.3478) (1.5198,0.3333) (1.5218,0.3189) (1.5240,0.3045) (1.5262,0.2900) (1.5286,0.2756) (1.5311,0.2612) (1.5337,0.2468) (1.5366,0.2324) (1.5396,0.2181) (1.5428,0.2038) (1.5463,0.1897) (1.5501,0.1756) (1.5543,0.1616) (1.5588,0.1477) (1.5638,0.1340) (1.5692,0.1205)  
};
\end{tikzpicture}
\caption{Implementing the shooting method:
Trajectory (1) exists $B$ through the ``roof'' $\{\alpha=0\}$. 
Trajectory (2) exits $B$ through the ``side'' $\{r=\pi/2\}$. 
Trajectory (3) is the desired curve. 
}%
\label{fig:shooting}%
\end{figure} 

\begin{lemma}[Exit through the roof 
$\{\alpha=0\}$, 
see Figure \ref{fig:shooting}\;(1)]\label{lem:roof}
There exists a constant $c_n>0$ depending only on $2\leq n\in\N$ such that if the initial data $r_0\in\interval{0,\pi/2}$ is sufficiently small, then the solution constructed in Lemma~\ref{lem:well-posed} satisfies $\alpha(s_*)=0$ and $r(s_*)\leq c_n r_0$. 
\end{lemma}

\begin{proof}
Let $r_0\in\interval{0,\pi/8}$ be arbitrary. 
We may assume $r(s_*)>2r_0$, otherwise the claim follows directly. 
Let $0<s_1<s_*$ such that $r(s_1)=2r_0$. 
Then, 
\begin{align*}
\theta(s)<\frac{\pi}{4}-\frac{1}{6n}
\end{align*}
 for all $s\in[s_1,s_*]$ by Lemma \ref{lem:dipping_down} and by monotonicity of $\theta$. 
This implies $\cot(2\theta)>\tan(1/(3n))$ for all $s\in[s_1,s_*]$ and we obtain 
\begin{align}\label{eqn:da/ds-estimate}
\frac{d\alpha}{ds}
&=\Biggl(\frac{(2n-2)\cot(2\theta)}{\sin(r)} -(2n-1)\cot(r)\tan(\alpha)\Biggr)\frac{dr}{ds}
\geq\frac{(2n-2)}{r}\tan\biggl(\frac{1}{3n}\biggr)\frac{dr}{ds}
\end{align}
for all $s\in[s_1,s_*]$. 
Integrating \eqref{eqn:da/ds-estimate} yields
\begin{align}\label{eqn:alpha(s)}
\frac{\pi}{2}\geq \alpha(s_*)-\alpha(s_1)&\geq(2n-2)\tan\biggl(\frac{1}{3n}\biggr)\log\biggl(\frac{r(s_*)}{2r_0}\biggr)
\end{align}
which implies 
\[r(s_*)\leq2r_0\exp\Biggl(\frac{\pi}{4n-4}\cot\biggl(\frac{1}{3n}\biggr)\Biggr)=\vcentcolon c_n r_0\]
 and thus proves the second claim. 
If the initial value $r_0>0$ is chosen such that $c_n r_0<\pi/2$ then $r(s_*)<\pi/2$ and we obtain $\alpha(s_*)=0$ from Lemma~\ref{lem:well-posed}. 
\end{proof}

\begin{lemma}[Exit through the side
$\{r=\pi/2\}$, 
see Figure \ref{fig:shooting}\;(2)]\label{lem:side}
If the initial data $r_0\in\interval{0,\pi/2}$ is sufficiently close to $\pi/2$, then the solution constructed in Lemma \ref{lem:well-posed} satisfies $r(s_*)=\pi/2$. 
\end{lemma}

\begin{proof}
Our task amounts to finding  initial data such that $\alpha(s_*)<0$. 
We perform the following change of variables, which turns out to be especially useful in analysing the motion for initial data $r_0\in\interval{0,\pi/2}$ close to $\pi/2$. 
\begin{align*}
x&=\tan(r), & 
y&=\cot(2\theta), &
z&=-\cot(\alpha), 
\\
\cos(r)&=\frac{1}{\sqrt{x^2+1}}, &
\sin(2\theta)&=\frac{1}{\sqrt{y^2+1}},
&
\sin(\alpha)&=-\frac{1}{\sqrt{z^2+1}},
&
\\
\sin(r)&=\frac{x}{\sqrt{x^2+1}},
&
\cos(2\theta)&=\frac{y}{\sqrt{y^2+1}},
&
\cos(\alpha)&=\frac{z}{\sqrt{z^2+1}}. 
\end{align*}
For $s\in\Interval{0,s_*}$ the system \eqref{eqn:dr/ds},\eqref{eqn:dt/ds},\eqref{eqn:da/ds} takes the equivalent form
\begin{align}
\label{eqn:dx/ds}
\frac{dx}{ds}
&=\frac{(x^2+1)z}{\sqrt{z^2+1}},
\\[1ex]
\label{eqn:dy/ds}
\smash{\left\{\vphantom{
\begin{aligned}
\frac{dx}{ds}&=\frac{(x^2+1)z}{\sqrt{z^2+1}},\\[1ex]
\frac{dy}{ds}&=\frac{2(y^2+1)\sqrt{x^2+1}}{x\sqrt{z^2+1}},\\[1ex]
\frac{dz}{ds}&=\frac{(2n-2)yz\sqrt{z^2+1}\sqrt{x^2+1}}{x}+\frac{(2n-1)}{x}\sqrt{z^2+1}.
\end{aligned}}\right.}\,
\frac{dy}{ds}
&=\frac{2(y^2+1)\sqrt{x^2+1}}{x\sqrt{z^2+1}},
\\[1ex]
\label{eqn:dz/ds}
\frac{dz}{ds}
&=\frac{(2n-2)yz\sqrt{z^2+1}\sqrt{x^2+1}}{x}+\frac{(2n-1)}{x}\sqrt{z^2+1}.
\end{align}
Note that the domain for the ODE above is $(x,y,z)\in Q=\interval{0,\infty}\times\Interval{0,\infty}\times \Interval{0,\infty}$.
The variables $x,y,z$ are all weakly increasing along the motion, with initial values 
$x(0)=x_0\vcentcolon=\tan(r_0)>0$ and $y(0)=0=z(0)$.

\emph{Step 1: Preliminary bound on $y$ for the first part of the motion.} 
Equations \eqref{eqn:dy/ds} and \eqref{eqn:dz/ds} imply 
\begin{align}\label{eqn:dz/dy}
\frac{dz}{ds}
&=\Biggl(
\frac{(n-1)yz(z^2+1)} {(y^2+1) }
+\frac{(2n-1)(z^2+1)} {2(y^2+1)\sqrt{x^2+1}}\Biggr)
\frac{dy}{ds}.
\end{align}
Let $s_1\in\interval{0,s_*}$ be fixed, to be chosen. 
Let  $x_1=x(s_1)$, $y_1=y(s_1)$, $z_1=z(s_1)$ and then set
\begin{align*}
\varepsilon=\frac{(2n-1)}{2\sqrt{x_1^2+1}}.
\end{align*}
Since $\frac{dz}{ds}(0)>\varepsilon\frac{dy}{ds}(0)$ 
there exists $s_\varepsilon\in\intervaL{0,s_1}$ such that 
$z\geq \varepsilon y$ for all $s\in [0,s_\varepsilon]$. 
In particular, 
\begin{align*}
\frac{dz}{ds}
&\geq \Biggl(
\frac{(n-1)y^2\varepsilon(z^2+1)} {(y^2+1) }
+\frac{(z^2+1)\varepsilon} {(y^2+1)}\Biggr)
\frac{dy}{ds} 
\geq(z^2+1)\varepsilon\frac{dy}{ds}
\end{align*}
for all $s\in[0,s_\varepsilon]$, which implies 
$\arctan(z)\geq\varepsilon y$ or, equivalently, 
$z\geq\tan(\varepsilon y)$
for all $s\in[0,s_\varepsilon]$. 
Thus, in fact, $s_\varepsilon=s_1$ since $\tan(\varepsilon y)>\varepsilon y$ if $y>0$. 
To conclude, we obtain the preliminary bound:
\begin{align}\label{eqn:y1estimate}
y_1\leq \frac{z_1}{\varepsilon}\leq z_1\sqrt{x_1^2+1}.
\end{align}
\begin{remark}
The bound \eqref{eqn:y1estimate} is global in time in the sense that one can take any value $s_1=s$ in the interval $]0,s_{\ast}[$, but we will introduce a stopping condition at a suitably chosen intermediate time $s_1$ and use the estimate above for the \emph{first part of the motion} only.
\end{remark}

\emph{Step 2: Preliminary bound on $y$ for the second part of the motion.} 
Equation \eqref{eqn:dz/dy} also implies (recalling that $n\geq 2$ throughout our discussion)
\begin{align}
\label{eqn:20210425-2}
\frac{d}{ds}\log\biggl(\frac{z^2}{z^2+1}\biggr)
=
\frac{2}{z(z^2+1)}\frac{dz}{ds}
&\geq\frac{2y}{y^2+1}\frac{dy}{ds}
= \frac{d}{ds}\log(y^2+1)
\end{align}
for all $s\in\interval{0,s_*}$. 
Integrating \eqref{eqn:20210425-2} from $s_1$ to $s\in\Interval{s_1,s_*}$ we obtain  
\begin{align}\label{eqn:20210425-3}
\log\Biggl(\frac{z^2(z_1^2+1)}{(z^2+1)z_1^2}\Biggr)
&\geq\log\Biggl(\frac{y^2+1}{y_1^2+1}\Biggr). 
\end{align}
Estimates \eqref{eqn:y1estimate} and \eqref{eqn:20210425-3} imply 
for any $s\in\Interval{s_1,s_*}$
\begin{align}\notag
y^2+1
&\leq(y_1^2+1)\frac{z^2(z_1^2+1)}{(z^2+1)z_1^2}
\leq\Bigl(x_1^2+1+z_1^{-2}\Bigr)\frac{z^2(z_1^2+1)}{(z^2+1)}.
\shortintertext{In particular, }
\label{eqn:y^2+1}
\frac{y\sqrt{z^2+1}}{z}
&\leq\sqrt{x_1^2+1+z_1^{-2}} \sqrt{z_1^2+1}.
\end{align}
\emph{Step 3: Global bound on $z$.}  
Equations \eqref{eqn:dx/ds} and \eqref{eqn:dz/ds} in combination with \eqref{eqn:y^2+1} imply 
\begin{align}\notag
\frac{dz}{ds}
&=\Biggl(\frac{(2n-2)y(z^2+1)}{x\sqrt{x^2+1}}+\frac{(2n-1)(z^2+1)}{x(x^2+1)z}\Biggr)
\frac{dx}{ds}
\\
&=\Biggl(\frac{(2n-2)y\sqrt{z^2+1}}{\sqrt{1+x^{-2}}~z}+\frac{(2n-1)\sqrt{z^2+1}}{x(1+x^{-2})z^2}\Biggr)
\frac{z\sqrt{z^2+1}}{x^2}
\frac{dx}{ds}
\label{eqn:20210505}
\\
&\leq C_1\frac{z\sqrt{z^2+1}}{x^2}
\frac{dx}{ds}
\label{eqn:20210425-5}
\end{align}
for all $s\in\Interval{s_1,s_*}$, where we
have set 
\begin{align*}
C_1\vcentcolon=(2n-1)\sqrt{z_1^2+1}\Biggl( \sqrt{x_1^2+1+z_1^{-2}}\sqrt{z_1^2+1}+\frac{1}{x_1 z_1^2}\Biggr)
\end{align*}
and used that the function $0< z\mapsto z^{-2}\sqrt{z^2+1}$ in the second summand of \eqref{eqn:20210505} is decreasing. 
Dividing \eqref{eqn:20210425-5} by $z\sqrt{z^2+1}$ and integrating, we obtain  
\begin{align}\label{eqn:integrate(z)}
\int^{s_*}_{s_1}\frac{1}{z\sqrt{z^2+1}}\frac{dz}{ds}\,ds
&\leq\frac{C_1}{x_1}
\leq(2n-1)\sqrt{z_1^2+1}\Biggl( \frac{\sqrt{x_0^2+1+z_1^{-2}}}{x_0}\sqrt{z_1^2+1}+\frac{1}{z_1^2x_0^2} \Biggr),
\end{align}
%
Towards a contradiction, suppose that $s\mapsto z(s)$ is unbounded for any choice of initial value $x_0>1$.  
Then, in particular, we may choose $s_1\in\interval{0,s_*}$ such that $z_1=z(s_1)=1/x_0$.  
With this choice, the right-hand side of \eqref{eqn:integrate(z)} is bounded from above by $12n$ uniformly in $x_0>1$, while the left-hand side of \eqref{eqn:integrate(z)} is bounded from below by 
\begin{align*}
\int^{1}_{1/x_0}\frac{1}{z\sqrt{z^2+1}}\,dz
\end{align*}
which diverges as $x_0\to\infty$. 
This contradiction yields a global upper bound on $z$ provided that $x_0$ is chosen sufficiently large. 
Rephrasing this outcome back in the original variables, the third component $\alpha$ of trajectories with $r_0$ close enough to $\pi/2$ remains negative (bounded away from zero), which means that such trajectories exit the domain $B$ from the side $r=\pi/2$, as claimed. 
\end{proof}

We can now proceed and wrap things together to prove our main theorem.

\begin{proof}[Proof of Theorem \ref{thm:main}]
Appealing to Lemma \ref{lem:roof} and Lemma \ref{lem:side} respectively, we can find  $r'_0, r''_0 \in\interval{0,\pi/2}$ with $r'_0<r''_0$ such that the trajectories emanating from $(r'_0, \pi/4,-\pi/2)$ and $(r''_0, \pi/4, -\pi/2)$ leave the domain $B$ from the roof ($\alpha=0$) and from the side ($r=\pi/2$) respectively. Now, consider the image through the flow map of the relatively compact interval $[r'_0, r''_0]\subset\interval{0,\pi/2}$. It is understood that each trajectory is followed till $\partial B$, where there is a transverse crossing, and not beyond. In particular, by considering the evaluation of each trajectory at the exit time, the flow defines a $C^0$ curve, whose image will be henceforth denoted by $\Gamma$. Note that (by continuous dependence on initial data) $\Gamma$ is a connected set and by definition (keeping in mind Lemma \ref{lem:well-posed}) its image will be contained in the dihedron 
\[
\Delta\vcentcolon=(B \cap \left\{\alpha=0 \right\}) \cup (B \cap \left\{r=\pi/2\right\}).
\]
Furthermore, $\Gamma$ contains points both on $B \cap \left\{\alpha=0 \right\}$ and on $B \cap \left\{r=\pi/2\right\}$. Hence, by connectedness $\Gamma$ shall also contain point on the intersection
\[
(B \cap \left\{\alpha=0 \right\}) \cap (B \cap \left\{r=\pi/2\right\}).
\]
In other words, if we get back to the picture in the planar domain with coordinates $(r,\theta)$ only,  we have singled out a monotone arc that starts at a point $(r_0, \pi/4)$ and reaches the segment $r=\pi/2$ orthogonally.
Hence, by simply reflecting such a trajectory along the axes $r=\pi/2$ and $\theta=\pi/4$ of our planar domain $(r,\theta)\in\interval{0,\pi}\times\interval{0,\pi/2}$ we get a smooth closed orbit solving our ODE system. Thereby, by virtue of the equivariant reduction we presented at the beginning of the present section, the conclusion follows. 
\end{proof}

\begin{remark}
Concerning a different geometric problem, we recall here that Angenent \cite{Angenent1992} used the shooting method to construct an embedded, rotationally symmetric, self-shrinking $S^{n-1}\times S^1$ in $\R^{n+1}$.  Concerning, instead, the contributions we will present in the next section, note that, later, more examples of immersed, rotationally symmetric self-shrinkers were found in \cite{Drugan2018}.  
\end{remark}

Numerical simulations seem to clearly indicate that the closed periodic orbits constructed above are unique, and thus the corresponding minimal embedding of $S^{n-1}\times S^{n-1}\times S^1$ in $\Sp^{2n}$ should be unique in the class of $G$-equivariant maps for any $n\geq 2$. In fact, in the case when $n=2$ we wish to formulate here a stronger conjecture:

\begin{conj}\label{conj:unique} There exists, up to ambient isometry, a unique minimal embedding of the three-dimensional torus in the round four-dimensional sphere.
\end{conj}

This can be regarded as a higher-dimensional counterpart of Lawson's conjecture, confirmed by Brendle in \cite{Brendle2013}. 
It is not clear at the moment whether we have compelling evidence to envision a similar landscape in $\Sp^{m+1}$ when $m>3$ (and especially for $m>4$, cf. Appendix~\ref{sec:HsiangClaim}), i.\,e. whether minimal embeddings of $T^{m}$ in $\Sp^{m+1}$ should actually exist and, if so, whether such embeddings should be unique. 
These questions stand as fascinating, challenging open problems in the field.

\section{Existence of infinitely many minimally immersed hypertori and minimally embedded hyperspheres}\label{sec:immersed}

In this section it is convenient to introduce the shifted variable $\vartheta(s)=\theta(s)-\pi/4$. 
Then, the system \eqref{eqn:dr/ds},\eqref{eqn:dt/ds},\eqref{eqn:da/ds} takes the form
\begin{align}
\label{eqn:dr/ds2}
\frac{dr}{ds}&=\cos(\alpha), \\[1ex]
\label{eqn:dt/ds2}
\smash{\left\{\vphantom{
\begin{aligned}
\frac{dr}{ds}&=\cos(\alpha), \\[1ex]
\frac{d\vartheta}{ds}&=\frac{\sin(\alpha)}{\sin(r)}, \\[1ex]
\frac{d\alpha}{ds}&=-\frac{(2n-2)\tan(2\vartheta)}{\sin(r)}\cos(\alpha)-(2n-1)\cot(r)\sin(\alpha).
\end{aligned}}\right.}\,
\frac{d\vartheta}{ds}&=\frac{\sin(\alpha)}{\sin(r)}, \\[1ex]
\label{eqn:da/ds2}
\frac{d\alpha}{ds}&=-\frac{(2n-2)\tan(2\vartheta)}{\sin(r)}\cos(\alpha)-(2n-1)\cot(r)\sin(\alpha).
\end{align}
We will study solutions of the system \eqref{eqn:dr/ds2},\eqref{eqn:dt/ds2},\eqref{eqn:da/ds2} with initial data $(r,\vartheta,\alpha)(0)=(r_0,0,\alpha_0)$, where $r_0\in\Interval{\pi/2,\pi}$ and $\alpha_0\in\interval{0,\pi/2}$.  
We will then show that if $n\in\{2,3\}$ and if $\alpha_0>0$ is sufficiently small, the trajectory will intersect $\{\vartheta=0\}$ again for some $s>0$.  
Moreover, $\alpha$ will be arbitrarily close to zero at the point of intersection if $\alpha_0>0$ is sufficiently small. 
Since the problem is symmetric with respect to the reflection 
$(r,\vartheta,\alpha)\mapsto(r,-\vartheta,-\alpha)$ we may iterate our estimates, and will prove the existence of trajectories whose second component $\vartheta$ has arbitrarily many zeros.

\begin{lemma}\label{lem:well-posed_centre}
Given any $r_0\in\Interval{\pi/2,\pi}$ and any $\alpha_0\in\interval{0,\pi/2}$ there exists $0<s_*\leq\infty$
such that the system \eqref{eqn:dr/ds2},\eqref{eqn:dt/ds2},\eqref{eqn:da/ds2} has a unique solution 
\[
(r,\vartheta,\alpha)\colon\Interval{0,s_*}\to D\vcentcolon=\Interval{\tfrac{\pi}{2},\pi}\times\interval{-\tfrac{\pi}{4},\tfrac{\pi}{4}}\times\interval{-\tfrac{\pi}{2},\tfrac{\pi}{2}}
\] 
with initial data $(r,\vartheta,\alpha)(0)=(r_0,0,\alpha_0)$, on which the solution depends continuously, satisfying
$dr/ds>0$ and one of the following alternatives: 
\begin{align}\label{eqn:20210831-r=pi}
\lim_{s\to s_*}r(s)&=\pi,
&
\limsup_{s\to s_*}\abs{\vartheta(s)}&=\frac{\pi}{4},
&
\limsup_{s\to s_*}\abs{\alpha(s)}&=\frac{\pi}{2}. 
\end{align}
Furthermore, if $s_{\ast}=\infty$, then the third condition holds true.
The conditions in \eqref{eqn:20210831-r=pi} are not mutually exclusive (cf. Lemma \ref{lem:iteration}), but if either $\alpha(s)$ or $\vartheta(s)$ converges to a nonzero limit as $s\to s_*$ then 
\begin{align}\label{eqn:20210831-r<pi}
\lim_{s\to s_*}r(s)<\pi.
\end{align}
\end{lemma}

\begin{proof}
By our choice of initial data, there exists a solution to system \eqref{eqn:dr/ds2},\eqref{eqn:dt/ds2},\eqref{eqn:da/ds2} which stays in $D$ at least for a short time. 
Since $-\pi/2<\alpha<\pi/2$, equation \eqref{eqn:dr/ds2} implies $dr/ds=\cos(\alpha)>0$ as claimed. 
The vector field on the right-hand side of the ODE system has no stationary points in $D$; in fact  $(dr/ds)^2+(d\vartheta/ds)^2\geq1$ by \eqref{eqn:constraint}. 
As a result, keeping in mind the very definition of the (bounded) domain $D$, the motion will either satisfy one of the conditions \eqref{eqn:20210831-r=pi} for finite $s_{\ast}$ or else will exist for all positive times and, in that case, must necessarily satisfy $\limsup_{s\to \infty}\abs{\alpha(s)}=\frac{\pi}{2}$ by \eqref{eqn:dr/ds2}.
Thus, there exists $0<s_*\leq\infty$ such that at least one of the alternatives in \eqref{eqn:20210831-r=pi} holds. 
Let us now justify the final claim in the statement instead.

\emph{Case 1:} $\lim_{s\to s_*}\alpha(s)=\alpha_*\neq0$ exists. 
Since the problem is invariant under the reflection $(r,\vartheta,\alpha)\mapsto(r,-\vartheta,-\alpha)$ we may assume $\alpha_*>0$ without loss of generality. 
By continuity of $\alpha$ there exist $\varepsilon\in\interval{0,\pi/2}$ and $0<s_\varepsilon<s_*$ such that $\alpha(s)\geq\varepsilon$ for all $s\in\Interval{s_\varepsilon,s_*}$. 
Equations \eqref{eqn:dr/ds2} and \eqref{eqn:dt/ds2} then imply 
\begin{align}\label{eqn:20210730-1}
\frac{d\vartheta}{ds}
&=\frac{\tan(\alpha)}{\sin(r)}\frac{dr}{ds}
\geq\frac{\varepsilon}{\sin(r)}\frac{dr}{ds}.
\end{align}
Integrating \eqref{eqn:20210730-1} from $s=s_\varepsilon$ to any $s\in\interval{s_\varepsilon,s_*}$ yields 
\begin{align}\label{eqn:20210730-2}
\frac{\pi}{2}\geq\vartheta(s)-\vartheta(s_\varepsilon)\geq
\varepsilon\log\Biggl(\tan\biggl(\frac{r(s)}{2}\biggr)\Biggr)
-\varepsilon\log\Biggl(\tan\biggl(\frac{r(s_\varepsilon)}{2}\biggr)\Biggr).
\end{align}
Passing to the limit $s\to s_*$ in \eqref{eqn:20210730-2}, keeping in mind the monotonicity of $s\mapsto r(s)$, proves \eqref{eqn:20210831-r<pi}. 

\emph{Case 2:} $\lim_{s\to s_*}\vartheta(s)=\vartheta_*\neq0$ exists. 
Without loss of generality we may assume $\vartheta_*>0$ by symmetry of the problem.  
Therefore, there exist $\varepsilon>0$ and $0<s_\varepsilon<s_*$ such that $\tan(2\vartheta)\geq\varepsilon$ for all $s\in\Interval{s_\varepsilon,s_*}$. 
Equations \eqref{eqn:dr/ds2}--\eqref{eqn:da/ds2} imply 
\begin{align}\label{eqn:20210730-3}
\frac{d\alpha}{ds}&=-(2n-2)\frac{\tan(2\vartheta)}{\sin(r)}\frac{dr}{ds}-(2n-1)\cos(r)\frac{d\vartheta}{ds}. 
\end{align}
Integrating \eqref{eqn:20210730-3} from $s=s_\varepsilon$ to any $\sigma\in\interval{s_\varepsilon,s_*}$ by parts, we obtain 
\begin{align}\label{eqn:20230607}
\int^{\sigma}_{s_\varepsilon}\frac{(2n-2)\varepsilon}{\sin(r)}\frac{dr}{ds}\,ds
&\leq-\int^{\sigma}_{s_\varepsilon}\frac{d\alpha}{ds}\,ds
-(2n-1)\int^{\sigma}_{s_\varepsilon}\cos(r)\frac{d\vartheta}{ds}\,ds  
\\\notag
&=\Bigl(-\alpha(s)-(2n-1)\vartheta(s)\cos\bigl(r(s)\bigr)\Bigr)\Big\vert^{\sigma}_{s_\varepsilon}
-(2n-1)\int^{\sigma}_{s_\varepsilon}\vartheta\sin(r)\frac{dr}{ds}\,ds  
\\\notag
&\leq -\alpha(\sigma)+\alpha(s_\varepsilon)-(2n-1)\Bigl(\vartheta(\sigma)\cos\bigl(r(\sigma)\bigr)-\vartheta(s_\varepsilon)\cos\bigl(r(s_\varepsilon)\bigr)\Bigr)
\leq\pi+n\pi.
\end{align}
Since $r(s)$ is strictly increasing in $s$, its limit as $s\to s_*$ exists and \eqref{eqn:20210831-r<pi} follows by letting $\sigma\to s_*$ in \eqref{eqn:20230607} and arguing exactly as we did in closing the treatment of Case 1.  
\end{proof}

\subsection{Preparation: Estimates for small initial angles}

\begin{lemma}[{see Figure \ref{fig:centre2}}]\label{lem:20210612-1}
Given initial data $r_0\in\Interval{\pi/2,\pi}$ and $\alpha_0\in\interval{0,\pi/2}$ let $(r,\vartheta,\alpha)\colon\allowbreak\Interval{0,s_*}\to D$ be the solution constructed in Lemma~\ref{lem:well-posed_centre}.  
Then there exists $s_1\in\interval{0,s_*}$ such that 
\begin{align*}
\vartheta(s_1)&=\frac{1}{4n}\Bigl(\frac{\pi}{2}-\alpha_0\Bigr)\alpha_0.  
\end{align*}
Moreover, $0<\vartheta(s)\leq\vartheta(s_1)$ and $\tan\bigl(\alpha(s)\bigr)\geq \frac{1}{2}\tan(\alpha_0)$ for all $s\in\intervaL{0,s_1}$.
\end{lemma}

\begin{proof}
Let $s_1$ be the least upper bound (sup) of the set of all $b\in\Interval{0,s_*}$
with the property that for all $s\in\IntervaL{0,b}$
\begin{align}\label{eqn:20210830-s1}
0\leq\vartheta(s)\leq\frac{1}{4n}\Bigl(\frac{\pi}{2}-\alpha_0\Bigr)\alpha_0.
\end{align}
Then, $s_1>0$ because $\vartheta(0)=0$ and $\alpha_0\in \interval{0,\pi/2}$. 
Similarly and independently we consider 
$s_\alpha\vcentcolon=\sup\{b\in\Interval{0,s_*}\st \forall s\in\Interval{0,b} ~~\alpha(s)>0\}$. 
By equation~\eqref{eqn:dt/ds2} we have $d\vartheta/ds\geq 0$ in the interval $\Interval{0,s_{\alpha}}$ and $\cos(r)\leq 0$ since $r\in[\pi/2,\pi]$. 
Thus, \eqref{eqn:dt/ds2} and \eqref{eqn:da/ds2} imply
\begin{align}\label{eqn:20210612-da/dt}
\frac{d\alpha}{ds}
&=\Bigl(-(2n-2)\tan(2\vartheta)\cot(\alpha)-(2n-1)\cos(r) \Bigr)\frac{d\vartheta}{ds}
\geq -(2n-2)\tan(2\vartheta)\cot(\alpha) \frac{d\vartheta}{ds}
\end{align}
for all $s\in\Interval{0,s_{\alpha}}$. 
Multiplying estimate \eqref{eqn:20210612-da/dt} by $\tan(\alpha)>0$, we obtain 
\begin{align}
\label{eqn:20210612-0}
-\frac{d}{ds}\log\bigl(\cos(\alpha)\bigr)
=\tan(\alpha)\frac{d\alpha}{ds}
&\geq-(2n-2)\tan(2\vartheta) \frac{d\vartheta}{ds}
=(n-1)\frac{d}{ds}\log\bigl(\cos(2\vartheta)\bigr).
\end{align}
Integrating \eqref{eqn:20210612-0} from $s=0$ to any $s\in\Interval{0,s_{\alpha}}$ implies  
\begin{align}\label{eqn:20210612-1}
\frac{\cos(\alpha_0)}{\cos(\alpha)}
&\geq\bigl(\cos(2\vartheta)\bigr)^{n-1}.
\intertext{Estimate \eqref{eqn:20210612-1} is equivalent to }
\notag
\cos(\alpha)&\leq\bigl(\cos(2\vartheta)\bigr)^{1-n}\cos(\alpha_0).
\end{align}
Since $\tan\bigl(\arccos(x)\bigr)=\sqrt{x^{-2}-1}$ for $0<x\leq1$, we obtain 
for any $s\in\Interval{0,s_{\alpha}}$ 
\begin{align} \label{eqn:20210612-2}
\tan(\alpha)&\geq
\sqrt{\frac{\cos^{2n-2}(2\vartheta)}{\cos^2(\alpha_0)}-1} 
=\tan(\alpha_0)\sqrt{\frac{\cos^{2n-2}(2\vartheta)-\cos^2(\alpha_0)}{\sin^2(\alpha_0)}}.
\end{align}
On the other hand, in $\Interval{0,s_1}$ the assumption $0\leq2\vartheta\leq(\pi/2-\alpha_0)\alpha_0/(2n)\leq\alpha_0/n$ implies 
$\cos(2\vartheta)\geq\cos(\alpha_0/n)$;  
the map $f\colon\intervaL{0,\pi/2}\to\R$ given by 
\begin{align*}
f(a)= \frac{\cos^{2n-2}(a/n)-\cos^2(a)}{\sin^2(a)}
\end{align*}
is decreasing with minimum value $f(\pi/2)>1/4$ (since $n\geq 2$), hence we obtain 
\begin{align}\label{eqn:20210612-4}
\tan(\alpha)\geq \tan(\alpha_0)\sqrt{f(\alpha_0)}
\geq \frac{1}{2}\tan(\alpha_0)
\end{align}
for all $s\in\Interval{0,s_{\alpha}}\cap \Interval{0,s_1}$. 
In particular, \eqref{eqn:20210612-4} implies that $\alpha$ is strictly positive in the intersection of both intervals. 
Hence, $s_1\leq s_\alpha$ because otherwise, we would have $s_\alpha<s_1\leq s_*$ and thus $\alpha(s_\alpha)=0$ in contradiction to  \eqref{eqn:20210612-4}. 
In particular, $\vartheta$ is strictly increasing in $\Interval{0,s_1}$. 
It remains to show $s_1<s_*$ (in particular: finiteness of $s_1$ when $s_{\ast}=\infty$) because then the upper bound in \eqref{eqn:20210830-s1} must be attained and the claim follows. Since equations \eqref{eqn:dt/ds2} and \eqref{eqn:da/ds2} imply $d\alpha/ds\leq (2n-1)d\vartheta/ds$ in $\Interval{0,s_1}$ we have 
\begin{align}\label{eqn:20210705-alpha}
\alpha(s)
&\leq\alpha_0+(2n-1)\vartheta(s)
\leq\alpha_0+\frac{1}{2}\Bigl(\frac{\pi}{2}-\alpha_0\Bigr)\alpha_0<\frac{\pi}{2}
\end{align}
for all $s\in\Interval{0,s_1}$. 
Indeed, the function $h\colon[0,\pi/2]\to\R$ given by $h(a)=a+(\pi/2-a)a/2$ satisfies $h'(a)=1+\pi/4-a>0$. 
Therefore $h(\alpha_0)<h(\pi/2)=\pi/2$ for any choice of $\alpha_0\in\interval{0,\pi/2}$. If $s_1=s_*$ then \eqref{eqn:20210705-alpha} and \eqref{eqn:20210830-s1} together with Lemma \ref{lem:well-posed_centre} \eqref{eqn:20210831-r=pi} imply $r(s)\to\pi$ as $s\to s_*$. 
However, since $\vartheta$ is bounded and strictly increasing in $\Interval{0,s_1}$, the limit 
$\lim_{s\to s_*}\vartheta(s)>0$ exists. 
Hence, Lemma \ref{lem:well-posed_centre} \eqref{eqn:20210831-r<pi} applies, which yields a contradiction.  
As a result, $s_1<s_*$ and the claim follows.  
\end{proof}

\begin{lemma}[{see Figure \ref{fig:centre2}}]\label{lem:20210705}
Let $n\in\{2,3\}$ and let $\delta>0$ be arbitrary. 
Then there exists $0<a_\delta<\pi/2$ such that for any $\alpha_0\in\intervaL{0,a_\delta}$ and any $r_0\in\Interval{\pi/2,\pi}$ the solution $(r,\vartheta,\alpha)\colon\Interval{0,s_*}\to D$ 
with $(r,\vartheta,\alpha)(0)=(r_0,0,\alpha_0)$ as constructed in Lemma~\ref{lem:well-posed_centre} allows for $s_2\in\interval{0,s_*}$ such that 
\[\alpha(s_2)=0\]
and, in addition, $\alpha(s)>0$ and $\tan\bigl(2\vartheta(s)\bigr)<\delta$ for all $s\in\Interval{0,s_2}$. 
\end{lemma}%

\begin{figure}\centering
\begin{tikzpicture}[line cap=round,line join=round,baseline={(0,0)},scale=\textwidth/3.5cm] 
\draw[very thin](0,0)grid [xstep=pi/2,ystep=pi/4](pi,pi/2);
\draw[-latex](0,pi/4)--++(pi,0)node[below left]{$r$};
\draw[-latex](0,0)--(0,pi/2)node[below left]{$\vartheta$};
\draw plot[plus](0,0)node[below]{$0$};
\draw plot[plus](0,0)node[left]{$-\frac{\pi}{4}$};
\draw plot[plus](pi/2,0)node[below]{$\frac{\pi}{2}$};
\draw plot[plus](pi,0)node[below]{$\pi$};
\draw plot[plus](0,pi/4)node[left]{$0$};
\draw[smooth,semithick]plot coordinates {  
(2.0944,0.7854) (2.1100,0.7870) (2.1256,0.7887) (2.1412,0.7904) (2.1568,0.7922) (2.1723,0.7940) (2.1879,0.7960) (2.2035,0.7979) (2.2191,0.8000) (2.2346,0.8022) (2.2502,0.8044) (2.2657,0.8067) (2.2813,0.8091) (2.2968,0.8116) (2.3124,0.8141) (2.3279,0.8168) (2.3434,0.8196) (2.3589,0.8225) (2.3745,0.8255) (2.3900,0.8286) (2.4055,0.8318) (2.4210,0.8352) (2.4365,0.8387) (2.4519,0.8423) (2.4674,0.8461) (2.4829,0.8500) (2.4983,0.8541) (2.5138,0.8584) (2.5292,0.8628) (2.5447,0.8673) (2.5601,0.8721) (2.5755,0.8770) (2.5910,0.8821) (2.6064,0.8874) (2.6218,0.8928) (2.6372,0.8985) (2.6526,0.9043) (2.6680,0.9102) (2.6834,0.9163) (2.6989,0.9226) (2.7143,0.9289) (2.7297,0.9353) (2.7452,0.9417) (2.7607,0.9479) (2.7762,0.9539) (2.7917,0.9596) (2.8072,0.9646) (2.8228,0.9685) (2.8383,0.9713) (2.8539,0.9725) (2.8696,0.9715) (2.8854,0.9673) (2.9010,0.9587) (2.9163,0.9440) (2.9312,0.9214) (2.9454,0.8886) (2.9584,0.8431) (2.9699,0.7825) (2.9789,0.7058) (2.9847,0.6145) (2.9865,0.5143) 
};
\draw plot[bullet](2.8539,0.9725)coordinate(s2)node[anchor=-90]{$s_2$};
\draw plot[bullet](2.9865,0.5143)node[anchor=180]{$s_*$};
\draw plot[vdash](2*pi/3,pi/4)node[below]{$r_0$}; 
\draw plot[hdash](pi/2,0.9725)coordinate(t2)node[left]{$\vartheta(s_2)$};
\draw[densely dashed](t2)--(s2);
\draw plot[bullet](2.23,{pi/4+(pi/2-pi/36)*(pi/36)/8})coordinate(s1)node[anchor=-90]{$s_1$};
\end{tikzpicture}
\caption{Visualisation of Lemmata \ref{lem:20210612-1} and \ref{lem:20210705} for $n=2$, $r_0=2\pi/3$ and $\alpha_0=\pi/36$.}%
\label{fig:centre2}%
\end{figure}
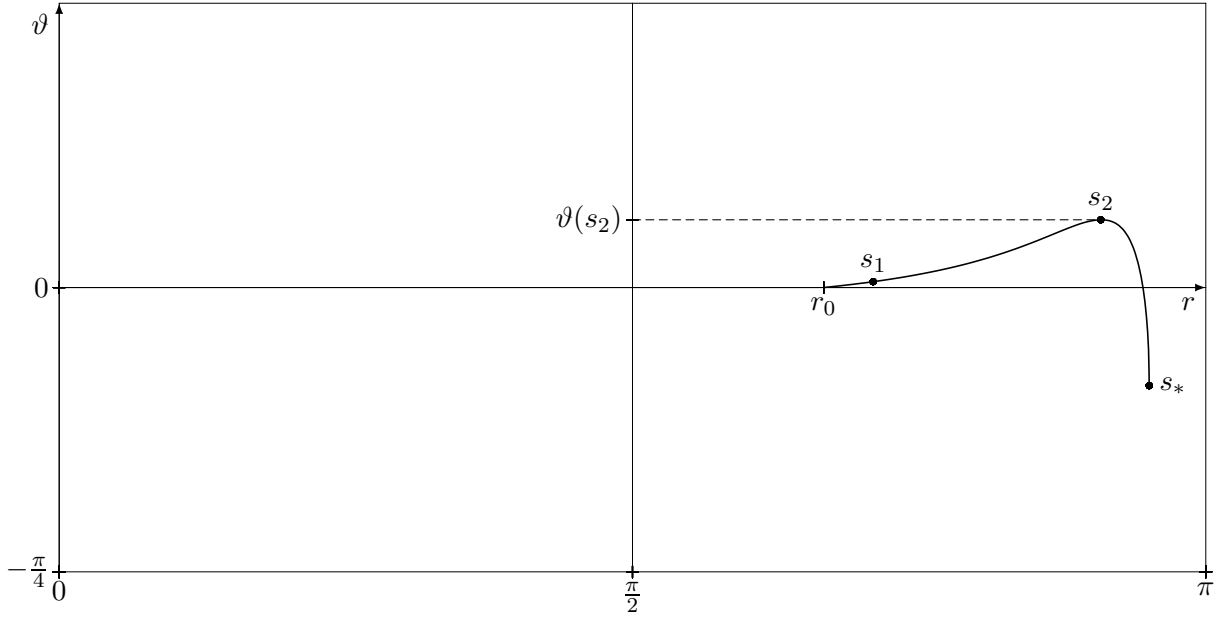%

\begin{proof} 
Let $s_1\in\interval{0,s_*}$ be as in Lemma \ref{lem:20210612-1}. 
and let $s_2$ be the least upper bound of the set of all $b\in\Interval{0,s_*}$ with the property that for all $s\in\Interval{0,b}$
\begin{align}\label{eqn:20210831-s2}
0<\alpha(s)<\tfrac{1}{3}.
\end{align}
We claim that if $\alpha_0>0$ is sufficiently small, then $s_2>s_1$. 
Since Lemma \ref{lem:20210612-1} implies that $\alpha$ is positive in $[0,s_1]$ it suffices to prove $\alpha(s)<1/3$ for all $s\in[0,s_1]$.  
As in estimate~\eqref{eqn:20210705-alpha}  we obtain for all $s\in[0,s_1]$
\begin{align} 
\alpha(s)
\leq\alpha_0+(2n-1)\vartheta(s)
\leq\alpha_0+\frac{2n-1}{4n}\Bigl(\frac{\pi}{2}-\alpha_0\Bigr)\alpha_0 \leq2\alpha_0. 
\end{align}
In particular, we have $\alpha(s)<1/3$ for all $s\in[0,s_1]$ provided that $\alpha_0<1/6$. 
Thus, $s_2>s_1$ for any choice of  $0<\alpha_0<1/6$.  
Since $\vartheta(s)>0$ for all $s\in\Interval{s_1,s_2}$ there exists $\beta\colon\Interval{s_1,s_2}\to\R$ such that 
\begin{align*}
\alpha(s)=\vartheta(s)\,\beta(s)
\end{align*}
for all $s\in\Interval{s_1,s_2}$.  
Equations \eqref{eqn:dt/ds2} and \eqref{eqn:da/ds2} then imply 
\begin{align}\notag
\vartheta\frac{d\beta}{ds}
&=\Bigl(-(2n-2)\tan(2\vartheta)\cot(\alpha)-\beta-(2n-1)\cos(r)\Bigr)\frac{d\vartheta}{ds}
\\\notag
&\leq\Bigl(-4(n-1)\vartheta\cot(\alpha)-\frac{\alpha}{\vartheta}+(2n-1)\Bigr)\frac{d\vartheta}{ds}
\\\label{eqn:20210629-tdv/ds}
&\leq\Bigl(-4\sqrt{(n-1)\alpha\cot(\alpha)}+(2n-1)\Bigr)\frac{d\vartheta}{ds}
\end{align}
in $\Interval{s_1,s_2}$, where the last inequality above relies on the standard AM-GM inequality.
The assumptions $n\in\{2,3\}$ and $0<\alpha< 1/3$ in $\Interval{s_1,s_2}$ imply 
\begin{align}\label{eqn:20210629-negative}
-4\sqrt{(n-1)\alpha\cot(\alpha)}+(2n-1)<-\frac{1}{2}.
\end{align}
The negativity of the left-hand side of \eqref{eqn:20210629-negative} is crucial here and in fact fails for $n\geq4$ even if $\alpha$ is arbitrarily small.  
Dividing estimate \eqref{eqn:20210629-tdv/ds} by $\vartheta>0$ and integrating from $s=s_1$ to any $s\in\Interval{s_1,s_2}$, we thus obtain 
\begin{align}\label{eqn:20210629-log}
\beta(s) 
-\beta(s_1)&\leq
\frac{1}{2}\log\bigl(\vartheta(s_1)\bigr)
-\frac{1}{2}\log\bigl(\vartheta(s)\bigr).
\end{align}
Moreover, since $\vartheta(s_1)=(\pi/2-\alpha_0)\alpha_0/(4n)\geq\alpha_0/9$ for $\alpha_0<1/6$ and $n\in\{2,3\}$ we have 
\begin{align*}
\beta(s_1)=\frac{\alpha(s_1)}{\vartheta(s_1)}
\leq\frac{\alpha_0}{\vartheta(s_1)}+(2n-1)
\leq9+5=14.
\end{align*}
We also have $\vartheta(s_1)\leq\pi\alpha_0/(8n)<\alpha_0/5$. 
Estimate \eqref{eqn:20210629-log} then implies for any $s\in\Interval{s_1,s_2}$  
\begin{align}\label{eqn:20210629-alpha}
\alpha&\leq\biggl(14
+\frac{1}{2}\log\Bigl(\frac{\alpha_0}{5}\Bigr)\biggr)\vartheta
-\frac{\vartheta}{2}\log(\vartheta).
\end{align}%
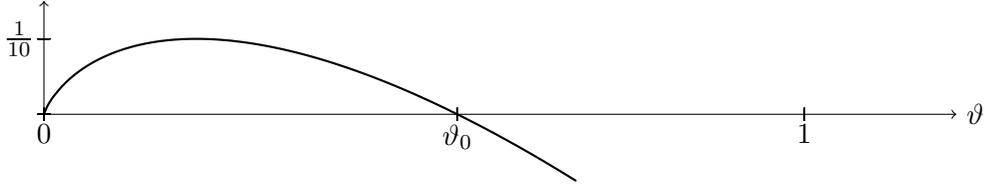
\begin{figure}\centering
\begin{tikzpicture}[line cap=round,line join=round,baseline={(0,0)},scale=10]
\draw[->](-0,0)--(1.2,0)node[right]{$\vartheta$};
\draw[->](0,-0)--(0,0.15);
\draw plot[plus](0,0)node[below]{$0$};
\draw plot[vdash](1,0)node[below]{$1$};
\draw plot[vdash]({exp(1)/5},0)node[below]{$\vartheta_0$};
\draw plot[hdash](0,1/10)node[left]{$\frac{1}{10}$};
\draw[thick,domain=0.00001:0.7,smooth,samples=180,variable=\t]plot({\t},{
(-0.3047189562170498)*\t-\t*ln(sqrt(\t))
});
\end{tikzpicture}
\caption{Plot of the function $\vartheta\mapsto\Bigl(14+\frac{1}{2}\log\bigl(\alpha_0/5\bigr)\Bigr)\vartheta
-\frac{1}{2}\vartheta\log(\vartheta)$ for $\alpha_0=e^{-27}$.}%
\label{fig:plot-t-t*log(t)}%
\end{figure}%
The function on the right-hand side of \eqref{eqn:20210629-alpha} can be made arbitrarily small by choosing $\alpha_0>0$ sufficiently small. 
In particular, we can achieve (say) $\alpha<1/10$ for all $s\in\Interval{s_1,s_2}$. 
Moreover, as illustrated in Figure \ref{fig:plot-t-t*log(t)}, the right-hand side of \eqref{eqn:20210629-alpha} as a function of $\vartheta>0$ vanishes for some $\vartheta_0>0$ which again is arbitrarily small if $\alpha_0>0$ is chosen sufficiently small. 
Since $\vartheta$ is strictly increasing in $\Interval{0,s_2}$, we obtain $\vartheta(s)<\vartheta_0$ for all $s\in\Interval{0,s_2}$. 

It remains to prove $s_2<s_*$ (which, again, also means finiteness of $s_2$ when $s_{\ast}=\infty$) because then the lower bound in \eqref{eqn:20210831-s2} must be attained and $\alpha(s_2)=0$ would follow. 
Towards a contradiction we assume $s_2=s_*$.  
Then, the bounds on $\vartheta$ and $\alpha$ imply $r(s)\to\pi$ as $s\to s_*$ by Lemma \ref{lem:well-posed_centre} \eqref{eqn:20210831-r=pi}. 
However, since $\vartheta$ is strictly increasing it has a nonzero limit as $s\to s_*$. 
Consequently,  Lemma \ref{lem:well-posed_centre} \eqref{eqn:20210831-r<pi} applies, which yields a contradiction. 
\end{proof}

\begin{lemma}\label{lem:small_alpha}
Let $n\in\{2,3\}$ and let $0<\varepsilon<\pi/2$ be arbitrary. 
Then there exists $a_\varepsilon\in\interval{0,\pi/2}$ 
such that for any $r_0\in\Interval{\pi/2,\pi}$ and any $\alpha_0\in\intervaL{0,a_\varepsilon}$ 
the solution $(r,\vartheta,\alpha)\colon\Interval{0,s_*}\to D$ 
with initial data $(r,\vartheta,\alpha)(0)=(r_0,0,\alpha_0)$ as constructed in Lemma~\ref{lem:well-posed_centre} allows for $s_3\in\interval{0,s_*}$ such that
\begin{align*}
\vartheta(s_3)&=0 , & \abs{\alpha(s_3)}&<\varepsilon 
\end{align*}
and such that $\vartheta(s)>0$ for all $s\in\interval{0,s_3}$. 
\end{lemma}

\begin{proof}
Given any $\delta>0$ let $0<\alpha_0<a_\delta$ and $s_2\in\interval{0,s_*}$ be as in Lemma \ref{lem:20210705}. 
Furthermore, consider $s_3\vcentcolon=\sup\{b\in\Interval{0,s_*}\st\forall s\in[0,b]~\vartheta(s)\geq0\}$.
By Lemma \ref{lem:20210705} we have $\alpha>0$ and hence $d\vartheta/ds>0$ in the interval $\Interval{0,s_2}$. Therefore, $s_2<s_3$. 
Concerning the subsequent part of the motion, equation \eqref{eqn:da/ds2} implies
\begin{align}\label{eqn:20210831-da/ds<0}
\frac{d\alpha}{ds}\leq-\frac{(2n-2)\tan(2\vartheta)}{\sin(r)}\cos(\alpha)<0 
\end{align} 
as long as $\vartheta>0$ and $\alpha\leq0$. 
In the interval $\interval{s_2,s_3}$, we thus have $\alpha<0$ and $d\vartheta/ds<0$.
Therefore, the bound $\tan(2\vartheta)\leq\delta$ from Lemma \ref{lem:20210705} extends to the interval $\Interval{s_2,s_3}$.  
Equations \eqref{eqn:dt/ds2} and \eqref{eqn:da/ds2} then yield 
\begin{align}\label{eqn:20210615-da/dt}
\frac{d\alpha}{ds}
&=\Bigl(-(2n-2)\tan(2\vartheta) \cot(\alpha)-(2n-1)\cos(r)\Bigr)\frac{d\vartheta}{ds}
\geq\Bigl(-\delta\cot(\alpha)+1\Bigr) (2n-1) \frac{d\vartheta}{ds}
\end{align}
in $\Interval{s_2,s_3}$. 
Dividing inequality \eqref{eqn:20210615-da/dt} by $(-\delta\cot(\alpha)+1)>0$ and integrating from $s=s_2$ to any $s\in\Interval{s_2,s_3}$ we obtain 
\begin{align*}
\frac{\alpha+\delta\log\bigl(\delta\cos(\alpha)-\sin(\alpha)\bigr)}{\delta^2+1}
-\frac{\delta\log(\delta)}{\delta^2+1}
&\geq(2n-1)\bigl(\vartheta-\vartheta(s_2)\bigr)
\geq-(2n-1)\delta 
\end{align*}
which implies
\begin{align}
\label{eqn:20210615-a<delta}
\frac{\alpha(s)}{\delta^2+1}&\geq-(2n-1)\delta-\frac{\delta\log(\delta+1)}{\delta^2+1}
+\frac{\delta\log(\delta)}{\delta^2+1}
\end{align}
for any $s\in\Interval{s_2,s_3}$. 
The right-hand side of \eqref{eqn:20210615-a<delta} converges to $0$ as $\delta\searrow0$. 
Therefore, we may choose $\delta>0$ such that \eqref{eqn:20210615-a<delta} implies $\alpha(s)>-\varepsilon$ for all $s\in\Interval{s_2,s_3}$. 

We must have $s_3<s_*$ because otherwise, Lemma \ref{lem:well-posed_centre} \eqref{eqn:20210831-r=pi} implies $r(s)\to\pi$ as $s\to s_*$ which by Lemma \ref{lem:well-posed_centre} \eqref{eqn:20210831-r<pi} contradicts the fact that $\alpha(s)$ must have a nonzero limit as $s\to s_*$ due to \eqref{eqn:20210831-da/ds<0}. 
Consequently, $\vartheta(s_3)=0$ and the claim follows. 
\end{proof}

\begin{lemma}\label{lem:small_alpha-iteration}
Let $n\in\{2,3\}$ and let $0<\varepsilon<\pi/2$ and $1\leq k\in\N$ be arbitrary. 
Then there exists $a_{\varepsilon,k}\in\interval{0,\pi/2}$ 
such that for any $\alpha_0\in\intervaL{0,a_{\varepsilon,k}}$ the solution $(r,\vartheta,\alpha)\colon\Interval{0,s_*}\to D$ 
with initial data $(r,\vartheta,\alpha)(0)=(\pi/2,0,\alpha_0)$ as constructed in Lemma~\ref{lem:well-posed_centre} allows 
$0<\sigma_1<\sigma_2<\ldots<\sigma_k<s_*$ such that for all $j\in\{1,\ldots,k\}$ 
\begin{align*}
\vartheta(\sigma_j)&=0, &
\abs{\alpha(\sigma_j)}&\leq\varepsilon.
\end{align*}
\end{lemma}

\begin{proof}
We prove the claim by induction in $k$. 
The case $k=1$ follows by Lemma \ref{lem:small_alpha} with $\sigma_1=s_3$. 
Given any fixed $\varepsilon>0$ let $a_\varepsilon\in\interval{0,\pi/2}$ be as in Lemma \ref{lem:small_alpha}. 
If the statement is true for a certain $k\in\N$, we may apply it with $\min\{\varepsilon,a_\varepsilon\}$ in place of $\varepsilon$ to further obtain $\abs{\alpha(\sigma_k)}\leq a_\varepsilon$. 
Since the problem is invariant under the reflection 
$(r,\vartheta,\alpha)\mapsto(r,-\vartheta,-\alpha)$ we may reapply Lemmata \ref{lem:well-posed_centre}--\ref{lem:small_alpha} with initial data 
$\bigl(r(\sigma_k),0,\abs{\alpha(\sigma_k)}\bigr)$ to conclude that the original solution allows $\sigma_{k+1}\in\interval{\sigma_k,s_*}$ with $\vartheta(\sigma_{k+1})=0$ and $\abs{\alpha(\sigma_{k+1})}\leq\varepsilon$.  
\end{proof}

The following estimate can be interpreted as a complement to Lemma \ref{lem:small_alpha}, which we will crucially employ at the end of the next section.

\begin{lemma}\label{lem:zeros}
Given $r_0\in\interval{\pi/2,\pi}$ and $\alpha_0\in\interval{0,\pi/2}$ let $(r,\vartheta,\alpha)\colon\Interval{0,s_*}\to D$ be the solution to system \eqref{eqn:dr/ds2},\eqref{eqn:dt/ds2},\eqref{eqn:da/ds2} with initial data $(r,\vartheta,\alpha)(0)=(r_0,0,\alpha_0)$. 
If there exists $\sigma\in\interval{0,s_*}$ such that $\vartheta(\sigma)=0$ and $\vartheta(s)>0$ for all $s\in\interval{0,\sigma}$, then
\begin{align*}
\abs[\bigg]{\frac{\alpha(\sigma)}{\alpha_0}}&\geq\max\left\{1,~\frac{2n-1}{2n-2}\abs[\big]{\cos(r_0)}\right\}.  
\end{align*}
Moreover, there exists $s_2\in\interval{0,\sigma}$ such that  
\begin{align*}
\max_{s\in[0,\sigma]}\abs{\vartheta(s)}&=\vartheta(s_2)\geq\frac{\alpha_0}{2n-2}. 
\end{align*}
\end{lemma}

\begin{proof}
By assumption, $\vartheta$ must attain an interior maximum in $[0,\sigma]$. 
In particular, by equations \eqref{eqn:dt/ds2} and \eqref{eqn:da/ds2}, there exists $s_2\in\interval{0,\sigma}$ such that 
\begin{align}\label{eqn:20210915-1}
\vartheta(s_2)&>0, & \alpha(s_2)&=0. 
\end{align}
By equation \eqref{eqn:da/ds2}, we have $d\alpha/ds(s_2)<0$, and in fact $d\alpha/ds(s)<0$  as long as $\vartheta(s)>0$ and $\alpha(s)\leq0$. 
Therefore, $\alpha(s)$ is strictly decreasing in $\Interval{s_2,\sigma}$ and $s_2\in\interval{0,\sigma}$ is uniquely determined by property \eqref{eqn:20210915-1}.  
As in \eqref{eqn:20210612-1} above, we can show 
\begin{align}\label{eqn:20210915>}
\bigl(\cos\bigl(2\vartheta(s_2)\bigr)\bigr)^{n-1}&\leq\cos(\alpha_0). 
\shortintertext{This implies}
\label{eqn:20210915-cos(t2)}
\cos\bigl(2\vartheta(s_2)\bigr)&\leq\bigl(\cos(\alpha_0)\bigr)^{\frac{1}{n-1}}
\leq\cos\Bigl(\frac{\alpha_0}{n-1}\Bigr), 
\end{align}
where the last inequality follows for all $\alpha_0\in\interval{0,\pi/2}$ and all $2\leq n\in\N$ simply by computing 
\begin{align*}
\frac{d}{d\alpha}\bigl(\cos(\alpha)\bigr)^{\frac{1}{n-1}}
&=\frac{-\sin(\alpha)}{n-1}\bigl(\cos(\alpha)\bigr)^{-\frac{n-2}{n-1}}
\leq\frac{-\sin(\alpha)}{n-1} 
\leq\frac{-1}{n-1}\sin\Bigl(\frac{\alpha}{n-1}\Bigr)
=\frac{d}{d\alpha}\cos\Bigl(\frac{\alpha}{n-1}\Bigr).
\end{align*}
In particular, \eqref{eqn:20210915-cos(t2)} implies 
\begin{align}\label{eqn:20210915-t2}
\vartheta(s_2)&\geq\frac{\alpha_0}{2n-2}.
\end{align} 
For any $s\in[s_2,\sigma]$ we have 
$d\vartheta/ds(s)\leq 0$ and so
\begin{align}\label{eqn:20210915=}
\frac{d\alpha}{ds}
&=\Bigl(-(2n-2)\tan(2\vartheta)\cot(\alpha)-(2n-1)\cos(r) \Bigr)\frac{d\vartheta}{ds}
\leq -(2n-2)\tan(2\vartheta)\cot(\alpha)\frac{d\vartheta}{ds}.
\end{align}
Multiplying by $\tan(\alpha)<0$ we obtain exactly the same estimate as in \eqref{eqn:20210612-0}, here valid for $s\in[s_2,\sigma]$. 
Integration implies
$-\log\bigl(\cos\bigl(\alpha(\sigma)\bigr)\bigr) 
\geq-(n-1)\log\bigl(\cos\bigl(2\vartheta(s_2)\bigr)\bigr)$
which is equivalent to 
\begin{align}\label{eqn:20210915<}
 \cos\bigl(\alpha(\sigma)\bigr)
\leq \bigl(\cos\bigl(2\vartheta(s_2)\bigr)\bigr)^{n-1}.
\end{align}
Combining \eqref{eqn:20210915>} and \eqref{eqn:20210915<} we obtain $\cos(\alpha(\sigma))\leq\cos(\alpha_0)$ which implies $\abs{\alpha(\sigma)}\geq \abs{\alpha_0}$. 
In analogy to estimate \eqref{eqn:20210915=} we have $d\alpha/ds\leq-(2n-1)\cos(r_0)d\vartheta/ds$ for $s\in[s_2,\sigma]$. 
Integration then yields 
\begin{align*}
\alpha(\sigma)\leq (2n-1)\cos(r_0)\,\vartheta(s_2)
\leq \tfrac{2n-1}{2n-2}\cos(r_0)\,\alpha_0
\end{align*}
where we used $\cos(r_0)\leq0$ and estimate \eqref{eqn:20210915-t2}. 
Since $\alpha(\sigma)<0<\alpha_0$ the claim follows. 
\end{proof}

\begin{remark}\label{rem:MonotonTheta}
For later reference, we note that it follows from the proof of the previous lemma that 
\[
f\vcentcolon
=\Bigl(\frac{dr}{ds}\Bigr)^{-1}\frac{d\vartheta}{ds}
=\frac{\tan(\alpha)}{\sin(r)}
\]
is weakly decreasing on $[s_2,\sigma]$ and, in addition, it satisfies $-f(\sigma)\geq f(0)$ because by monotonicity in the first coordinate $\pi>r(\sigma)\geq r_0\geq\pi/2$ and by what we just proved $-\alpha(\sigma)\geq\alpha_0$.
\end{remark}

\subsection{Start of the Induction: The \texorpdfstring{$\boldsymbol\infty$}{infinity}-figure}

The construction of an immersed minimal hypertorus marks the start of an iteration process which will ultimately provide a family of infinitely many examples. 
We apply a continuity argument twice: it first yields an auxiliary trajectory that serves as input for when we use the argument a second time. 
The entire approach, which we are about to explain in detail, is summarised visually in Figure \ref{fig:shooting2}. 

\begin{lemma}[{see Figure \ref{fig:shooting2}\;(1)}]\label{lem:20210828}
For every $\rho\in\interval{\pi/2,\pi}$ there exists $a\in\interval{0,\pi/2}$ such that for all 
$r_0\in\Interval{\rho,\pi}$ and all 
$\alpha_0\in\Interval{a,\pi/2}$ the solution $(r,\vartheta,\alpha)\colon\Interval{0,s_*}\to D$  with $(r,\vartheta,\alpha)(0)=(r_0,0,\alpha_0)$ as constructed in Lemma~\ref{lem:well-posed_centre} satisfies 
\(\vartheta(s)>0\) 
for all $s\in\interval{0,s_*}$ and $\alpha(s)\to\pi/2$ as $s\nearrow s_*$.  
\end{lemma}
 
\begin{proof}
Let $\rho\in\interval{\pi/2,\pi}$ be fixed and $r_0\in\Interval{\rho,\pi}$ arbitrary. 
We recall from equations \eqref{eqn:dr/ds2} and \eqref{eqn:dt/ds2} that $r$ is strictly increasing, and that $\vartheta$ is also strictly increasing as long as $\alpha>0$. 
Equation \eqref{eqn:da/ds2} implies $d\alpha/ds(0)=-(2n-1)\cot(r_0)\sin(\alpha_0)>0$. 
Hence there exists $s_\alpha>0$ which is largest with the property that $\alpha$ is strictly increasing in $s\in\Interval{0,s_\alpha}$. 
Let $s_\vartheta$ be the least upper bound of the set of all $b\in\Interval{0,s_*}$ with the property that 
for all $s\in[0,b]$ 
\begin{align}\label{eqn:20210828-2}
\tan\bigl(2\vartheta(s)\bigr)\leq-\cos(\rho)\tan(\alpha_0).
\end{align}  
We claim $s_\alpha\geq s_\vartheta$. 
Equations \eqref{eqn:dt/ds2} and \eqref{eqn:da/ds2} imply 
\begin{align*} 
\frac{d\alpha}{ds}
&=\Bigl(-(2n-2)\tan(2\vartheta)\cot(\alpha)-(2n-1)\cos(r) \Bigr)\frac{d\vartheta}{ds}. 
\end{align*}
If $s_\alpha<s_\vartheta$, then $d\alpha/ds(s_\alpha)=0$. 
Thus, since $-\cos(r)$ and $\tan(\alpha)$ are both strictly increasing in $\Interval{0,s_\alpha}$, 
\begin{align*}
\tan\bigl(2\vartheta(s_\alpha)\bigr)
&=-\frac{(2n-1)}{(2n-2)}\cos\bigl(r(s_\alpha)\bigr)
\tan\bigl(\alpha(s_\alpha)\bigr)
>-\cos(\rho)\tan(\alpha_0)
\end{align*}
in contradiction to \eqref{eqn:20210828-2}. 
Hence, the claim holds and we have indeed $s_\alpha\geq s_\vartheta$. 
In the interval $\Interval{0,s_\alpha}$, the estimate 
\begin{align}\label{eqn:20210828-1}
\frac{d\alpha}{ds}
&\geq \Bigl(-(2n-2)\tan(2\vartheta)\cot(\alpha_0)-(2n-1)\cos(\rho) \Bigr)\frac{d\vartheta}{ds} 
\end{align}
is true. 
Integrating \eqref{eqn:20210828-1} from $s=0$ to any $s\in\interval{0,s_\alpha}$, we obtain 
\begin{align}\label{eqn:20210828-3}
\frac{\pi}{2}-\alpha_0\geq
\alpha-\alpha_0&\geq(n-1)\cot(\alpha_0)\log\bigl(\cos(2\vartheta)\bigr)-(2n-1)\cos(\rho)\vartheta. 
\end{align}
Now, towards a contradiction, suppose that $s_{\vartheta}<s_*$. 
Then, by definition of $s_{\vartheta}$, 
\begin{align*}
\tan\bigl(2\vartheta(s_{\vartheta})\bigr)&=-\cos(\rho)\tan(\alpha_0)=\vcentcolon c_0
  &\Rightarrow~
\cos\bigl(2\vartheta(s_{\vartheta})\bigr)&=
\bigl(c_0^2+1\bigr)^{-\frac{1}{2}}. 
\end{align*}
Multiplying estimate \eqref{eqn:20210828-3} by $-2/\cos(\rho)>0$, we obtain  
\begin{align}\notag
\frac{\pi-2\alpha_0}{-\cos(\rho)}
+2(n-1)\frac{\log\bigl(\cos(2\vartheta)\bigr)}{\cos(\rho)\tan(\alpha_0)}
&\geq(2n-1)2\vartheta
\intertext{which, specified to $s=s_\vartheta$, yields}
\label{eqn:20210921-1}
\frac{\pi-2\alpha_0}{-\cos(\rho)}
+(n-1)\frac{\log\bigl(c_0^2+1\bigr)}{c_0}
&\geq(2n-1)\arctan(c_0).
\end{align}
Since $c_0\to\infty$ as $\alpha_0\nearrow\pi/2$, the left-hand side of \eqref{eqn:20210921-1} converges to zero as $\alpha_0\nearrow\pi/2$ while the right-hand side of \eqref{eqn:20210921-1} converges to $(2n-1)\pi/2$. 
This contradiction proves  $s_\vartheta=s_*$ provided that $\alpha_0$ is sufficiently close to $\pi/2$. 
Since $s_\alpha\geq s_\vartheta$ we also have $s_\alpha=s_*$ in this case.  
In particular, all three variables are strictly increasing for $s\in\Interval{0,s_*}$.  
Hence, $\vartheta(s)\to\vartheta_*>0$ and $\alpha(s)\to\alpha_*>0$ as $s\nearrow s_*$. 
Thus, by Lemma \ref{lem:well-posed_centre} \eqref{eqn:20210831-r<pi}, we have $\lim_{s\to s_*}r(s)<\pi$. 
Moreover, estimate \eqref{eqn:20210828-2} implies $\vartheta_*<\pi/4$. 
Therefore, $\alpha_*=\pi/2$ follows by Lemma \ref{lem:well-posed_centre} \eqref{eqn:20210831-r=pi}. 
\end{proof}

\begin{remark}\label{rem:20210829}
The assumption $r_0\neq\pi/2$ is crucial in the proof of Lemma \ref{lem:20210828}. 
In fact, (e.\,g. by analysing the second derivative of $\alpha$) one can show for $r_0=\pi/2$ and any choice of $\alpha_0\in\interval{0,\pi/2}$ that $\alpha(s)$ is actually \emph{decreasing} while $\vartheta\geq0$. 
In what follows, we specialise to the case $r_0=2\pi/3$. 
\end{remark}

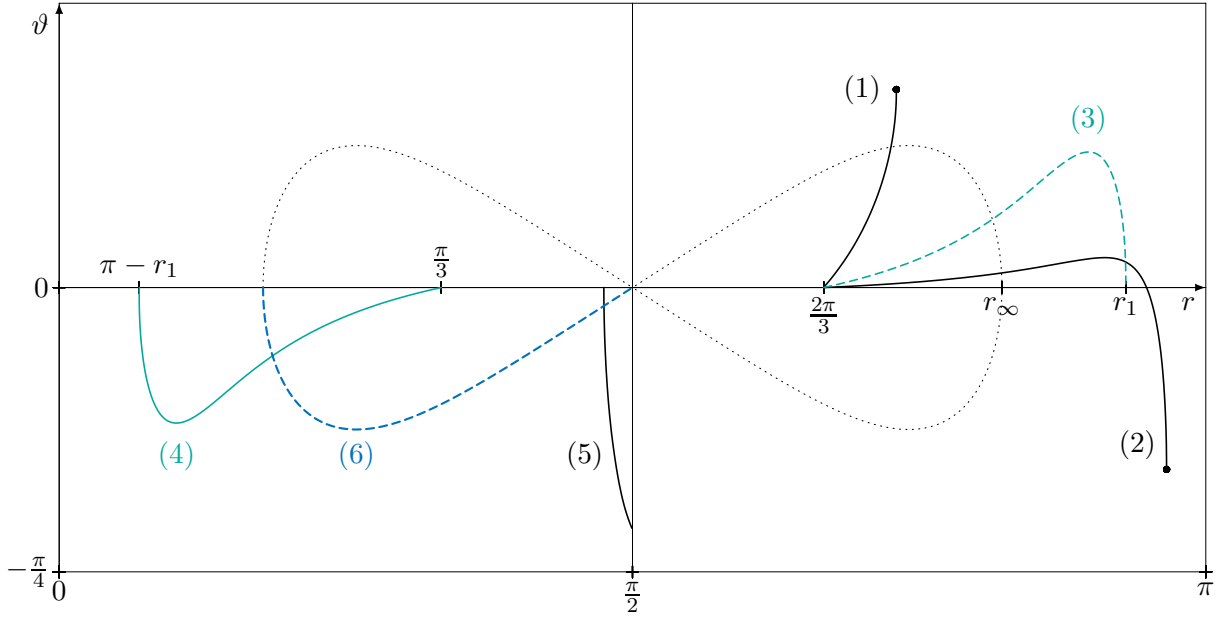
\begin{figure}[t]\centering
\begin{tikzpicture}[line cap=round,line join=round,baseline={(0,0)},scale=\textwidth/3.5cm]
\draw[very thin](0,0)grid [xstep=pi/2,ystep=pi/4](pi,pi/2);
\draw[-latex](0,pi/4)--++(pi,0)node[below left]{$r$};
\draw[-latex](0,0)--(0,pi/2)node[below left]{$\vartheta$};
\draw plot[plus](0,0)node[below]{$0$};
\draw plot[plus](0,0)node[left]{$-\frac{\pi}{4}$};
\draw plot[plus](pi/2,0)node[below]{$\frac{\pi}{2}$};
\draw plot[plus](pi,0)node[below]{$\pi$};
\draw plot[plus](0,pi/4)node[left]{$0$};
\draw plot[vdash](2*pi/3,pi/4)node[below]{$\frac{2\pi}{3}$}; 
\draw plot[vdash](2.9220,pi/4)node[below]{$r_1$};  
\draw plot[vdash](pi-0.5586,pi/4)node[below]{$r_\infty$};  
\draw plot[vdash](pi-2.9220,pi/4)node[above]{$\pi-r_1$};  
\draw plot[vdash](pi/3,pi/4)node[above]{$\frac{\pi}{3}$};  
\draw plot[bullet](2.2937,1.3326)node[left=.4ex]{(1)};
\draw plot[bullet](3.0341,0.2829)node[above left]{(2)};
\draw[color={cmyk,1:yellow,0.5;cyan,1}]
(2.8199,pi/2-0.39) node[above]{(3)}
(pi-2.8199,0.39) node[below]{(4)};
\draw (1.5198,0.39) node[below left]{(5)};
\draw[torustrajectory] (0.8146,0.39) node[below]{(6)};
\draw[smooth,dotted,scale=-1,shift={(-pi,-pi/2)}]plot coordinates { 
(0.5586,0.7854) (0.5587,0.7708) (0.5591,0.7562) (0.5599,0.7417) (0.5609,0.7273) (0.5622,0.7129) (0.5637,0.6987) (0.5656,0.6847) (0.5677,0.6708) (0.5701,0.6572) (0.5727,0.6437) (0.5757,0.6305) (0.5789,0.6176) (0.5823,0.6050) (0.5860,0.5926) (0.5899,0.5806) (0.5941,0.5689) (0.5985,0.5576) (0.6031,0.5466) (0.6079,0.5360) (0.6130,0.5258) (0.6182,0.5159) (0.6237,0.5065) (0.6293,0.4974) (0.6351,0.4887) (0.6411,0.4805) (0.6472,0.4726) (0.6535,0.4652) (0.6600,0.4581) (0.6665,0.4515) (0.6732,0.4453) (0.6800,0.4394) (0.6870,0.4340) (0.6940,0.4289) (0.7011,0.4242) (0.7083,0.4199) (0.7156,0.4159) (0.7230,0.4123) (0.7304,0.4090) (0.7379,0.4061) (0.7455,0.4035) (0.7531,0.4012) (0.7607,0.3993) (0.7683,0.3976) (0.7760,0.3962) (0.7837,0.3951) (0.7914,0.3943) (0.7992,0.3937) (0.8069,0.3933) (0.8146,0.3932) (0.8224,0.3934) (0.8301,0.3937) (0.8378,0.3942) (0.8456,0.3950) (0.8533,0.3959) (0.8610,0.3970) (0.8686,0.3983) (0.8763,0.3997) (0.8839,0.4013) (0.8916,0.4031) (0.8992,0.4050) (0.9067,0.4070) (0.9143,0.4091) (0.9218,0.4114) (0.9293,0.4137) (0.9368,0.4162) (0.9443,0.4188) (0.9517,0.4214) (0.9591,0.4242) (0.9665,0.4270) (0.9738,0.4300) (0.9811,0.4330) (0.9885,0.4360) (0.9957,0.4392) (1.0030,0.4424) (1.0102,0.4457) (1.0174,0.4490) (1.0246,0.4524) (1.0318,0.4558) (1.0389,0.4593) (1.0460,0.4628) (1.0531,0.4664) (1.0602,0.4700) (1.0672,0.4736) (1.0743,0.4773) (1.0813,0.4810) (1.0883,0.4848) (1.0952,0.4886) (1.1022,0.4924) (1.1091,0.4962) (1.1160,0.5001) (1.1230,0.5040) (1.1298,0.5079) (1.1367,0.5118) (1.1436,0.5158) (1.1504,0.5197) (1.1572,0.5237) (1.1641,0.5277) (1.1709,0.5317) (1.1776,0.5357) (1.1844,0.5398) (1.1912,0.5438) (1.1979,0.5479) (1.2047,0.5520) (1.2114,0.5561) (1.2181,0.5602) (1.2248,0.5643) (1.2315,0.5684) (1.2382,0.5725) (1.2449,0.5766) (1.2516,0.5808) (1.2582,0.5849) (1.2649,0.5891) (1.2715,0.5932) (1.2782,0.5974) (1.2848,0.6015) (1.2914,0.6057) (1.2980,0.6099) (1.3046,0.6141) (1.3113,0.6182) (1.3178,0.6224) (1.3244,0.6266) (1.3310,0.6308) (1.3376,0.6350) (1.3442,0.6392) (1.3507,0.6434) (1.3573,0.6476) (1.3639,0.6518) (1.3704,0.6560) (1.3770,0.6602) (1.3835,0.6644) (1.3901,0.6686) (1.3966,0.6728) (1.4031,0.6770) (1.4097,0.6812) (1.4162,0.6854) (1.4227,0.6896) (1.4293,0.6938) (1.4358,0.6980) (1.4423,0.7022) (1.4488,0.7064) (1.4553,0.7106) (1.4618,0.7148) (1.4684,0.7190) (1.4749,0.7233) (1.4814,0.7275) (1.4879,0.7317) (1.4944,0.7359) (1.5009,0.7401) (1.5074,0.7443) (1.5139,0.7485) (1.5204,0.7527) (1.5269,0.7569) (1.5334,0.7611) (1.5399,0.7653) (1.5464,0.7696) (1.5528,0.7738) (1.5593,0.7780) (1.5658,0.7822) (1.5723,0.7864) 
};
\draw[smooth,dotted,xscale=-1,shift={(-pi,0)}]plot coordinates { 
(0.5586,0.7854) (0.5587,0.7708) (0.5591,0.7562) (0.5599,0.7417) (0.5609,0.7273) (0.5622,0.7129) (0.5637,0.6987) (0.5656,0.6847) (0.5677,0.6708) (0.5701,0.6572) (0.5727,0.6437) (0.5757,0.6305) (0.5789,0.6176) (0.5823,0.6050) (0.5860,0.5926) (0.5899,0.5806) (0.5941,0.5689) (0.5985,0.5576) (0.6031,0.5466) (0.6079,0.5360) (0.6130,0.5258) (0.6182,0.5159) (0.6237,0.5065) (0.6293,0.4974) (0.6351,0.4887) (0.6411,0.4805) (0.6472,0.4726) (0.6535,0.4652) (0.6600,0.4581) (0.6665,0.4515) (0.6732,0.4453) (0.6800,0.4394) (0.6870,0.4340) (0.6940,0.4289) (0.7011,0.4242) (0.7083,0.4199) (0.7156,0.4159) (0.7230,0.4123) (0.7304,0.4090) (0.7379,0.4061) (0.7455,0.4035) (0.7531,0.4012) (0.7607,0.3993) (0.7683,0.3976) (0.7760,0.3962) (0.7837,0.3951) (0.7914,0.3943) (0.7992,0.3937) (0.8069,0.3933) (0.8146,0.3932) (0.8224,0.3934) (0.8301,0.3937) (0.8378,0.3942) (0.8456,0.3950) (0.8533,0.3959) (0.8610,0.3970) (0.8686,0.3983) (0.8763,0.3997) (0.8839,0.4013) (0.8916,0.4031) (0.8992,0.4050) (0.9067,0.4070) (0.9143,0.4091) (0.9218,0.4114) (0.9293,0.4137) (0.9368,0.4162) (0.9443,0.4188) (0.9517,0.4214) (0.9591,0.4242) (0.9665,0.4270) (0.9738,0.4300) (0.9811,0.4330) (0.9885,0.4360) (0.9957,0.4392) (1.0030,0.4424) (1.0102,0.4457) (1.0174,0.4490) (1.0246,0.4524) (1.0318,0.4558) (1.0389,0.4593) (1.0460,0.4628) (1.0531,0.4664) (1.0602,0.4700) (1.0672,0.4736) (1.0743,0.4773) (1.0813,0.4810) (1.0883,0.4848) (1.0952,0.4886) (1.1022,0.4924) (1.1091,0.4962) (1.1160,0.5001) (1.1230,0.5040) (1.1298,0.5079) (1.1367,0.5118) (1.1436,0.5158) (1.1504,0.5197) (1.1572,0.5237) (1.1641,0.5277) (1.1709,0.5317) (1.1776,0.5357) (1.1844,0.5398) (1.1912,0.5438) (1.1979,0.5479) (1.2047,0.5520) (1.2114,0.5561) (1.2181,0.5602) (1.2248,0.5643) (1.2315,0.5684) (1.2382,0.5725) (1.2449,0.5766) (1.2516,0.5808) (1.2582,0.5849) (1.2649,0.5891) (1.2715,0.5932) (1.2782,0.5974) (1.2848,0.6015) (1.2914,0.6057) (1.2980,0.6099) (1.3046,0.6141) (1.3113,0.6182) (1.3178,0.6224) (1.3244,0.6266) (1.3310,0.6308) (1.3376,0.6350) (1.3442,0.6392) (1.3507,0.6434) (1.3573,0.6476) (1.3639,0.6518) (1.3704,0.6560) (1.3770,0.6602) (1.3835,0.6644) (1.3901,0.6686) (1.3966,0.6728) (1.4031,0.6770) (1.4097,0.6812) (1.4162,0.6854) (1.4227,0.6896) (1.4293,0.6938) (1.4358,0.6980) (1.4423,0.7022) (1.4488,0.7064) (1.4553,0.7106) (1.4618,0.7148) (1.4684,0.7190) (1.4749,0.7233) (1.4814,0.7275) (1.4879,0.7317) (1.4944,0.7359) (1.5009,0.7401) (1.5074,0.7443) (1.5139,0.7485) (1.5204,0.7527) (1.5269,0.7569) (1.5334,0.7611) (1.5399,0.7653) (1.5464,0.7696) (1.5528,0.7738) (1.5593,0.7780) (1.5658,0.7822) (1.5723,0.7864) 
};
\draw[smooth,dotted,yscale=-1,shift={(0,-pi/2)}]plot coordinates { 
(0.5586,0.7854) (0.5587,0.7708) (0.5591,0.7562) (0.5599,0.7417) (0.5609,0.7273) (0.5622,0.7129) (0.5637,0.6987) (0.5656,0.6847) (0.5677,0.6708) (0.5701,0.6572) (0.5727,0.6437) (0.5757,0.6305) (0.5789,0.6176) (0.5823,0.6050) (0.5860,0.5926) (0.5899,0.5806) (0.5941,0.5689) (0.5985,0.5576) (0.6031,0.5466) (0.6079,0.5360) (0.6130,0.5258) (0.6182,0.5159) (0.6237,0.5065) (0.6293,0.4974) (0.6351,0.4887) (0.6411,0.4805) (0.6472,0.4726) (0.6535,0.4652) (0.6600,0.4581) (0.6665,0.4515) (0.6732,0.4453) (0.6800,0.4394) (0.6870,0.4340) (0.6940,0.4289) (0.7011,0.4242) (0.7083,0.4199) (0.7156,0.4159) (0.7230,0.4123) (0.7304,0.4090) (0.7379,0.4061) (0.7455,0.4035) (0.7531,0.4012) (0.7607,0.3993) (0.7683,0.3976) (0.7760,0.3962) (0.7837,0.3951) (0.7914,0.3943) (0.7992,0.3937) (0.8069,0.3933) (0.8146,0.3932) (0.8224,0.3934) (0.8301,0.3937) (0.8378,0.3942) (0.8456,0.3950) (0.8533,0.3959) (0.8610,0.3970) (0.8686,0.3983) (0.8763,0.3997) (0.8839,0.4013) (0.8916,0.4031) (0.8992,0.4050) (0.9067,0.4070) (0.9143,0.4091) (0.9218,0.4114) (0.9293,0.4137) (0.9368,0.4162) (0.9443,0.4188) (0.9517,0.4214) (0.9591,0.4242) (0.9665,0.4270) (0.9738,0.4300) (0.9811,0.4330) (0.9885,0.4360) (0.9957,0.4392) (1.0030,0.4424) (1.0102,0.4457) (1.0174,0.4490) (1.0246,0.4524) (1.0318,0.4558) (1.0389,0.4593) (1.0460,0.4628) (1.0531,0.4664) (1.0602,0.4700) (1.0672,0.4736) (1.0743,0.4773) (1.0813,0.4810) (1.0883,0.4848) (1.0952,0.4886) (1.1022,0.4924) (1.1091,0.4962) (1.1160,0.5001) (1.1230,0.5040) (1.1298,0.5079) (1.1367,0.5118) (1.1436,0.5158) (1.1504,0.5197) (1.1572,0.5237) (1.1641,0.5277) (1.1709,0.5317) (1.1776,0.5357) (1.1844,0.5398) (1.1912,0.5438) (1.1979,0.5479) (1.2047,0.5520) (1.2114,0.5561) (1.2181,0.5602) (1.2248,0.5643) (1.2315,0.5684) (1.2382,0.5725) (1.2449,0.5766) (1.2516,0.5808) (1.2582,0.5849) (1.2649,0.5891) (1.2715,0.5932) (1.2782,0.5974) (1.2848,0.6015) (1.2914,0.6057) (1.2980,0.6099) (1.3046,0.6141) (1.3113,0.6182) (1.3178,0.6224) (1.3244,0.6266) (1.3310,0.6308) (1.3376,0.6350) (1.3442,0.6392) (1.3507,0.6434) (1.3573,0.6476) (1.3639,0.6518) (1.3704,0.6560) (1.3770,0.6602) (1.3835,0.6644) (1.3901,0.6686) (1.3966,0.6728) (1.4031,0.6770) (1.4097,0.6812) (1.4162,0.6854) (1.4227,0.6896) (1.4293,0.6938) (1.4358,0.6980) (1.4423,0.7022) (1.4488,0.7064) (1.4553,0.7106) (1.4618,0.7148) (1.4684,0.7190) (1.4749,0.7233) (1.4814,0.7275) (1.4879,0.7317) (1.4944,0.7359) (1.5009,0.7401) (1.5074,0.7443) (1.5139,0.7485) (1.5204,0.7527) (1.5269,0.7569) (1.5334,0.7611) (1.5399,0.7653) (1.5464,0.7696) (1.5528,0.7738) (1.5593,0.7780) (1.5658,0.7822) (1.5723,0.7864) 
};
\draw[smooth,trajectory]plot coordinates { 
(2.0944,0.7854) (2.0998,0.7918) (2.1052,0.7982) (2.1106,0.8047) (2.1159,0.8113) (2.1211,0.8180) (2.1263,0.8248) (2.1314,0.8316) (2.1364,0.8385) (2.1414,0.8455) (2.1464,0.8526) (2.1513,0.8598) (2.1561,0.8670) (2.1609,0.8744) (2.1656,0.8818) (2.1702,0.8893) (2.1748,0.8968) (2.1793,0.9045) (2.1837,0.9122) (2.1881,0.9201) (2.1924,0.9280) (2.1966,0.9360) (2.2008,0.9440) (2.2048,0.9522) (2.2089,0.9604) (2.2128,0.9687) (2.2167,0.9771) (2.2204,0.9855) (2.2242,0.9941) (2.2278,1.0027) (2.2313,1.0114) (2.2348,1.0202) (2.2382,1.0290) (2.2415,1.0379) (2.2447,1.0469) (2.2479,1.0560) (2.2509,1.0651) (2.2539,1.0743) (2.2567,1.0836) (2.2595,1.0929) (2.2622,1.1024) (2.2648,1.1118) (2.2673,1.1214) (2.2697,1.1310) (2.2720,1.1406) (2.2742,1.1504) (2.2763,1.1601) (2.2782,1.1700) (2.2801,1.1799) (2.2819,1.1898) (2.2835,1.1998) (2.2851,1.2098) (2.2865,1.2199) (2.2878,1.2300) (2.2890,1.2401) (2.2901,1.2503) (2.2910,1.2606) (2.2918,1.2708) (2.2925,1.2811) (2.2930,1.2914) (2.2934,1.3017) (2.2936,1.3120) (2.2937,1.3223) (2.2937,1.3326) 
};
\draw[smooth,trajectory]plot coordinates { 
(2.0944,0.7854) (2.1021,0.7857) (2.1099,0.7861) (2.1176,0.7864) (2.1253,0.7868) (2.1331,0.7872) (2.1408,0.7875) (2.1485,0.7879) (2.1563,0.7883) (2.1640,0.7887) (2.1717,0.7891) (2.1795,0.7895) (2.1872,0.7899) (2.1949,0.7903) (2.2026,0.7908) (2.2104,0.7912) (2.2181,0.7916) (2.2258,0.7921) (2.2336,0.7926) (2.2413,0.7930) (2.2490,0.7935) (2.2568,0.7940) (2.2645,0.7945) (2.2722,0.7950) (2.2799,0.7955) (2.2877,0.7960) (2.2954,0.7966) (2.3031,0.7971) (2.3109,0.7977) (2.3186,0.7982) (2.3263,0.7988) (2.3340,0.7994) (2.3418,0.8000) (2.3495,0.8006) (2.3572,0.8012) (2.3649,0.8018) (2.3727,0.8025) (2.3804,0.8032) (2.3881,0.8038) (2.3958,0.8045) (2.4036,0.8052) (2.4113,0.8059) (2.4190,0.8066) (2.4267,0.8074) (2.4345,0.8081) (2.4422,0.8089) (2.4499,0.8097) (2.4576,0.8105) (2.4654,0.8113) (2.4731,0.8121) (2.4808,0.8129) (2.4885,0.8138) (2.4962,0.8147) (2.5040,0.8156) (2.5117,0.8165) (2.5194,0.8174) (2.5271,0.8183) (2.5348,0.8193) (2.5426,0.8203) (2.5503,0.8213) (2.5580,0.8223) (2.5657,0.8233) (2.5734,0.8244) (2.5811,0.8255) (2.5889,0.8266) (2.5966,0.8277) (2.6043,0.8288) (2.6120,0.8300) (2.6197,0.8312) (2.6274,0.8323) (2.6352,0.8336) (2.6429,0.8348) (2.6506,0.8360) (2.6583,0.8373) (2.6660,0.8386) (2.6737,0.8399) (2.6815,0.8412) (2.6892,0.8426) (2.6969,0.8439) (2.7046,0.8453) (2.7123,0.8467) (2.7200,0.8480) (2.7278,0.8494) (2.7355,0.8508) (2.7432,0.8522) (2.7509,0.8536) (2.7586,0.8549) (2.7664,0.8563) (2.7741,0.8576) (2.7818,0.8588) (2.7895,0.8601) (2.7972,0.8613) (2.8050,0.8624) (2.8127,0.8636) (2.8204,0.8646) (2.8282,0.8656) (2.8359,0.8664) (2.8437,0.8671) (2.8514,0.8676) (2.8591,0.8679) (2.8669,0.8680) (2.8746,0.8677) (2.8824,0.8671) (2.8901,0.8661) (2.8978,0.8646) (2.9055,0.8626) (2.9132,0.8600) (2.9209,0.8566) (2.9285,0.8523) (2.9361,0.8470) (2.9437,0.8406) (2.9513,0.8327) (2.9589,0.8231) (2.9664,0.8115) (2.9738,0.7975) (2.9811,0.7808) (2.9883,0.7610) (2.9952,0.7374) (3.0019,0.7096) (3.0082,0.6770) (3.0141,0.6391) (3.0196,0.5951) (3.0244,0.5444) (3.0284,0.4870) (3.0315,0.4234) (3.0335,0.3547) (3.0341,0.2829) 
};
\draw[smooth,trajectory,densely dashed,color={cmyk,1:yellow,0.5;cyan,1}]plot coordinates { 
(2.0944,0.7854) (2.1020,0.7871) (2.1096,0.7888) (2.1172,0.7905) (2.1248,0.7923) (2.1324,0.7940) (2.1400,0.7959) (2.1475,0.7977) (2.1551,0.7996) (2.1627,0.8015) (2.1703,0.8035) (2.1778,0.8055) (2.1854,0.8075) (2.1929,0.8096) (2.2005,0.8117) (2.2080,0.8138) (2.2156,0.8160) (2.2231,0.8182) (2.2306,0.8204) (2.2382,0.8227) (2.2457,0.8251) (2.2532,0.8274) (2.2607,0.8298) (2.2682,0.8323) (2.2757,0.8348) (2.2832,0.8374) (2.2907,0.8400) (2.2982,0.8426) (2.3056,0.8453) (2.3131,0.8480) (2.3206,0.8508) (2.3280,0.8537) (2.3355,0.8566) (2.3429,0.8595) (2.3504,0.8625) (2.3578,0.8656) (2.3652,0.8687) (2.3726,0.8718) (2.3800,0.8751) (2.3874,0.8784) (2.3948,0.8817) (2.4022,0.8851) (2.4096,0.8886) (2.4170,0.8921) (2.4243,0.8957) (2.4317,0.8994) (2.4390,0.9031) (2.4464,0.9069) (2.4537,0.9108) (2.4610,0.9147) (2.4684,0.9188) (2.4757,0.9228) (2.4830,0.9270) (2.4903,0.9313) (2.4976,0.9356) (2.5048,0.9400) (2.5121,0.9444) (2.5194,0.9490) (2.5266,0.9536) (2.5339,0.9584) (2.5411,0.9632) (2.5484,0.9680) (2.5556,0.9730) (2.5628,0.9780) (2.5700,0.9832) (2.5772,0.9884) (2.5844,0.9937) (2.5916,0.9991) (2.5988,1.0045) (2.6060,1.0100) (2.6132,1.0157) (2.6204,1.0214) (2.6276,1.0271) (2.6348,1.0330) (2.6420,1.0389) (2.6492,1.0449) (2.6564,1.0509) (2.6636,1.0570) (2.6708,1.0632) (2.6781,1.0693) (2.6853,1.0755) (2.6925,1.0817) (2.6998,1.0879) (2.7071,1.0941) (2.7144,1.1002) (2.7217,1.1063) (2.7291,1.1123) (2.7364,1.1181) (2.7438,1.1238) (2.7513,1.1293) (2.7588,1.1346) (2.7663,1.1396) (2.7738,1.1442) (2.7814,1.1484) (2.7891,1.1521) (2.7967,1.1553) (2.8044,1.1576) (2.8122,1.1592) (2.8199,1.1598) (2.8276,1.1593) (2.8354,1.1575) (2.8430,1.1544) (2.8506,1.1496) (2.8581,1.1429) (2.8655,1.1341) (2.8726,1.1231) (2.8795,1.1097) (2.8862,1.0938) (2.8924,1.0754) (2.8982,1.0544) (2.9036,1.0309) (2.9084,1.0049) (2.9125,0.9764) (2.9160,0.9457) (2.9188,0.9130) (2.9208,0.8789) (2.9220,0.8439) (2.9224,0.8085) (2.9220,pi/4) 
};

\draw[smooth,trajectory,color={cmyk,1:yellow,0.5;cyan,1},scale=-1,shift={(-pi,-pi/2)}]plot coordinates { 
(2.0944,0.7854) (2.1020,0.7871) (2.1096,0.7888) (2.1172,0.7905) (2.1248,0.7923) (2.1324,0.7940) (2.1400,0.7959) (2.1475,0.7977) (2.1551,0.7996) (2.1627,0.8015) (2.1703,0.8035) (2.1778,0.8055) (2.1854,0.8075) (2.1929,0.8096) (2.2005,0.8117) (2.2080,0.8138) (2.2156,0.8160) (2.2231,0.8182) (2.2306,0.8204) (2.2382,0.8227) (2.2457,0.8251) (2.2532,0.8274) (2.2607,0.8298) (2.2682,0.8323) (2.2757,0.8348) (2.2832,0.8374) (2.2907,0.8400) (2.2982,0.8426) (2.3056,0.8453) (2.3131,0.8480) (2.3206,0.8508) (2.3280,0.8537) (2.3355,0.8566) (2.3429,0.8595) (2.3504,0.8625) (2.3578,0.8656) (2.3652,0.8687) (2.3726,0.8718) (2.3800,0.8751) (2.3874,0.8784) (2.3948,0.8817) (2.4022,0.8851) (2.4096,0.8886) (2.4170,0.8921) (2.4243,0.8957) (2.4317,0.8994) (2.4390,0.9031) (2.4464,0.9069) (2.4537,0.9108) (2.4610,0.9147) (2.4684,0.9188) (2.4757,0.9228) (2.4830,0.9270) (2.4903,0.9313) (2.4976,0.9356) (2.5048,0.9400) (2.5121,0.9444) (2.5194,0.9490) (2.5266,0.9536) (2.5339,0.9584) (2.5411,0.9632) (2.5484,0.9680) (2.5556,0.9730) (2.5628,0.9780) (2.5700,0.9832) (2.5772,0.9884) (2.5844,0.9937) (2.5916,0.9991) (2.5988,1.0045) (2.6060,1.0100) (2.6132,1.0157) (2.6204,1.0214) (2.6276,1.0271) (2.6348,1.0330) (2.6420,1.0389) (2.6492,1.0449) (2.6564,1.0509) (2.6636,1.0570) (2.6708,1.0632) (2.6781,1.0693) (2.6853,1.0755) (2.6925,1.0817) (2.6998,1.0879) (2.7071,1.0941) (2.7144,1.1002) (2.7217,1.1063) (2.7291,1.1123) (2.7364,1.1181) (2.7438,1.1238) (2.7513,1.1293) (2.7588,1.1346) (2.7663,1.1396) (2.7738,1.1442) (2.7814,1.1484) (2.7891,1.1521) (2.7967,1.1553) (2.8044,1.1576) (2.8122,1.1592) (2.8199,1.1598) (2.8276,1.1593) (2.8354,1.1575) (2.8430,1.1544) (2.8506,1.1496) (2.8581,1.1429) (2.8655,1.1341) (2.8726,1.1231) (2.8795,1.1097) (2.8862,1.0938) (2.8924,1.0754) (2.8982,1.0544) (2.9036,1.0309) (2.9084,1.0049) (2.9125,0.9764) (2.9160,0.9457) (2.9188,0.9130) (2.9208,0.8789) (2.9220,0.8439) (2.9224,0.8085) (2.9220,pi/4) 
};
\draw[smooth,trajectory]plot coordinates { 
(1.4923,0.7854) (1.4923,0.7708) (1.4924,0.7561) (1.4925,0.7415) (1.4927,0.7269) (1.4930,0.7122) (1.4933,0.6976) (1.4936,0.6830) (1.4940,0.6684) (1.4944,0.6537) (1.4949,0.6391) (1.4955,0.6245) (1.4961,0.6099) (1.4968,0.5953) (1.4975,0.5807) (1.4982,0.5660) (1.4991,0.5514) (1.5000,0.5368) (1.5009,0.5223) (1.5019,0.5077) (1.5030,0.4931) (1.5041,0.4785) (1.5053,0.4639) (1.5066,0.4494) (1.5079,0.4348) (1.5094,0.4203) (1.5109,0.4057) (1.5124,0.3912) (1.5141,0.3767) (1.5159,0.3623) (1.5178,0.3478) (1.5198,0.3333) (1.5218,0.3189) (1.5240,0.3045) (1.5262,0.2900) (1.5286,0.2756) (1.5311,0.2612) (1.5337,0.2468) (1.5366,0.2324) (1.5396,0.2181) (1.5428,0.2038) (1.5463,0.1897) (1.5501,0.1756) (1.5543,0.1616) (1.5588,0.1477) (1.5638,0.1340) (1.5692,0.1205)  
};
\draw[smooth,torustrajectory,densely dashed,thick]plot coordinates { 
(0.5586,0.7854) (0.5587,0.7708) (0.5591,0.7562) (0.5599,0.7417) (0.5609,0.7273) (0.5622,0.7129) (0.5637,0.6987) (0.5656,0.6847) (0.5677,0.6708) (0.5701,0.6572) (0.5727,0.6437) (0.5757,0.6305) (0.5789,0.6176) (0.5823,0.6050) (0.5860,0.5926) (0.5899,0.5806) (0.5941,0.5689) (0.5985,0.5576) (0.6031,0.5466) (0.6079,0.5360) (0.6130,0.5258) (0.6182,0.5159) (0.6237,0.5065) (0.6293,0.4974) (0.6351,0.4887) (0.6411,0.4805) (0.6472,0.4726) (0.6535,0.4652) (0.6600,0.4581) (0.6665,0.4515) (0.6732,0.4453) (0.6800,0.4394) (0.6870,0.4340) (0.6940,0.4289) (0.7011,0.4242) (0.7083,0.4199) (0.7156,0.4159) (0.7230,0.4123) (0.7304,0.4090) (0.7379,0.4061) (0.7455,0.4035) (0.7531,0.4012) (0.7607,0.3993) (0.7683,0.3976) (0.7760,0.3962) (0.7837,0.3951) (0.7914,0.3943) (0.7992,0.3937) (0.8069,0.3933) (0.8146,0.3932) (0.8224,0.3934) (0.8301,0.3937) (0.8378,0.3942) (0.8456,0.3950) (0.8533,0.3959) (0.8610,0.3970) (0.8686,0.3983) (0.8763,0.3997) (0.8839,0.4013) (0.8916,0.4031) (0.8992,0.4050) (0.9067,0.4070) (0.9143,0.4091) (0.9218,0.4114) (0.9293,0.4137) (0.9368,0.4162) (0.9443,0.4188) (0.9517,0.4214) (0.9591,0.4242) (0.9665,0.4270) (0.9738,0.4300) (0.9811,0.4330) (0.9885,0.4360) (0.9957,0.4392) (1.0030,0.4424) (1.0102,0.4457) (1.0174,0.4490) (1.0246,0.4524) (1.0318,0.4558) (1.0389,0.4593) (1.0460,0.4628) (1.0531,0.4664) (1.0602,0.4700) (1.0672,0.4736) (1.0743,0.4773) (1.0813,0.4810) (1.0883,0.4848) (1.0952,0.4886) (1.1022,0.4924) (1.1091,0.4962) (1.1160,0.5001) (1.1230,0.5040) (1.1298,0.5079) (1.1367,0.5118) (1.1436,0.5158) (1.1504,0.5197) (1.1572,0.5237) (1.1641,0.5277) (1.1709,0.5317) (1.1776,0.5357) (1.1844,0.5398) (1.1912,0.5438) (1.1979,0.5479) (1.2047,0.5520) (1.2114,0.5561) (1.2181,0.5602) (1.2248,0.5643) (1.2315,0.5684) (1.2382,0.5725) (1.2449,0.5766) (1.2516,0.5808) (1.2582,0.5849) (1.2649,0.5891) (1.2715,0.5932) (1.2782,0.5974) (1.2848,0.6015) (1.2914,0.6057) (1.2980,0.6099) (1.3046,0.6141) (1.3113,0.6182) (1.3178,0.6224) (1.3244,0.6266) (1.3310,0.6308) (1.3376,0.6350) (1.3442,0.6392) (1.3507,0.6434) (1.3573,0.6476) (1.3639,0.6518) (1.3704,0.6560) (1.3770,0.6602) (1.3835,0.6644) (1.3901,0.6686) (1.3966,0.6728) (1.4031,0.6770) (1.4097,0.6812) (1.4162,0.6854) (1.4227,0.6896) (1.4293,0.6938) (1.4358,0.6980) (1.4423,0.7022) (1.4488,0.7064) (1.4553,0.7106) (1.4618,0.7148) (1.4684,0.7190) (1.4749,0.7233) (1.4814,0.7275) (1.4879,0.7317) (1.4944,0.7359) (1.5009,0.7401) (1.5074,0.7443) (1.5139,0.7485) (1.5204,0.7527) (1.5269,0.7569) (1.5334,0.7611) (1.5399,0.7653) (1.5464,0.7696) (1.5528,0.7738) (1.5593,0.7780) (1.5658,0.7822) (1.5723,0.7864) 
};
\end{tikzpicture}
\caption{Trajectories used to construct the ``$\infty$-figure'': 
\ref{trajectory(1)}--\ref{trajectory(3)} start from fixed $r_0=2\pi/3$ with varying $\alpha_0\in\interval{0,\pi/2}$ while \ref{trajectory(4)}--\ref{trajectory(6)} start with fixed $\alpha_0=-\pi/2$ and $r_0\in\interval{0,\pi/2}$ varies.}
\label{fig:shooting2}%
\begin{enumerate}[label={(\arabic*)}]
\item\label{trajectory(1)} see Lemma \ref{lem:20210828}: 
If $\alpha_0$ is close to $\pi/2$ and $r_0>\pi/2$ then 
$\alpha(s)\to\pi/2$ as $s\nearrow s_*$.  
\item\label{trajectory(2)} see Lemma \ref{lem:small_alpha}: If $\alpha_0$ is close to $0$ then the  trajectory intersects $\{\vartheta=0\}$ for some $s>0$. 
\item\label{trajectory(3)} see Lemma \ref{lem:20210829-green}: 
There exists $\alpha_0\in\interval{0,\pi/2}$ such that 
$(\vartheta,\alpha)(s)\to(0,-\pi/2)$ as $s\nearrow s_*$.
\item\label{trajectory(4)} reflection of trajectory \ref{trajectory(3)}: 
Intersection with $\{\vartheta=0\}\cap\{r<\pi/2\}$ for $r_0=\pi-r_1$. 
\item\label{trajectory(5)} see Lemma \ref{lem:side}: Intersection with $\{r=\pi/2\}$ if $r_0\in\interval{0,\pi/2}$ is close to $\pi/2$ and $\alpha_0=-\pi/2$. 
\item\label{trajectory(6)} see Lemma \ref{lem:20210829-blue}:  Intersection with $\{r=\pi/2\}\cap\{\vartheta=0\}$ for some $r_0\in\interval{0,\pi/2}$ and $\alpha_0=-\pi/2$. 
\end{enumerate}
\end{figure}

\begin{lemma}[{see Figure \ref{fig:shooting2}\;(3)}]\label{lem:20210829-green}
Let $n\in\{2,3\}$ and $r_0=2\pi/3$. 
There exists $\alpha_0\in\interval{0,\pi/2}$ such that the solution 
$(r,\vartheta,\alpha)\colon\Interval{0,s_*}\to D$ with initial data $(r_0,0,\alpha_0)$ satisfies $\vartheta(s)>0$ for all $s\in\interval{0,s_*}$ and 
\begin{align*}
\lim_{s\to s_*}(r,\vartheta,\alpha)(s)=(r_1,0,-\tfrac{\pi}{2})
\end{align*}
for some $r_1\in\interval{r_0,\pi}$. 
\end{lemma}

\begin{proof}
On the one hand, Lemma \ref{lem:small_alpha} implies that if $\alpha_0\in\interval{0,\pi/2}$ is chosen sufficiently small, then there exists $\sigma_1\in\interval{0,s_*}$ such that $\vartheta(\sigma_1)=0$ and $\vartheta(s)>0$ for all $s\in\interval{0,\sigma_1}$.  
On the other hand, for any choice of $\alpha_0\in\interval{0,\pi/2}$ which is sufficiently close to $\pi/2$ Lemma \ref{lem:20210828} implies $\vartheta(s)>0$ for all $s\in\interval{0,s_*}$. 
Therefore, there exists $\tilde\alpha_0\in\interval{0,\pi/2}$ which is largest with the property that for all $\alpha_0\in\interval{0,\tilde\alpha_0}$ the corresponding trajectory intersects $\{\vartheta=0\}$ for some $s=\sigma_1>0$.  

As shown in the proof of Lemma \ref{lem:zeros}, such a trajectory has a unique $s_2\in\interval{0,\sigma_1}$ where $\vartheta$ attains a local maximum and $\alpha(s_2)=0$ vanishes. 
We claim that $\vartheta(s_2)$ stays a positive distance away from $\{\vartheta=\pi/4\}$ for any choice of $\alpha_0\in\interval{0,\tilde\alpha_0}$. 
Indeed, equations \eqref{eqn:dt/ds2} and \eqref{eqn:da/ds2} imply 
$d\alpha/ds\leq-(2n-1)\cos(r)d\vartheta/ds$ 
in $[s_2,\sigma_1]$, where $\cos(r)\leq\cos(r_0)=-1/2$ by monotonicity of $r$ and our choice of $r_0=2\pi/3$.  
Hence, if the trajectory intersects $\{\vartheta=\pi/8\}$ at some $\tau\in\interval{s_2,\sigma_1}$, we obtain $\alpha(\tau)\geq(n-1/2)\pi/8-\pi/2$ by integrating the differential inequality over $[\tau,\sigma_1]$. 
In particular, $\cot(\alpha)<-1/2$ in $\interval{s_2,\tau}$ since $\alpha$ is strictly decreasing there. 
Equations \eqref{eqn:dt/ds2} and \eqref{eqn:da/ds2} then imply 
$d\alpha/ds\leq(n-1)\tan(2\vartheta)d\vartheta/ds$ in $\interval{s_2,\tau}$.  
Integrating this inequality over $\interval{s_2,\tau}$ then proves that $\vartheta(s_2)$ is uniformly bounded away from $\pi/4$.  
Let
\(
(\tilde{r},\tilde{\vartheta},\tilde{\alpha})\colon\Interval{0,\tilde{s}_*}\to D
\) 
be the solution to system \eqref{eqn:dr/ds2},\eqref{eqn:dt/ds2},\eqref{eqn:da/ds2} with initial data $(\tilde{r},\tilde{\vartheta},\tilde{\alpha})(0)=(2\pi/3,0,\tilde{\alpha}_0)$ as defined above. 
If there exists $\sigma_1\in\interval{0,\tilde{s}_*}$ such that $\tilde\vartheta(\sigma_1)=0$, then $\tilde\alpha(\sigma_1)\neq0$ because otherwise $\tilde\vartheta(s)$ and $\tilde\alpha(s)$ vanish identically. 
Then, let us recall that, by definition of $D$, we also have $\tilde\alpha(\sigma_1)>-\pi/2$. 
Therefore, $d\tilde\vartheta/ds(\sigma_1)\neq0$ and ${\sigma}_1$ is a nondegenerate zero of the map $s\mapsto\tilde\vartheta(s)$. 
By continuous dependence on the initial data there exists some $\tilde{\alpha}'_0>\tilde\alpha_0$ such that all solutions with $\alpha_0<\tilde{\alpha}'_0$ still intersect $\{\vartheta=0\}$. 
This contradicts our choice of $\tilde\alpha_0$.
Therefore, $\tilde\vartheta(s)>0$ for all $s\in\interval{0,\tilde{s}_*}$. 
Moreover, again by continuous dependence on the initial data, we have
\begin{align*} 
\lim_{s\to s_*}\tilde\vartheta(s)=0.
\end{align*}
Hence, $\tilde\vartheta$ attains an interior, positive local maximum at some $s_2\in\interval{0,\tilde{s}_*}$, where $\tilde\alpha(s_2)=0$.  
Moreover, as long as $\tilde\vartheta>0$ and $\tilde\alpha\leq0$ equation \eqref{eqn:da/ds2} implies $d\tilde\alpha/ds<0$. 
Therefore, $\tilde\alpha(s)$ is strictly decreasing in $\interval{s_2,\tilde{s}_*}$ and the limit 
\[
\lim_{s\to s_*}\tilde\alpha(s)=\tilde\alpha_*<0 
\]
exists. 
Lemma \ref{lem:well-posed_centre} \eqref{eqn:20210831-r<pi} then implies 
$\lim_{s\to s_*}\tilde{r}(s)<\pi$ and thus necessarily $\tilde\alpha_*=-\pi/2$ by Lemma~\ref{lem:well-posed_centre}~\eqref{eqn:20210831-r=pi}.   
\end{proof}

The following lemma differs from the previous one only in the fact that we assume $r_0=\pi/2$ rather than $r_0=2\pi/3$. 
However, the proof is very different because (as mentioned in Remark \ref{rem:20210829}) we cannot rely on Lemma \ref{lem:20210828} in the case $r_0=\pi/2$. 
Instead we aim for the reflection of the desired trajectory using the results from Section \ref{sec:embedded}. 

\begin{lemma}[{see Figure \ref{fig:shooting2}\;(6)}]\label{lem:20210829-blue}
Let $n\in\{2,3\}$. 
There exists $\alpha_\infty\in\interval{0,\pi/2}$ such that the solution 
$(r,\vartheta,\alpha)\colon\Interval{0,s_\infty}\to D$ with initial data $(r,\vartheta,\alpha)(0)=(\pi/2,0,\alpha_\infty)$ satisfies $\vartheta(s)>0$ for all $s\in\interval{0,s_\infty}$ and 
\begin{align*}
\lim_{s\to s_\infty}(r,\vartheta,\alpha)(s)=(r_\infty,0,-\tfrac{\pi}{2})
\end{align*}
for some $r_\infty\in\interval{\pi/2,\pi}$. 
\end{lemma}

\begin{proof}
As in Section \ref{sec:embedded}, we consider solutions of the system \eqref{eqn:dr/ds2},\eqref{eqn:dt/ds2},\eqref{eqn:da/ds2} with initial data 
\begin{align}\label{eqn:20210829-initial}
r(0)&=r_0\in\interval{0,\tfrac{\pi}{2}},& 
\vartheta(0)&=0,&
\alpha(0)&=-\tfrac{\pi}{2}
\end{align}
(recalling $\vartheta=\theta-\pi/4$) but now we extend the domain $\intervaL{0,\tfrac{\pi}{2}}
\times\intervaL{-\tfrac{\pi}{4},0}
\times[-\tfrac{\pi}{2},0]$ we had employed there to 
\begin{align*}
\tilde{B}\vcentcolon=\intervaL{0,\tfrac{\pi}{2}}
\times\intervaL{-\tfrac{\pi}{4},0}
\times[-\tfrac{\pi}{2},\tfrac{\pi}{2}].
\end{align*}
We recall that the motion is weakly monotone in each of the three coordinates as long as $\alpha\leq0$. 
If $\alpha$ becomes positive then we might lose its monotonicity but we still have $dr/ds\geq 0$. 
As long as $\vartheta<0$ and $r<\pi/2$, equation~\eqref{eqn:da/ds2} ensures the implications 
\begin{align*}
\alpha=0&~\Rightarrow~\frac{d\alpha}{ds}>0, &
\alpha=\frac{\pi}{2}&~\Rightarrow~\frac{d\alpha}{ds}<0.
\end{align*}
This implies that if $\alpha(s')\geq0$ for some $s'\in\interval{0,s_*}$ then $\alpha(s)\in\interval{0,\pi/2}$ for all $s\in\interval{s',s_*}$ and, in addition, $\limsup_{s\to s_*}\alpha(s)<\pi/2$. Furthermore,  
by Lemma~\ref{lem:well-posed} we also have 
$\vartheta(s_*)>-\pi/4$
as long as we stay in the regime $\alpha\leq0$ for the whole lifespan of the solution; 
if instead $\alpha$ ever becomes positive, then the inequality $d\vartheta/ds\geq0$ ensures the inversion of the motion (so that in particular the trajectory will stay a positive distance away from $\left\{\vartheta=-\pi/4\right\}$).
So, to summarize, the trajectory can exit $\tilde{B}$ only at $r=\pi/2$ or at $\vartheta=0$.
On top of that, keeping in mind that the trajectory satisfies the constraint \eqref{eqn:constraint}, since $r$ is weakly increasing and $d\vartheta/ds$ can change its sign at most once, the trajectory exits $\tilde{B}$ in finite time, i.\,e. here we have $s_*<\infty$ as in Lemma \ref{lem:well-posed}.

On the one hand, Lemma \ref{lem:side} states that if $r_0\in\interval{0,\pi/2}$ is chosen sufficiently close to $\pi/2$ then the trajectory exits the domain $\tilde{B}$ at $r(s_*)=\pi/2$. 
On the other hand, 
we may consider a central reflection of the trajectory constructed in Lemma \ref{lem:20210829-green} 
with respect to $r=\pi/2, \vartheta=0, \alpha=0$ and suitably reparametrise to obtain a solution 
$(\tilde{r},\tilde{\vartheta},\tilde{\alpha})\colon\Interval{0,s_*}\to\tilde{B}$ with initial data 
\begin{align*}
(\tilde{r},\tilde{\vartheta},\tilde{\alpha})(0)=(\pi-r_1,0,-\tfrac{\pi}{2})
\end{align*}
where $r_1\in\interval{\pi/2,\pi}$ is as in Lemma \ref{lem:20210829-green} (see also Figure \ref{fig:shooting2}) 
which exits the domain $\tilde{B}$ through $\{\vartheta=0\}$. 
Hence, arguing as in the proof of Theorem \ref{thm:main} there exists some $r_0\in\interval{0,\pi/2}$ such that the corresponding solution with initial data \eqref{eqn:20210829-initial} satisfies $r(s_*)=\pi/2$ and $\vartheta(s_*)=0$. Reflecting back and setting $r_\infty=\pi-r_0$ proves the claim. 
\end{proof}

\subsection{The Iteration process}

\begin{lemma}\label{lem:iteration}
Let $n\in\{2,3\}$. 
For all $1\leq k\in\N$ there exist $0<\hat\alpha_0^{(k)}<\check\alpha_0^{(k)}<\pi/2$ such that $\check\alpha_0^{(k+1)}<\hat\alpha_0^{(k)}$ for all $k$ and such that 
the solutions
\begin{align*}
\bigl(\hat{r}^{(k)},\hat{\vartheta}^{(k)},\hat{\alpha}^{(k)}\bigr)\colon\Interval{0,\hat{s}_*^{(k)}}&\to D, &
\bigl(\check{r}^{(k)},\check{\vartheta}^{(k)},\check{\alpha}^{(k)}\bigr)\colon\Interval{0,\check{s}_*^{(k)}}&\to D 
\intertext{with initial data $\bigl(\pi/2,0,\hat\alpha_0^{(k)}\bigr)$ respectively $\bigl(\pi/2,0,\check\alpha_0^{(k)}\bigr)$ as constructed in Lemma~\ref{lem:well-posed_centre} satisfy }
\lim_{s\to\hat{s}_*^{(k)}}\bigl(\hat{\vartheta}^{(k)},\hat{\alpha}^{(k)}\bigr)(s)&=(-1)^{k}\Bigl(\frac{\pi}{4},\frac{\pi}{2}\Bigr), & 
\lim_{s\to\check{s}_*^{(k)}}\bigl(\check{\vartheta}^{(k)},\check{\alpha}^{(k)}\bigr)(s)&=(-1)^{k}\Bigl(0,\frac{\pi}{2}\Bigr)
\end{align*}
and such that the closures of their images intersect the plane $\{\vartheta=0\}$ exactly $k+1$ times.
\end{lemma} 

\begin{figure}\centering
\begin{tikzpicture}[line cap=round,line join=round,baseline={(0,0)},scale=\textwidth/3.5cm]
\draw[very thin](0,0)grid [xstep=pi/2,ystep=pi/4](pi,pi/2);
\draw[-latex](0,0)--(pi,0)node[below left]{$r$};
\draw[-latex](0,0)--(0,pi/2)node[below left]{$\theta$};
\draw plot[plus](0,0)node[below]{$0$};
\draw plot[plus](pi/2,0)node[below]{$\frac{\pi}{2}$};
\draw plot[plus](0,pi/4)node[left]{$\frac{\pi}{4}$};
\draw[smooth,spheretrajectory,yscale=-1,shift={(0,-pi/2)}]plot coordinates {(1.5708,0.7854) (1.5782,0.7848) (1.5855,0.7843) (1.5929,0.7837) (1.6003,0.7832) (1.6077,0.7826) (1.6151,0.7821) (1.6224,0.7815) (1.6298,0.7810) (1.6372,0.7804) (1.6446,0.7799) (1.6519,0.7793) (1.6593,0.7788) (1.6667,0.7782) (1.6741,0.7777) (1.6814,0.7771) (1.6888,0.7766) (1.6962,0.7760) (1.7036,0.7755) (1.7110,0.7749) (1.7183,0.7744) (1.7257,0.7738) (1.7331,0.7733) (1.7405,0.7727) (1.7478,0.7722) (1.7552,0.7716) (1.7626,0.7711) (1.7700,0.7705) (1.7773,0.7700) (1.7847,0.7694) (1.7921,0.7689) (1.7995,0.7683) (1.8069,0.7678) (1.8142,0.7672) (1.8216,0.7667) (1.8290,0.7661) (1.8364,0.7656) (1.8437,0.7650) (1.8511,0.7645) (1.8585,0.7639) (1.8659,0.7634) (1.8733,0.7628) (1.8806,0.7623) (1.8880,0.7617) (1.8954,0.7612) (1.9028,0.7606) (1.9101,0.7601) (1.9175,0.7595) (1.9249,0.7590) (1.9323,0.7584) (1.9397,0.7579) (1.9470,0.7574) (1.9544,0.7568) (1.9618,0.7563) (1.9692,0.7557) (1.9766,0.7552) (1.9839,0.7547) (1.9913,0.7541) (1.9987,0.7536) (2.0061,0.7530) (2.0135,0.7525) (2.0209,0.7520) (2.0282,0.7515) (2.0356,0.7509) (2.0430,0.7504) (2.0504,0.7499) (2.0578,0.7493) (2.0651,0.7488) (2.0725,0.7483) (2.0799,0.7478) (2.0873,0.7473) (2.0947,0.7467) (2.1021,0.7462) (2.1094,0.7457) (2.1168,0.7452) (2.1242,0.7447) (2.1316,0.7442) (2.1390,0.7437) (2.1464,0.7432) (2.1538,0.7427) (2.1611,0.7422) (2.1685,0.7418) (2.1759,0.7413) (2.1833,0.7408) (2.1907,0.7403) (2.1981,0.7399) (2.2055,0.7394) (2.2129,0.7390) (2.2202,0.7385) (2.2276,0.7381) (2.2350,0.7377) (2.2424,0.7372) (2.2498,0.7368) (2.2572,0.7364) (2.2646,0.7360) (2.2720,0.7356) (2.2794,0.7352) (2.2867,0.7348) (2.2941,0.7345) (2.3015,0.7341) (2.3089,0.7338) (2.3163,0.7334) (2.3237,0.7331) (2.3311,0.7328) (2.3385,0.7325) (2.3459,0.7322) (2.3533,0.7320) (2.3607,0.7317) (2.3681,0.7315) (2.3755,0.7313) (2.3829,0.7311) (2.3903,0.7309) (2.3977,0.7307) (2.4051,0.7306) (2.4125,0.7305) (2.4199,0.7304) (2.4272,0.7303) (2.4346,0.7303) (2.4420,0.7303) (2.4494,0.7303) (2.4568,0.7303) (2.4642,0.7304) (2.4716,0.7305) (2.4790,0.7307) (2.4864,0.7309) (2.4938,0.7311) (2.5012,0.7314) (2.5086,0.7317) (2.5160,0.7321) (2.5234,0.7325) (2.5308,0.7330) (2.5382,0.7335) (2.5456,0.7341) (2.5530,0.7347) (2.5603,0.7354) (2.5677,0.7362) (2.5751,0.7371) (2.5825,0.7380) (2.5899,0.7390) (2.5972,0.7401) (2.6046,0.7413) (2.6120,0.7426) (2.6193,0.7440) (2.6267,0.7455) (2.6341,0.7472) (2.6414,0.7489) (2.6487,0.7508) (2.6561,0.7528) (2.6634,0.7550) (2.6707,0.7574) (2.6780,0.7599) (2.6853,0.7625) (2.6926,0.7654) (2.6999,0.7685) (2.7072,0.7718) (2.7144,0.7753) (2.7217,0.7790) (2.7289,0.7831) (2.7361,0.7873) (2.7432,0.7919) (2.7504,0.7968) (2.7575,0.8021) (2.7646,0.8077) (2.7717,0.8137) (2.7787,0.8201) (2.7857,0.8269) (2.7927,0.8342) (2.7996,0.8419) (2.8065,0.8502) (2.8133,0.8590) (2.8200,0.8685) (2.8267,0.8786) (2.8333,0.8894) (2.8399,0.9009) (2.8463,0.9132) (2.8527,0.9263) (2.8590,0.9403) (2.8651,0.9551) (2.8712,0.9710) (2.8771,0.9879) (2.8828,1.0058) (2.8884,1.0248) (2.8939,1.0450) (2.8992,1.0664) (2.9042,1.0890) (2.9091,1.1130) (2.9137,1.1383) (2.9181,1.1649) (2.9222,1.1929) (2.9260,1.2222) (2.9296,1.2528) (2.9328,1.2847) (2.9357,1.3178) (2.9382,1.3520) (2.9404,1.3872) (2.9422,1.4233) (2.9436,1.4601) (2.9446,1.4974) (2.9452,1.5351) (2.9453,1.5686) };
\draw[smooth,spheretrajectory,xscale=-1,shift={(-pi,0)},dotted]plot coordinates {(1.5708,0.7854) (1.5782,0.7848) (1.5855,0.7843) (1.5929,0.7837) (1.6003,0.7832) (1.6077,0.7826) (1.6151,0.7821) (1.6224,0.7815) (1.6298,0.7810) (1.6372,0.7804) (1.6446,0.7799) (1.6519,0.7793) (1.6593,0.7788) (1.6667,0.7782) (1.6741,0.7777) (1.6814,0.7771) (1.6888,0.7766) (1.6962,0.7760) (1.7036,0.7755) (1.7110,0.7749) (1.7183,0.7744) (1.7257,0.7738) (1.7331,0.7733) (1.7405,0.7727) (1.7478,0.7722) (1.7552,0.7716) (1.7626,0.7711) (1.7700,0.7705) (1.7773,0.7700) (1.7847,0.7694) (1.7921,0.7689) (1.7995,0.7683) (1.8069,0.7678) (1.8142,0.7672) (1.8216,0.7667) (1.8290,0.7661) (1.8364,0.7656) (1.8437,0.7650) (1.8511,0.7645) (1.8585,0.7639) (1.8659,0.7634) (1.8733,0.7628) (1.8806,0.7623) (1.8880,0.7617) (1.8954,0.7612) (1.9028,0.7606) (1.9101,0.7601) (1.9175,0.7595) (1.9249,0.7590) (1.9323,0.7584) (1.9397,0.7579) (1.9470,0.7574) (1.9544,0.7568) (1.9618,0.7563) (1.9692,0.7557) (1.9766,0.7552) (1.9839,0.7547) (1.9913,0.7541) (1.9987,0.7536) (2.0061,0.7530) (2.0135,0.7525) (2.0209,0.7520) (2.0282,0.7515) (2.0356,0.7509) (2.0430,0.7504) (2.0504,0.7499) (2.0578,0.7493) (2.0651,0.7488) (2.0725,0.7483) (2.0799,0.7478) (2.0873,0.7473) (2.0947,0.7467) (2.1021,0.7462) (2.1094,0.7457) (2.1168,0.7452) (2.1242,0.7447) (2.1316,0.7442) (2.1390,0.7437) (2.1464,0.7432) (2.1538,0.7427) (2.1611,0.7422) (2.1685,0.7418) (2.1759,0.7413) (2.1833,0.7408) (2.1907,0.7403) (2.1981,0.7399) (2.2055,0.7394) (2.2129,0.7390) (2.2202,0.7385) (2.2276,0.7381) (2.2350,0.7377) (2.2424,0.7372) (2.2498,0.7368) (2.2572,0.7364) (2.2646,0.7360) (2.2720,0.7356) (2.2794,0.7352) (2.2867,0.7348) (2.2941,0.7345) (2.3015,0.7341) (2.3089,0.7338) (2.3163,0.7334) (2.3237,0.7331) (2.3311,0.7328) (2.3385,0.7325) (2.3459,0.7322) (2.3533,0.7320) (2.3607,0.7317) (2.3681,0.7315) (2.3755,0.7313) (2.3829,0.7311) (2.3903,0.7309) (2.3977,0.7307) (2.4051,0.7306) (2.4125,0.7305) (2.4199,0.7304) (2.4272,0.7303) (2.4346,0.7303) (2.4420,0.7303) (2.4494,0.7303) (2.4568,0.7303) (2.4642,0.7304) (2.4716,0.7305) (2.4790,0.7307) (2.4864,0.7309) (2.4938,0.7311) (2.5012,0.7314) (2.5086,0.7317) (2.5160,0.7321) (2.5234,0.7325) (2.5308,0.7330) (2.5382,0.7335) (2.5456,0.7341) (2.5530,0.7347) (2.5603,0.7354) (2.5677,0.7362) (2.5751,0.7371) (2.5825,0.7380) (2.5899,0.7390) (2.5972,0.7401) (2.6046,0.7413) (2.6120,0.7426) (2.6193,0.7440) (2.6267,0.7455) (2.6341,0.7472) (2.6414,0.7489) (2.6487,0.7508) (2.6561,0.7528) (2.6634,0.7550) (2.6707,0.7574) (2.6780,0.7599) (2.6853,0.7625) (2.6926,0.7654) (2.6999,0.7685) (2.7072,0.7718) (2.7144,0.7753) (2.7217,0.7790) (2.7289,0.7831) (2.7361,0.7873) (2.7432,0.7919) (2.7504,0.7968) (2.7575,0.8021) (2.7646,0.8077) (2.7717,0.8137) (2.7787,0.8201) (2.7857,0.8269) (2.7927,0.8342) (2.7996,0.8419) (2.8065,0.8502) (2.8133,0.8590) (2.8200,0.8685) (2.8267,0.8786) (2.8333,0.8894) (2.8399,0.9009) (2.8463,0.9132) (2.8527,0.9263) (2.8590,0.9403) (2.8651,0.9551) (2.8712,0.9710) (2.8771,0.9879) (2.8828,1.0058) (2.8884,1.0248) (2.8939,1.0450) (2.8992,1.0664) (2.9042,1.0890) (2.9091,1.1130) (2.9137,1.1383) (2.9181,1.1649) (2.9222,1.1929) (2.9260,1.2222) (2.9296,1.2528) (2.9328,1.2847) (2.9357,1.3178) (2.9382,1.3520) (2.9404,1.3872) (2.9422,1.4233) (2.9436,1.4601) (2.9446,1.4974) (2.9452,1.5351) (2.9453,1.5686) };
\draw[smooth,torustrajectory]plot coordinates {  
(1.5708,0.7854) (2.1050,0.7957) (2.1844,0.7971) (2.2362,0.7979) (2.2756,0.7985) (2.3078,0.7989) (2.3352,0.7992) (2.3591,0.7994) (2.3805,0.7996) (2.3998,0.7997) (2.4174,0.7998) (2.4337,0.7998) (2.4489,0.7998) (2.4630,0.7998) (2.4764,0.7997) (2.4890,0.7997) (2.5009,0.7996) (2.5122,0.7994) (2.5231,0.7993) (2.5334,0.7991) (2.5433,0.7989) (2.5529,0.7987) (2.5621,0.7985) (2.5709,0.7982) (2.5795,0.7980) (2.5877,0.7977) (2.5957,0.7974) (2.6035,0.7971) (2.6110,0.7967) (2.6184,0.7964) (2.6255,0.7960) (2.6324,0.7956) (2.6392,0.7952) (2.6458,0.7948) (2.6522,0.7943) (2.6585,0.7939) (2.6646,0.7934) (2.6706,0.7929) (2.6765,0.7924) (2.6823,0.7919) (2.6879,0.7913) (2.6934,0.7907) (2.6989,0.7902) (2.7042,0.7896) (2.7094,0.7889) (2.7145,0.7883) (2.7195,0.7876) (2.7245,0.7870) (2.7294,0.7863) (2.7341,0.7856) (2.7388,0.7848) (2.7435,0.7841) (2.7480,0.7833) (2.7525,0.7825) (2.7569,0.7817) (2.7613,0.7808) (2.7656,0.7799) (2.7698,0.7791) (2.7739,0.7781) (2.7781,0.7772) (2.7821,0.7763) (2.7861,0.7753) (2.7900,0.7743) (2.7939,0.7732) (2.7978,0.7722) (2.8016,0.7711) (2.8053,0.7700) (2.8090,0.7688) (2.8127,0.7677) (2.8163,0.7665) (2.8199,0.7653) (2.8234,0.7640) (2.8269,0.7627) (2.8303,0.7614) (2.8337,0.7601) (2.8371,0.7587) (2.8404,0.7573) (2.8437,0.7559) (2.8470,0.7544) (2.8502,0.7530) (2.8534,0.7514) (2.8566,0.7499) (2.8597,0.7483) (2.8628,0.7467) (2.8659,0.7450) (2.8689,0.7434) (2.8719,0.7416) (2.8749,0.7399) (2.8779,0.7381) (2.8808,0.7362) (2.8837,0.7344) (2.8865,0.7325) (2.8894,0.7305) (2.8922,0.7285) (2.8950,0.7265) (2.8977,0.7244) (2.9005,0.7223) (2.9032,0.7201) (2.9059,0.7179) (2.9085,0.7156) (2.9112,0.7133) (2.9138,0.7109) (2.9164,0.7085) (2.9190,0.7061) (2.9215,0.7035) (2.9241,0.7010) (2.9266,0.6984) (2.9291,0.6957) (2.9315,0.6930) (2.9340,0.6902) (2.9364,0.6873) (2.9388,0.6845) (2.9412,0.6815) (2.9436,0.6785) (2.9459,0.6754) (2.9482,0.6723) (2.9506,0.6691) (2.9529,0.6659) (2.9551,0.6626) (2.9574,0.6592) (2.9596,0.6558) (2.9619,0.6523) (2.9641,0.6488) (2.9663,0.6451) (2.9684,0.6414) (2.9706,0.6377) (2.9727,0.6339) (2.9749,0.6300) (2.9770,0.6260) (2.9791,0.6220) (2.9812,0.6179) (2.9832,0.6137) (2.9853,0.6095) (2.9873,0.6052) (2.9893,0.6008) (2.9913,0.5963) (2.9933,0.5918) (2.9953,0.5872) (2.9973,0.5826) (2.9992,0.5779) (3.0012,0.5731) (3.0031,0.5683) (3.0050,0.5634) (3.0069,0.5584) (3.0088,0.5534) (3.0107,0.5484) (3.0126,0.5433) (3.0145,0.5382) (3.0163,0.5330) (3.0182,0.5278) (3.0200,0.5225) (3.0218,0.5172) (3.0236,0.5118) (3.0255,0.5065) (3.0273,0.5011) (3.0291,0.4958) (3.0308,0.4905) (3.0326,0.4852) (3.0344,0.4800) (3.0362,0.4749) (3.0380,0.4699) (3.0397,0.4650) (3.0415,0.4603) (3.0433,0.4557) (3.0450,0.4513) (3.0468,0.4472) (3.0485,0.4433) (3.0503,0.4397) (3.0520,0.4365) (3.0538,0.4338) (3.0555,0.4315) (3.0573,0.4298) (3.0590,0.4286) (3.0608,0.4283) (3.0625,0.4287) (3.0642,0.4299) (3.0659,0.4322) (3.0676,0.4356) (3.0693,0.4401) (3.0709,0.4460) (3.0725,0.4533) (3.0741,0.4621) (3.0756,0.4725) (3.0770,0.4846) (3.0784,0.4985) (3.0798,0.5142) (3.0810,0.5318) (3.0821,0.5511) (3.0832,0.5722) (3.0841,0.5949) (3.0849,0.6191) (3.0856,0.6447) (3.0862,0.6714) (3.0866,0.6991) (3.0870,0.7274) (3.0871,0.7561) (3.0872,0.7850) 
}; 
\draw[smooth,torustrajectory,shift={(pi,0)},xscale=-1,dotted]plot coordinates {  
(1.5708,0.7854) (2.1050,0.7957) (2.1844,0.7971) (2.2362,0.7979) (2.2756,0.7985) (2.3078,0.7989) (2.3352,0.7992) (2.3591,0.7994) (2.3805,0.7996) (2.3998,0.7997) (2.4174,0.7998) (2.4337,0.7998) (2.4489,0.7998) (2.4630,0.7998) (2.4764,0.7997) (2.4890,0.7997) (2.5009,0.7996) (2.5122,0.7994) (2.5231,0.7993) (2.5334,0.7991) (2.5433,0.7989) (2.5529,0.7987) (2.5621,0.7985) (2.5709,0.7982) (2.5795,0.7980) (2.5877,0.7977) (2.5957,0.7974) (2.6035,0.7971) (2.6110,0.7967) (2.6184,0.7964) (2.6255,0.7960) (2.6324,0.7956) (2.6392,0.7952) (2.6458,0.7948) (2.6522,0.7943) (2.6585,0.7939) (2.6646,0.7934) (2.6706,0.7929) (2.6765,0.7924) (2.6823,0.7919) (2.6879,0.7913) (2.6934,0.7907) (2.6989,0.7902) (2.7042,0.7896) (2.7094,0.7889) (2.7145,0.7883) (2.7195,0.7876) (2.7245,0.7870) (2.7294,0.7863) (2.7341,0.7856) (2.7388,0.7848) (2.7435,0.7841) (2.7480,0.7833) (2.7525,0.7825) (2.7569,0.7817) (2.7613,0.7808) (2.7656,0.7799) (2.7698,0.7791) (2.7739,0.7781) (2.7781,0.7772) (2.7821,0.7763) (2.7861,0.7753) (2.7900,0.7743) (2.7939,0.7732) (2.7978,0.7722) (2.8016,0.7711) (2.8053,0.7700) (2.8090,0.7688) (2.8127,0.7677) (2.8163,0.7665) (2.8199,0.7653) (2.8234,0.7640) (2.8269,0.7627) (2.8303,0.7614) (2.8337,0.7601) (2.8371,0.7587) (2.8404,0.7573) (2.8437,0.7559) (2.8470,0.7544) (2.8502,0.7530) (2.8534,0.7514) (2.8566,0.7499) (2.8597,0.7483) (2.8628,0.7467) (2.8659,0.7450) (2.8689,0.7434) (2.8719,0.7416) (2.8749,0.7399) (2.8779,0.7381) (2.8808,0.7362) (2.8837,0.7344) (2.8865,0.7325) (2.8894,0.7305) (2.8922,0.7285) (2.8950,0.7265) (2.8977,0.7244) (2.9005,0.7223) (2.9032,0.7201) (2.9059,0.7179) (2.9085,0.7156) (2.9112,0.7133) (2.9138,0.7109) (2.9164,0.7085) (2.9190,0.7061) (2.9215,0.7035) (2.9241,0.7010) (2.9266,0.6984) (2.9291,0.6957) (2.9315,0.6930) (2.9340,0.6902) (2.9364,0.6873) (2.9388,0.6845) (2.9412,0.6815) (2.9436,0.6785) (2.9459,0.6754) (2.9482,0.6723) (2.9506,0.6691) (2.9529,0.6659) (2.9551,0.6626) (2.9574,0.6592) (2.9596,0.6558) (2.9619,0.6523) (2.9641,0.6488) (2.9663,0.6451) (2.9684,0.6414) (2.9706,0.6377) (2.9727,0.6339) (2.9749,0.6300) (2.9770,0.6260) (2.9791,0.6220) (2.9812,0.6179) (2.9832,0.6137) (2.9853,0.6095) (2.9873,0.6052) (2.9893,0.6008) (2.9913,0.5963) (2.9933,0.5918) (2.9953,0.5872) (2.9973,0.5826) (2.9992,0.5779) (3.0012,0.5731) (3.0031,0.5683) (3.0050,0.5634) (3.0069,0.5584) (3.0088,0.5534) (3.0107,0.5484) (3.0126,0.5433) (3.0145,0.5382) (3.0163,0.5330) (3.0182,0.5278) (3.0200,0.5225) (3.0218,0.5172) (3.0236,0.5118) (3.0255,0.5065) (3.0273,0.5011) (3.0291,0.4958) (3.0308,0.4905) (3.0326,0.4852) (3.0344,0.4800) (3.0362,0.4749) (3.0380,0.4699) (3.0397,0.4650) (3.0415,0.4603) (3.0433,0.4557) (3.0450,0.4513) (3.0468,0.4472) (3.0485,0.4433) (3.0503,0.4397) (3.0520,0.4365) (3.0538,0.4338) (3.0555,0.4315) (3.0573,0.4298) (3.0590,0.4286) (3.0608,0.4283) (3.0625,0.4287) (3.0642,0.4299) (3.0659,0.4322) (3.0676,0.4356) (3.0693,0.4401) (3.0709,0.4460) (3.0725,0.4533) (3.0741,0.4621) (3.0756,0.4725) (3.0770,0.4846) (3.0784,0.4985) (3.0798,0.5142) (3.0810,0.5318) (3.0821,0.5511) (3.0832,0.5722) (3.0841,0.5949) (3.0849,0.6191) (3.0856,0.6447) (3.0862,0.6714) (3.0866,0.6991) (3.0870,0.7274) (3.0871,0.7561) (3.0872,0.7850) 
}; 
\draw[smooth,torustrajectory,shift={(0,pi/2)},yscale=-1,dotted]plot coordinates {  
(1.5708,0.7854) (2.1050,0.7957) (2.1844,0.7971) (2.2362,0.7979) (2.2756,0.7985) (2.3078,0.7989) (2.3352,0.7992) (2.3591,0.7994) (2.3805,0.7996) (2.3998,0.7997) (2.4174,0.7998) (2.4337,0.7998) (2.4489,0.7998) (2.4630,0.7998) (2.4764,0.7997) (2.4890,0.7997) (2.5009,0.7996) (2.5122,0.7994) (2.5231,0.7993) (2.5334,0.7991) (2.5433,0.7989) (2.5529,0.7987) (2.5621,0.7985) (2.5709,0.7982) (2.5795,0.7980) (2.5877,0.7977) (2.5957,0.7974) (2.6035,0.7971) (2.6110,0.7967) (2.6184,0.7964) (2.6255,0.7960) (2.6324,0.7956) (2.6392,0.7952) (2.6458,0.7948) (2.6522,0.7943) (2.6585,0.7939) (2.6646,0.7934) (2.6706,0.7929) (2.6765,0.7924) (2.6823,0.7919) (2.6879,0.7913) (2.6934,0.7907) (2.6989,0.7902) (2.7042,0.7896) (2.7094,0.7889) (2.7145,0.7883) (2.7195,0.7876) (2.7245,0.7870) (2.7294,0.7863) (2.7341,0.7856) (2.7388,0.7848) (2.7435,0.7841) (2.7480,0.7833) (2.7525,0.7825) (2.7569,0.7817) (2.7613,0.7808) (2.7656,0.7799) (2.7698,0.7791) (2.7739,0.7781) (2.7781,0.7772) (2.7821,0.7763) (2.7861,0.7753) (2.7900,0.7743) (2.7939,0.7732) (2.7978,0.7722) (2.8016,0.7711) (2.8053,0.7700) (2.8090,0.7688) (2.8127,0.7677) (2.8163,0.7665) (2.8199,0.7653) (2.8234,0.7640) (2.8269,0.7627) (2.8303,0.7614) (2.8337,0.7601) (2.8371,0.7587) (2.8404,0.7573) (2.8437,0.7559) (2.8470,0.7544) (2.8502,0.7530) (2.8534,0.7514) (2.8566,0.7499) (2.8597,0.7483) (2.8628,0.7467) (2.8659,0.7450) (2.8689,0.7434) (2.8719,0.7416) (2.8749,0.7399) (2.8779,0.7381) (2.8808,0.7362) (2.8837,0.7344) (2.8865,0.7325) (2.8894,0.7305) (2.8922,0.7285) (2.8950,0.7265) (2.8977,0.7244) (2.9005,0.7223) (2.9032,0.7201) (2.9059,0.7179) (2.9085,0.7156) (2.9112,0.7133) (2.9138,0.7109) (2.9164,0.7085) (2.9190,0.7061) (2.9215,0.7035) (2.9241,0.7010) (2.9266,0.6984) (2.9291,0.6957) (2.9315,0.6930) (2.9340,0.6902) (2.9364,0.6873) (2.9388,0.6845) (2.9412,0.6815) (2.9436,0.6785) (2.9459,0.6754) (2.9482,0.6723) (2.9506,0.6691) (2.9529,0.6659) (2.9551,0.6626) (2.9574,0.6592) (2.9596,0.6558) (2.9619,0.6523) (2.9641,0.6488) (2.9663,0.6451) (2.9684,0.6414) (2.9706,0.6377) (2.9727,0.6339) (2.9749,0.6300) (2.9770,0.6260) (2.9791,0.6220) (2.9812,0.6179) (2.9832,0.6137) (2.9853,0.6095) (2.9873,0.6052) (2.9893,0.6008) (2.9913,0.5963) (2.9933,0.5918) (2.9953,0.5872) (2.9973,0.5826) (2.9992,0.5779) (3.0012,0.5731) (3.0031,0.5683) (3.0050,0.5634) (3.0069,0.5584) (3.0088,0.5534) (3.0107,0.5484) (3.0126,0.5433) (3.0145,0.5382) (3.0163,0.5330) (3.0182,0.5278) (3.0200,0.5225) (3.0218,0.5172) (3.0236,0.5118) (3.0255,0.5065) (3.0273,0.5011) (3.0291,0.4958) (3.0308,0.4905) (3.0326,0.4852) (3.0344,0.4800) (3.0362,0.4749) (3.0380,0.4699) (3.0397,0.4650) (3.0415,0.4603) (3.0433,0.4557) (3.0450,0.4513) (3.0468,0.4472) (3.0485,0.4433) (3.0503,0.4397) (3.0520,0.4365) (3.0538,0.4338) (3.0555,0.4315) (3.0573,0.4298) (3.0590,0.4286) (3.0608,0.4283) (3.0625,0.4287) (3.0642,0.4299) (3.0659,0.4322) (3.0676,0.4356) (3.0693,0.4401) (3.0709,0.4460) (3.0725,0.4533) (3.0741,0.4621) (3.0756,0.4725) (3.0770,0.4846) (3.0784,0.4985) (3.0798,0.5142) (3.0810,0.5318) (3.0821,0.5511) (3.0832,0.5722) (3.0841,0.5949) (3.0849,0.6191) (3.0856,0.6447) (3.0862,0.6714) (3.0866,0.6991) (3.0870,0.7274) (3.0871,0.7561) (3.0872,0.7850) 
}; 
\draw[smooth,torustrajectory,shift={(pi,pi/2)},scale=-1,dotted]plot coordinates {  
(1.5708,0.7854) (2.1050,0.7957) (2.1844,0.7971) (2.2362,0.7979) (2.2756,0.7985) (2.3078,0.7989) (2.3352,0.7992) (2.3591,0.7994) (2.3805,0.7996) (2.3998,0.7997) (2.4174,0.7998) (2.4337,0.7998) (2.4489,0.7998) (2.4630,0.7998) (2.4764,0.7997) (2.4890,0.7997) (2.5009,0.7996) (2.5122,0.7994) (2.5231,0.7993) (2.5334,0.7991) (2.5433,0.7989) (2.5529,0.7987) (2.5621,0.7985) (2.5709,0.7982) (2.5795,0.7980) (2.5877,0.7977) (2.5957,0.7974) (2.6035,0.7971) (2.6110,0.7967) (2.6184,0.7964) (2.6255,0.7960) (2.6324,0.7956) (2.6392,0.7952) (2.6458,0.7948) (2.6522,0.7943) (2.6585,0.7939) (2.6646,0.7934) (2.6706,0.7929) (2.6765,0.7924) (2.6823,0.7919) (2.6879,0.7913) (2.6934,0.7907) (2.6989,0.7902) (2.7042,0.7896) (2.7094,0.7889) (2.7145,0.7883) (2.7195,0.7876) (2.7245,0.7870) (2.7294,0.7863) (2.7341,0.7856) (2.7388,0.7848) (2.7435,0.7841) (2.7480,0.7833) (2.7525,0.7825) (2.7569,0.7817) (2.7613,0.7808) (2.7656,0.7799) (2.7698,0.7791) (2.7739,0.7781) (2.7781,0.7772) (2.7821,0.7763) (2.7861,0.7753) (2.7900,0.7743) (2.7939,0.7732) (2.7978,0.7722) (2.8016,0.7711) (2.8053,0.7700) (2.8090,0.7688) (2.8127,0.7677) (2.8163,0.7665) (2.8199,0.7653) (2.8234,0.7640) (2.8269,0.7627) (2.8303,0.7614) (2.8337,0.7601) (2.8371,0.7587) (2.8404,0.7573) (2.8437,0.7559) (2.8470,0.7544) (2.8502,0.7530) (2.8534,0.7514) (2.8566,0.7499) (2.8597,0.7483) (2.8628,0.7467) (2.8659,0.7450) (2.8689,0.7434) (2.8719,0.7416) (2.8749,0.7399) (2.8779,0.7381) (2.8808,0.7362) (2.8837,0.7344) (2.8865,0.7325) (2.8894,0.7305) (2.8922,0.7285) (2.8950,0.7265) (2.8977,0.7244) (2.9005,0.7223) (2.9032,0.7201) (2.9059,0.7179) (2.9085,0.7156) (2.9112,0.7133) (2.9138,0.7109) (2.9164,0.7085) (2.9190,0.7061) (2.9215,0.7035) (2.9241,0.7010) (2.9266,0.6984) (2.9291,0.6957) (2.9315,0.6930) (2.9340,0.6902) (2.9364,0.6873) (2.9388,0.6845) (2.9412,0.6815) (2.9436,0.6785) (2.9459,0.6754) (2.9482,0.6723) (2.9506,0.6691) (2.9529,0.6659) (2.9551,0.6626) (2.9574,0.6592) (2.9596,0.6558) (2.9619,0.6523) (2.9641,0.6488) (2.9663,0.6451) (2.9684,0.6414) (2.9706,0.6377) (2.9727,0.6339) (2.9749,0.6300) (2.9770,0.6260) (2.9791,0.6220) (2.9812,0.6179) (2.9832,0.6137) (2.9853,0.6095) (2.9873,0.6052) (2.9893,0.6008) (2.9913,0.5963) (2.9933,0.5918) (2.9953,0.5872) (2.9973,0.5826) (2.9992,0.5779) (3.0012,0.5731) (3.0031,0.5683) (3.0050,0.5634) (3.0069,0.5584) (3.0088,0.5534) (3.0107,0.5484) (3.0126,0.5433) (3.0145,0.5382) (3.0163,0.5330) (3.0182,0.5278) (3.0200,0.5225) (3.0218,0.5172) (3.0236,0.5118) (3.0255,0.5065) (3.0273,0.5011) (3.0291,0.4958) (3.0308,0.4905) (3.0326,0.4852) (3.0344,0.4800) (3.0362,0.4749) (3.0380,0.4699) (3.0397,0.4650) (3.0415,0.4603) (3.0433,0.4557) (3.0450,0.4513) (3.0468,0.4472) (3.0485,0.4433) (3.0503,0.4397) (3.0520,0.4365) (3.0538,0.4338) (3.0555,0.4315) (3.0573,0.4298) (3.0590,0.4286) (3.0608,0.4283) (3.0625,0.4287) (3.0642,0.4299) (3.0659,0.4322) (3.0676,0.4356) (3.0693,0.4401) (3.0709,0.4460) (3.0725,0.4533) (3.0741,0.4621) (3.0756,0.4725) (3.0770,0.4846) (3.0784,0.4985) (3.0798,0.5142) (3.0810,0.5318) (3.0821,0.5511) (3.0832,0.5722) (3.0841,0.5949) (3.0849,0.6191) (3.0856,0.6447) (3.0862,0.6714) (3.0866,0.6991) (3.0870,0.7274) (3.0871,0.7561) (3.0872,0.7850) 
}; 
\end{tikzpicture}
\caption{The trajectories of $\bigl(\hat{r}^{(1)},\hat{\vartheta}^{(1)}\bigr)$ (orange curve) and   $\bigl(\check{r}^{(2)},\check{\vartheta}^{(2)}\bigr)$ (blue curve) and their respective reflections (dotted curves), as described within the proof of Theorem~\ref{thm:mainImmersed} and Theorem~\ref{thm:mainHyperspheres}, in the case $n=2$. }%
\label{fig:immersed_torus2}%
\end{figure}
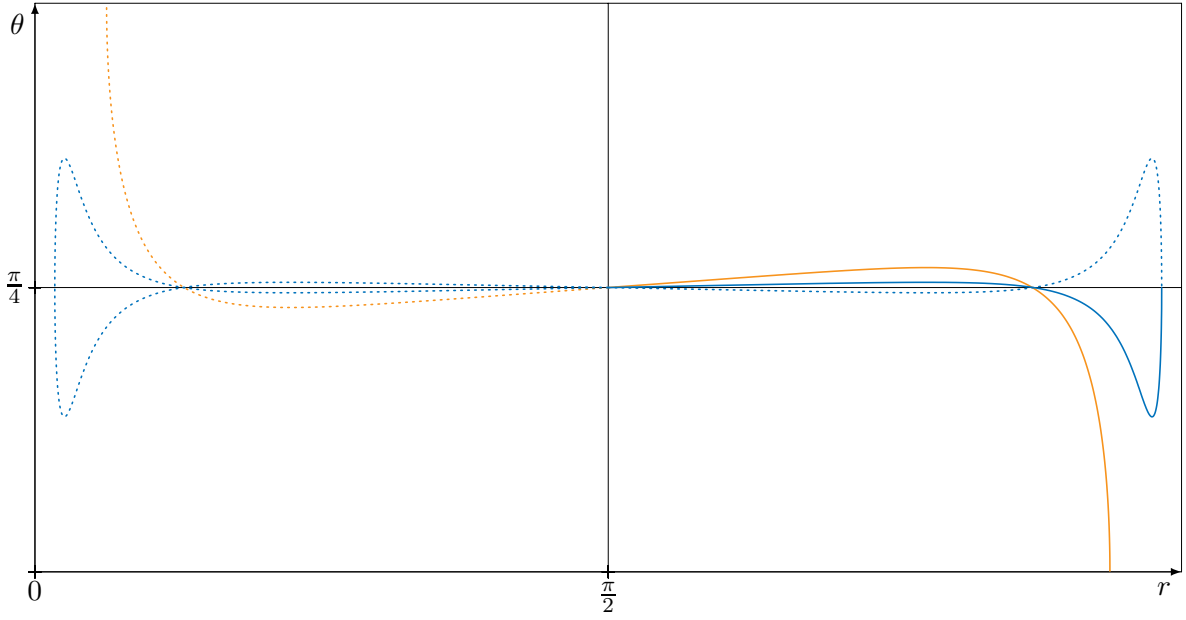 
 
As an example, we visualise the trajectories of $\bigl(\hat{r}^{(1)},\hat{\vartheta}^{(1)}\bigr)$ and  $\bigl(\check{r}^{(2)},\check{\vartheta}^{(2)}\bigr)$ for $n=2$ in Figure~\ref{fig:immersed_torus2}. 

\begin{proof}
\emph{Construction of $\check{\alpha}_0^{(1)}$.}
In Lemma \ref{lem:20210829-blue} we obtained the ``$\infty$-figure''  $(r,\vartheta,\alpha)\colon\Interval{0,s_\infty}\to D$, a solution of system \eqref{eqn:dr/ds2},\eqref{eqn:dt/ds2},\eqref{eqn:da/ds2} with initial data 
$(r,\vartheta,\alpha)(0)=(\pi/2,0,\alpha_\infty)$ for some $\alpha_\infty\in\interval{0,\pi/2}$. 
The trajectory of this solution exits the domain $D=\Interval{\tfrac{\pi}{2},\pi}\times\interval{-\tfrac{\pi}{4},\tfrac{\pi}{4}}\times\interval{-\tfrac{\pi}{2},\tfrac{\pi}{2}}$ via 
\[
\lim_{s\to s_\infty}(r,\vartheta,\alpha)(s)=(r_\infty,0,-\tfrac{\pi}{2})   
\]
for some $r_\infty\in\interval{\pi/2,\pi}$. 
Actually, we did not show that for any $\alpha_0\in\interval{0,\alpha_\infty}$ the corresponding trajectory will also intersect $\{\vartheta=0\}$ for $s>0$, but Lemma~\ref{lem:small_alpha} states that such intersections actually occur for any sufficiently small $\alpha_0$. 
Therefore, we may define $\check\alpha_0\in\intervaL{0,\alpha_\infty}$ which is largest with the property that 
for all $\alpha_0\in\interval{0,\check{\alpha}_0}$ the corresponding 
trajectory intersects $\{\vartheta=0\}$ for some $s=\sigma_1>0$. 
(Based on our numerical simulations we expect $\check{\alpha}_0=\alpha_\infty$ because otherwise we would obtain two distinct $\infty$-figures, which we did not observe.)

Any such trajectory has a unique $s_2\in\interval{0,\sigma_1}$ where $\vartheta$ attains a local maximum (so that $\alpha(s_2)=0$), as was shown in the discussion before \eqref{eqn:20210915-1}. If we consider a suitably small open neighbourhood of the segment of initial data $\left\{(\pi/2, 0,\alpha_0) \st \alpha_0\in\Interval{\check{\alpha}_0/2,\check{\alpha}_0}\right\}$, by looking at the right-hand side of the ODE in question all the first derivatives of the motion are bounded uniformly in the closure of the domain and so it takes a definite amount of time to leave it.
Hence, $r$ has to increase a bit before $\alpha$ reaches zero: 
there exists some $\delta>0$ such that $r(s_2)>\pi/2+\delta$ for any choice of initial $\alpha_0\in\Interval{\check{\alpha}_0/2,\check{\alpha}_0}$.  
This allows us to show (like we did in the proof of Lemma \ref{lem:20210829-green}) that $\vartheta(s_2)$ is uniformly bounded away from $\pi/4$. Let 
\begin{align}\label{eqn:k=1check}
(\check{r},\check{\vartheta},\check{\alpha})\colon\Interval{0,\check{s}_*}\to D
\end{align} 
be the solution to system \eqref{eqn:dr/ds2},\eqref{eqn:dt/ds2},\eqref{eqn:da/ds2} with initial data $(\check{r},\check{\vartheta},\check{\alpha})(0)=(\pi/2,0,\check{\alpha}_0)$ as defined above. 
With the same argument as in the proof of Lemma \ref{lem:20210829-green}, we can then show $\check\vartheta(s)>0$ for all $0<s<\check{s}_*$ and 
\begin{align}
\lim_{s\to \check{s}_*}(\check\vartheta,\check\alpha)(s)&=(0,-\tfrac{\pi}{2}).
\end{align}
In particular, the closure of the trajectory of \eqref{eqn:k=1check} intersects $\{\vartheta=0\}$ exactly twice and we may define $\check{\alpha}_0^{(1)}\vcentcolon=\check{\alpha}_0$.

\emph{Construction of $\hat{\alpha}_0^{(1)}$.}
Let $(r,\vartheta,\alpha)\colon\Interval{0,s_*}\to D$ be the solution with $(r,\vartheta,\alpha)(0)=(\pi/2,0,\alpha_0)$ for some $\alpha_0\in\interval{0,\check{\alpha}_0}$ to be chosen. 
By definition of $\check{\alpha}_0$, the set  $\vartheta^{-1}(\{0\})\setminus\{0\}$ is nonempty and 
Lemma~\ref{lem:20210612-1} implies that $s=0$ is an isolated zero of $\vartheta(s)$. 
Therefore, since $\vartheta$ is continuous
\[\sigma_1\vcentcolon=\min\bigl(\vartheta^{-1}(\{0\})\setminus\{0\}\bigr)\]
is well-defined. 
Then, $\alpha(\sigma_1)\leq0$ since $\vartheta(s)\geq0$ in $[0,\sigma_1]$. 
In fact, $\alpha(\sigma_1)<0$ because otherwise, $\alpha$ and $\vartheta$ would both vanish identically by ODE uniqueness. 
Hence, we may define 
$s_\alpha\vcentcolon=\sup\{b\in\interval{\sigma_1,s_*}\st\forall s\in\Interval{\sigma_1,b}~-\pi/2<\alpha(s)<0\}$.
For $\alpha_0=\check{\alpha}_0$ we just set $s_\alpha=s_*$ and recall that
in this case we have 
\begin{align}\label{eqn:20210901-1}
\lim_{s\to s_\alpha}\alpha(s)=-\frac{\pi}{2}.
\end{align}
For any sufficiently small $\alpha_0>0$ however, 
Lemma \ref{lem:small_alpha-iteration} implies $s_\alpha<s_*$ and $\alpha(s_\alpha)=0$. 
Therefore, there exists $\hat\alpha_0\in\intervaL{0,\check{\alpha}_0}$ which is smallest  
with the property that 
for all $\alpha_0\in\intervaL{\hat\alpha_0,\check\alpha_0}$ 
the corresponding solution has property \eqref{eqn:20210901-1}. 
Let 
\begin{align}\label{eqn:k=1hat}
(\hat{r},\hat{\vartheta},\hat{\alpha})\colon\Interval{0,\hat{s}_*}&\to D
\end{align}
be the solution with initial data $(\hat{r},\hat{\vartheta},\hat{\alpha})(0)=(\pi/2,0,\hat\alpha_0)$. 
If $s_{\hat{\alpha}}<\hat{s}_*$ then $\hat\alpha(s_{\hat{\alpha}})=0$ and equation~\eqref{eqn:da/ds2} implies $d\hat\alpha/ds(s_{\hat{\alpha}})\neq0$, that is, 
$s=s_{\hat{\alpha}}$ is a nondegenerate zero of $\hat\alpha$. 
Moreover, $\hat\alpha$ is bounded away from $-\pi/2$ in $\interval{\hat\sigma_1,s_{\hat\alpha}}$. 
Therefore, there exists $\alpha_0>\hat\alpha_0$ such that the corresponding trajectory still satisfies $\alpha(s_\alpha)=0$ in contradiction to our choice of $\hat\alpha_0$. 
Hence, we must have $s_{\hat{\alpha}}=\hat{s}_*$. 
In particular, $\hat\alpha(s)<0$ for all $s\in\interval{\sigma_1,\hat{s}_*}$, 
which also implies $\hat{s}_*<\infty$.
Then, by equation \eqref{eqn:dt/ds2}, $\hat\vartheta(s)$ is strictly decreasing and negative in $\interval{\sigma_1,\hat{s}_*}$. 
This implies on the one hand that $\hat\vartheta(s)$ converges to a limit $\hat\vartheta_*<0$  as $s\to \hat{s}_*$ and on the other hand that 
the trajectory of \eqref{eqn:k=1hat} intersects $\{\vartheta=0\}$ exactly twice: at $s=0$ and at $s=\sigma_1$. 
By Lemma \ref{lem:well-posed_centre}, 
$\hat{r}(s)\leq\hat{r}_*<\pi$ for all $s\in\Interval{0,\hat{s}_*}$ and we have 
\begin{align}
\label{eqn:20210902-t}
\hat\vartheta_*=\lim_{s\to\hat{s}_*}\hat\vartheta(s)&=-\frac{\pi}{4}  
\shortintertext{or}
\label{eqn:20210902-a}
\limsup_{s\to\hat{s}_*}\bigl(-\hat\alpha(s)\bigr)&=\frac{\pi}{2}.
\end{align}
We claim that in fact both things are true. 
If on the contrary $0>\hat\vartheta_*>-\pi/4$ then by \eqref{eqn:20210902-a} there is a sequence $s_k\to\hat{s}_*$ such that $\hat\alpha(s_k)\to-\pi/2$ as $k\to\infty$. 
For sufficiently large $k\in\N$, equation \eqref{eqn:da/ds2} and the monotonicity of $\hat{r}(s)$ imply
\begin{align}\label{eqn:20210907-a}
\frac{d\hat\alpha}{ds}(s_k)&\leq-\frac{(2n-2)\tan(2\hat\vartheta_*)}{\sin(\hat{r}_*)}\cos\bigl(\hat\alpha(s_k)\bigr)-(2n-1)\cot\bigl(\hat{r}(s_k)\bigr)\sin\bigl(\hat\alpha(s_k)\bigr)
\\   \label{eqn:20210907-a-2}
&\leq -\cot\bigl(\hat{r}(s_k)\bigr)\sin\bigl(\hat\alpha(s_k)\bigr)<0
\end{align}
because as soon as $\cos\bigl(\hat\alpha(s_k)\bigr)$ is sufficiently small, the second term in \eqref{eqn:20210907-a} dominates the first. 
In particular, for any $k$ large enough, there is $\tilde{s}_k>s_k$ such that $\hat\alpha(s)$ is strictly decreasing in $\Interval{s_k,\tilde{s}_k}$. 
Yet, combining this fact with the monotonicity of $\hat{r}(s)$ we obtain that actually   
\begin{align} 
\frac{d\hat\alpha}{ds}(s)&\leq -\cot\bigl(\hat{r}(s_k)\bigr)\sin\bigl(\hat\alpha(s_k)\bigr)<0
\end{align}
for all $s\in\Interval{s_k,\tilde{s}_k}$ and, therefore, we may choose $\tilde{s}_k=\hat{s}_*$. 
So, a posteriori, $\hat\alpha(s)$ is strictly decreasing in a left neighbourhood of $\hat{s}_*$ and \eqref{eqn:20210902-a} is in fact a limit. 
Then, since we assume $\hat\vartheta_*>-\pi/4$, we can extend the solution to the interval $[0,\hat{s}_*+\epsilon]$ for some small $\epsilon>0$ (allowing values $\hat\alpha<-\pi/2$) and the trajectory will ``turn clockwise'' in the sense that 
\begin{align*}
\frac{d\hat\alpha}{ds}(\hat{s}_{*})&= (2n-1)\cot(\hat{r}_{*})<0.
\end{align*} 
This implies that the function $[0,\hat{s}_*+\epsilon]\ni s\mapsto\hat\alpha(s)+\pi/2$ has a nondegenerate zero at $s=\hat{s}_*$. 
Hence, by continuous dependence on the initial data there exists $\hat\alpha'_0<\hat\alpha_0$ such that the (extended) solutions corresponding to any $\alpha_0 \in\intervaL{\hat\alpha'_0,\hat\alpha_0}$ still intersect $\{\alpha=-\pi/2\}$ in contradiction to our choice of $\hat\alpha_0$. 

Instead, if \eqref{eqn:20210902-t} was true but \eqref{eqn:20210902-a} was not, then equation \eqref{eqn:da/ds2} would imply that $d\hat\alpha/ds$ is unbounded (with divergent integral) as $s\to\hat{s}_*$ which again yields a contradiction. 
Therefore, \eqref{eqn:20210902-t} and \eqref{eqn:20210902-a} both hold. 
In fact, we can show that \eqref{eqn:20210902-a} is actually a limit. 
If, on the contrary, there was a sequence $s_k\nearrow s_*$ as $k\to\infty$ with $0>\hat\alpha(s_k)\geq\hat\alpha_*>-\pi/2$ for all $k\in\N$, then since 
$-\tan(2\hat\vartheta(s_k))\to+\infty$ by \eqref{eqn:20210902-t}, equation \eqref{eqn:da/ds2} implies 
\begin{align}\label{eqn:20210907-da>0}
\frac{d\hat\alpha}{ds}(s_k)&\geq-\frac{(2n-2)\tan\bigl(2\hat\vartheta(s_k)\bigr)}{\sin\bigl(\hat{r}(s_k)\bigr)}\cos(\hat\alpha_*)-(2n-1)\cot(\hat r_*)\sin\bigl(\hat\alpha(s_k)\bigr)>0
\end{align}
if $k$ is sufficiently large because then the first term in \eqref{eqn:20210907-da>0} eventually dominates the second. 
In particular, there exists $\tilde{s}_k>s_k$ such that $\hat{\alpha}(s)$ is strictly increasing in $\Interval{s_k,\tilde{s}_k}$.  
But then, since $\cos(\hat\alpha(s))$ and $\hat{r}(s)$ are strictly increasing, we obtain that actually $d\hat\alpha/ds(s)>0$ for all $s\in\Interval{s_k,\tilde{s}_k}$ with a lower bound only depending on $k$ (and not on $\tilde{s}_k$). 
Hence, we may choose $\tilde{s}_k=\hat{s}_*$ and conclude $\hat\alpha(s)\geq\hat\alpha(s_k)$ for all $s\in\Interval{s_k,\hat{s}_*}$ in contradiction to \eqref{eqn:20210902-a}.
 
With $\check{\alpha}_0^{(1)}\vcentcolon=\check{\alpha}_0$ and $\hat{\alpha}_0^{(1)}\vcentcolon=\hat{\alpha}_0$ the statement follows in the case $k=1$ since we showed that the solutions \eqref{eqn:k=1hat} and \eqref{eqn:k=1check} satisfy the desired properties.   
We now proceed by induction and suppose that the statement holds for some $1\leq k\in\N$. 
By the way we set up the construction, we have that for all $\alpha_0\in\interval{0,\check{\alpha}_0^{(k)}}$ the closure of the image of the corresponding solution intersects $\{\vartheta=0\}$ at least $k+1$ times for some $0=\sigma_0<\sigma_1<\ldots<\sigma_k< s_*$. 

\emph{Construction of $\check{\alpha}_0^{(k+1)}$.} 
On the one hand, the trajectory corresponding to $\alpha_0=\hat\alpha_0^{(k)}<\check{\alpha}_0^{(k)}$ intersects $\{\vartheta=0\}\cap D$ exactly $k+1$ times ($s=0$ included). 
On the other hand, Lemma \ref{lem:small_alpha-iteration} implies that if $\alpha_0\in\interval{0,\hat\alpha_0^{(k)}}$ is sufficiently small, then the corresponding trajectory intersects $\{\vartheta=0\}\cap D$ at least $k+2$ times. 
Therefore, there exists $\check{\alpha}_0^{(k+1)}\in\interval{0,\hat\alpha_0^{(k)}}$ which is largest with the property that for all $\alpha_0\in\interval{0,\check\alpha_0^{(k+1)}}$ the corresponding trajectory intersects $\{\vartheta=0\}\cap D$ at least $k+2$ times. 
Let 
\begin{align*}
\bigl(\check{r}^{(k+1)},\check{\vartheta}^{(k+1)},\check{\alpha}^{(k+1)}\bigr)\colon\Interval{0,\check{s}_*^{(k+1)}}\to D
\end{align*}
be the solution to system \eqref{eqn:dr/ds2},\eqref{eqn:dt/ds2},\eqref{eqn:da/ds2} with initial data $\bigl(\pi/2,0,\check{\alpha}_0^{(k+1)}\bigr)$ as defined above. 
With the same argument as above we can show that the trajectory of this solution intersects $\{\vartheta=0\}\cap D$ exactly $k+1$ times for some $0=\sigma_0<\sigma_1<\ldots<\sigma_k< s_*$ and, furthermore,   
\[
\bigl(\check{\vartheta}^{(k+1)},\check{\alpha}^{(k+1)}\bigr)(s)\to(-1)^{k+1}\bigl(0,\tfrac{\pi}{2}\bigr)
\] 
as $s\to \check{s}_*^{(k+1)}$. 
In particular, the closure of the trajectory intersects $\{\vartheta=0\}$ exactly $k+2$ times. 

\emph{Construction of $\hat{\alpha}_0^{(k+1)}$.} 
The claim follows by exactly the same arguments as in the construction of  $\hat{\alpha}_0^{(1)}$ replacing $\check\alpha^{(1)}_0$ by $\check{\alpha}_0^{(k+1)}$ and $\sigma_1$ by $\sigma_{k+1}\vcentcolon=\min\bigl(\vartheta^{-1}(\{0\})\setminus\{0,\sigma_1,\ldots,\sigma_k\}\bigr)$, plus accounting for the obvious sign changes depending on the parity of $k$.
\end{proof}

\begin{remark}\label{rem:AlphaGoesToZero}
In the notation of Lemma \ref{lem:iteration} above, we have $\hat{\alpha}_0^{(k)}\to0$ and $\check{\alpha}_0^{(k)}\to0$ as $k\to\infty$. 
Indeed, if either of these sequences was bounded from below by a positive constant, then
by Lemma~\ref{lem:zeros} the same lower bound would hold for $\abs{\alpha}$ evaluated at any of the zeros $0<\sigma_1<\ldots<\sigma_k< s_*$ of $\vartheta$, and the local extremal values of $\vartheta$ in between (alternating between local maxima and minima) are also uniformly bounded from below in absolute value. 
Since $s\mapsto r(s)$ is bounded and strictly increasing, there exists some sequence of integers $1\leq j_k\leq k/2$ such that $r(\sigma_{j_k})-r(\sigma_{j_k-1})\to0$ as $k\to\infty$. 
Recalling the content of Remark \ref{rem:MonotonTheta}, we thus obtain 
for every $\ell=j_k,j_{k+1},\ldots,k$
\[
\abs[\bigg]{\Bigl(\frac{dr}{ds}\Bigr)^{-1}\frac{d\vartheta}{ds}}(\sigma_{\ell})\to\infty 
\]
as $k\to\infty$, 
which, however, is only possible if either $r(\sigma_{\ell})\to\pi$ or $\abs{\alpha(\sigma_{\ell})}\to\pi/2$. 
By possibly employing the first conclusion of Lemma $\ref{lem:zeros}$ in the first case, we 
obtain (in either case) that, for $k$ large enough, $\alpha(\sigma_{k-1})$ would, depending on the parity of $k$, eventually either be above the threshold $a$ provided by Lemma \ref{lem:20210828} or below $-a$, hence forcing blow-up (before reaching $\sigma_k$) and thus a contradiction.  
\end{remark}

\begin{figure}
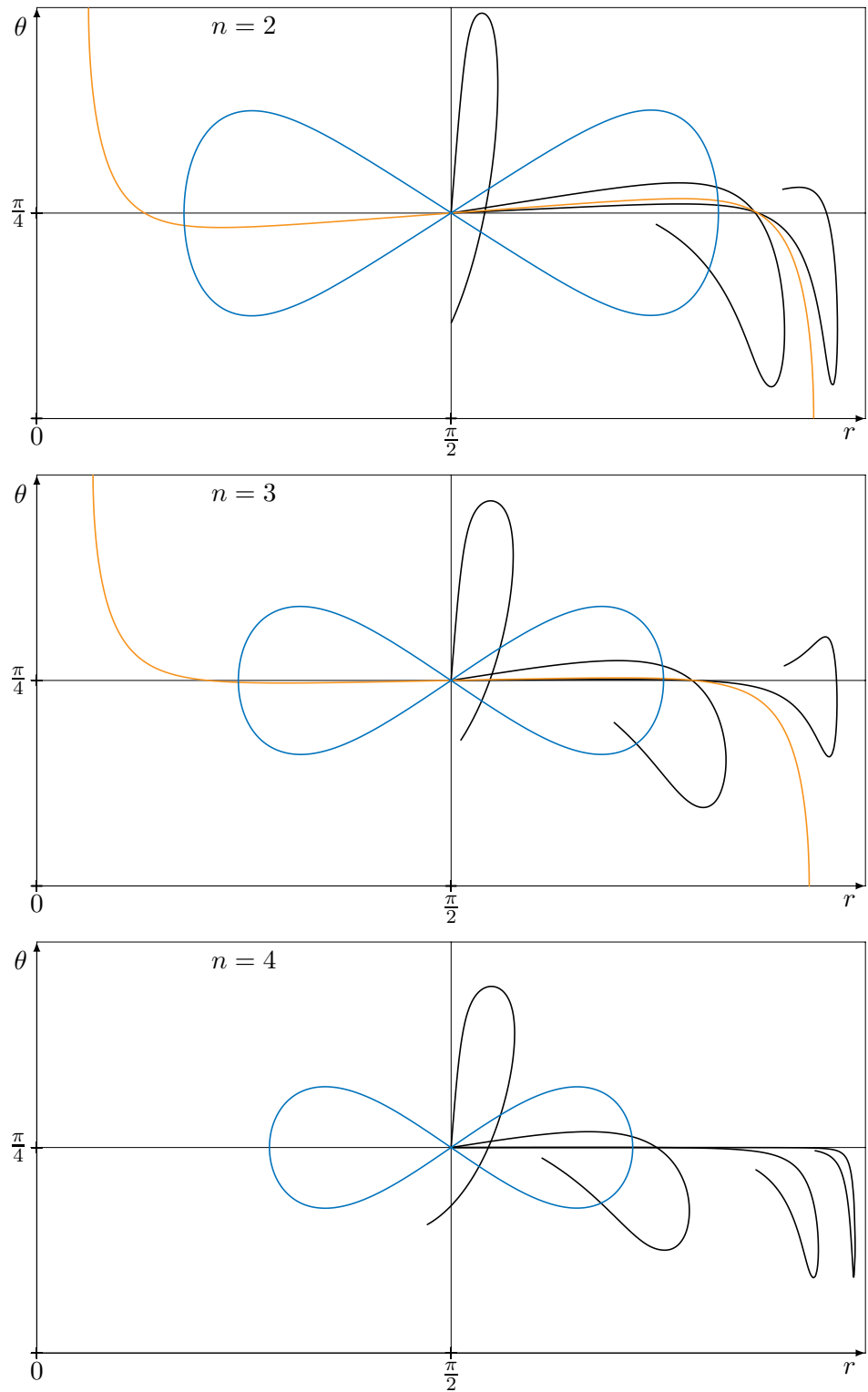
\centering
\pgfmathsetmacro{\scale}{\textheight/6cm}

\caption{Shooting from the centre with varying $\alpha_0\in\interval{0,\pi/2}$ for the cases $n\in\{2,3,4\}$. 
}%
\label{fig:n=2,3,4}%
\end{figure}

\begin{proof}[Proof of Theorems \ref{thm:mainImmersed} and \ref{thm:mainHyperspheres}]
At this stage, both theorems follow at once by simply reflecting the trajectories constructed in Lemma \ref{lem:iteration}. Specifically, the projection to the coordinates $(r,\vartheta)$ of the solutions of the first type, i.\,e.
\[
\bigl(\hat{r}^{(k)},\hat{\vartheta}^{(k)}\bigr)\colon\Interval{0,\hat{s}_*^{(k)}}\to \Interval{\tfrac{\pi}{2},\pi}\times\interval{-\tfrac{\pi}{4},\tfrac{\pi}{4}}
\]
when reflected centrally through the origin give rise to curves that lift to minimally embedded hyperspheres in $\Sp^{2n}$, while the solutions of the second type, i.\,e.
\[
\bigl(\check{r}^{(k)},\check{\vartheta}^{(k)}\bigr)\colon\Interval{0,\check{s}_*^{(k)}}\to \Interval{\tfrac{\pi}{2},\pi}\times\interval{-\tfrac{\pi}{4},\tfrac{\pi}{4}}
\]
when reflected with respect to both axes give rise to closed curves that lift to minimal immersions of $S^{n-1}\times S^{n-1}\times S^1$ in $\Sp^{2n}$. Also, it follows from our construction that any pair of elements of either family are mutually distinct, but we do need to discuss why we do actually get two infinite families of \emph{pairwise non-isometric} hypersurfaces. As far as Theorem \ref{thm:mainImmersed} is concerned, it suffices to note that the intersections of the trajectories we construct
with the $\vartheta$-bisector 
$Z\vcentcolon=\{\vartheta=0\}$ (which corresponds to $\{\theta=\pi/4\}$ in the notation of Sections \ref{sec:setup} and \ref{sec:embedded})
lift to the connected components of the 
self-intersections of the lifted hypersurface (and of course the number of such connected components is invariant under ambient isometries). For Theorem \ref{thm:mainHyperspheres} we need a somewhat less direct argument, instead. Clearly, if either the volume or the Morse index of the hyperspheres we construct were unbounded then we would be done, so let us assume, without loss of generality, that these two quantities are both uniformly bounded. By Sharp's compactness theorem, see \cite{Sha17}, there would be a subsequence of minimal hyperspheres converging smoothly and graphically, with finite multiplicity $m\geq 1$, and possibly away from finitely many points, to a smooth (embedded) minimal hypersurface in $\Sp^{2n}$. Such a hypersurface inherits the $G$-equivariance, and thus it is the lift of a solution curve in the $(r,\vartheta)$ domain. 
However, since each curve $(\hat{r}^{(k)},\hat{\vartheta}^{(k)})$ is actually a graph, and we know that $\hat{\alpha}^{(k)}_0\to 0$ as $k\to\infty$, we conclude that the convergence above actually happens with unit multiplicity (i.\,e. $m=1$), hence smoothly at all points, and the curve in question must be tangent to $Z$ at the origin. 
Hence, by the uniqueness result in the Cauchy--Lipschitz theorem for solutions to ODE, we conclude that the hypersurface in question must be the lift of $Z$ itself, which contradicts the smoothness of the limit hypersurface we had obtained.
\end{proof}

In Figure \ref{fig:n=2,3,4} we display multiple trajectories of solutions to system \eqref{eqn:dr/ds},\eqref{eqn:dt/ds},\eqref{eqn:da/ds} with initial data $(\pi/2,\pi/4,\alpha_0)$. 
Note that here we are working with respect to the original variable $\theta\in\interval{0,\pi/2}$ and do not restrict the trajectories to any specific subdomain. 
In each of the cases $n\in\{2,3,4\}$ we consider different choices of $\alpha_0\in\interval{0,\pi/2}$. 
When $n\in\{2,3\}$, we observe that depending on the choice of $\alpha_0$ the trajectory either turns clockwise ($d\alpha/ds\leq0$) or anticlockwise ($d\alpha/ds\geq0$) at the first interior local minimum of $\theta$. 
In fact, we either have $\alpha=-\pi$ or $\alpha=0$ at these points. 
In between there is a trajectory with $\theta\to0$ as $s\to s_{\ast}$. 
In the case $n=4$, however, we observe that trajectories always turn clockwise at the first interior local minimum of $\theta$ even if $\alpha_0>0$ is chosen extremely small. 
This suggests that no trajectory through the centre may exit at $\theta=0$. 
This phenomenon has been observed for all values of $n\geq4$ we ran the numerical simulations for, and leads to formulate what follows:

\begin{conj}\label{conj:n>3}
For $n\geq4$ any $G$-equivariant, minimal embedding of $S^{2n-1}$ in $\Sp^{2n}$ is equatorial.  
\end{conj}

\begin{remark} 
Conjecture \ref{conj:n>3} (where $G=\Ogroup(n)\times\Ogroup(n)$ consistently with the notation we have employed throughout the paper) should be compared with Hsiang's \cite{Hsiang1983} Conjecture 1 stating that for each $m\geq7$ there exist infinitely many distinct  minimal embeddings of $S^{m-1}$ into $\Sp^{m}$, and later settled in \cite{HsiSte86} (see Theorem 4 therein) in the category of \emph{immersed} minimal hypersurfaces.
\end{remark}

The existence of the $\infty$-figure however does not seem to be obstructed for any values of $n\geq4$ we ran numerical simulations for. 
One could attempt a proof using (slightly or entirely) different arguments, keeping in mind that our construction of the $\infty$-figure for $n\in\{2,3\}$ relies on Lemma \ref{lem:20210705} which breaks down for $n\geq4$ (for indeed the crucial inequality \eqref{eqn:20210629-negative} ceases to hold true in that case, no matter how small the shooting angle $\alpha_0$). 
A careful analysis of this phenomenon falls somewhat beyond the scope of the present article since -- given the geometric problem we started with -- we are interested in the case $n=2$, where the $\infty$-figure lifts to a minimal immersion of a hypertorus $T^3$ in $\Sp^4$.

\appendix 

\section{A minimal embedding of \texorpdfstring{$T^4$}{a hypertorus} in \texorpdfstring{$\Sp^5$}{the five-dimensional sphere}}\label{sec:HsiangClaim}

In this appendix, we describe how the analysis presented in Section \ref{sec:embedded} allows, with simple modifications, to also prove the following statement:

\begin{theorem}\label{thm:HsiangClaim}
There exists a minimal embedding of $S^{1}\times S^{1}\times S^1\times S^1$ in the round sphere $\Sp^{5}$.  
\end{theorem}

To that aim, one considers the action of $H=\Ogroup(n)\times\Ogroup(n)\times\Ogroup(n)$ on $\Sp^{3n-1}\subset \R^{3n}=\R^n\times \R^n\times \R^n$ given by the outer direct sum of the standard representations of the orthogonal group. The quotient $\Sp^{3n-1}/H$, endowed with the orbital metric, is (isometric to) the first octant of the unit round sphere, i.\,e. $\bigl\{(X,Y,Z)\in\R^3\st 0\leq X,Y,Z\leq 1,~ X^2+Y^2+Z^2=1\bigr\}$ and we note, following \cite{HsiHsi80}, that the geodesic arcs emanating from the vertices to the midpoints of the opposite sides divide such a domain into six subdomains, so that the problem amounts to constructing, in any one of those, an arc solving the appropriate geodesic curvature prescribed equation and meeting the boundary orthogonally (see Figure \ref{fig:shooting3}, right image). 
That being said, diverting from that source, we prefer to study the problem in polar coordinates and construct such a free boundary arc relying on arguments similar to those in Section \ref{sec:embedded} above. 

Thus, we work in coordinates $(r,\theta)\in [0,\pi/2]\times [0,\pi/2]$ and with respect to the spherical metric $g=dr^2+\sin^2(r)d\theta^2$; one needs to prove the existence of certain geodesic arcs in the conformal metric $W^2g$ where $W(r,\theta)=\omega^3_{n-1}\abs{\sin(r)\cos(\theta)}^{n-1}\abs{\sin(r)\sin(\theta)}^{n-1}\abs{\cos(r)}^{n-1}$ (if we take, from now onwards, $n=2$ then simply $W(r,\theta)=4\pi^3\sin^2(r)\cos(r)\sin(2\theta)$; this suffices for our purposes). 
Thus, proceeding exactly as in Section \ref{sec:setup}, we need to find a closed smooth orbit for the $3\times 3$ system given by \eqref{eqn:dr/ds}, \eqref{eqn:dt/ds} and 
\begin{align}\label{eqn:da/ds3}
\frac{d\alpha}{ds}&=\frac{2\cot(2\theta)}{\sin(r)}\cos(\alpha)
-\frac{2+4\cos(2r)}{\sin(2r)}\sin(\alpha).
\end{align}
By virtue of the aforementioned symmetry arguments we work in the `fundamental domain' (see Figure \ref{fig:shooting3} left image) defined in these coordinates by letting 
\begin{align}\label{eqn:domainB'}
B'\vcentcolon=\left\{(r,\theta)\in\intervaL{0,\pi/2}\times\intervaL{0,\pi/4}\st \tan(r)\cos(\theta)\leq 1\right\}
\end{align}
and will further consider $B\vcentcolon=B'\times \IntervaL{-\pi/2,0}$. We will consider solutions of the system above with initial data of the form $(r_0,\pi/4,-\pi/2)$, exactly as in Section \ref{sec:embedded}. For the sake of convenience, in the setting above let us define $u(r,\theta)=\tan(r)\cos(\theta)$ and, furthermore
\[
v(r,\theta)=-\arctan\bigl(\cos(r)\tan(\theta)\bigr),
\]
that is precisely the value of $\alpha$ corresponding to an outward-pointing vector normal, in metric $g$, to the regular curve $\left\{u(r,\theta)=1\right\}$ at the point $(r,\theta)\in\partial B'$.

We note that, set $\varrho\vcentcolon=\arctan(\sqrt{2})<\pi/3$, 
as long as $(r,\theta,\alpha)\in\intervaL{0,\varrho}\times\intervaL{0,\pi/4}\times\IntervaL{-\pi/2,0}$ equations \eqref{eqn:dr/ds}, \eqref{eqn:dt/ds} and \eqref{eqn:da/ds3} imply that $r$ and $\alpha$ are weakly increasing and $\theta$ is weakly decreasing, and there are no stationary points for the motion. One can then easily prove a well-posedness lemma analogous to Lemma \ref{lem:well-posed}, where (in the final clause) one shall read that \emph{either} $\tan(r(s_*))\cos(\theta(s_*))=1$ \emph{or} $\alpha(s_*)=0$; in either case $\theta(s_*)>0$. 
Indeed, equations \eqref{eqn:da/ds} and \eqref{eqn:da/ds3} for $\alpha$ have the same first summand up to a constant factor and their second summand (restricted to the respective domain) is nonnegative in both cases. Therefore, we may argue exactly as in the proof of Lemma~\ref{lem:well-posed}, where we simply drop the second summand.  
The following estimates are analogous to Lemmata \ref{lem:dipping_down} and \ref{lem:roof}. 

\begin{lemma}\label{lem:dipping_down-appendix}
Given $r_0\in\interval{0,\pi/8}$ let $(r,\theta,\alpha)\colon[0,s_*]\to B$ be the solution to system \eqref{eqn:dr/ds}, \eqref{eqn:dt/ds}, \eqref{eqn:da/ds3} and let $s_1\in\intervaL{0,s_*}$ be arbitrary. 
If $r(s_1)\geq2r_0$ then $\theta(s_1)<\pi/4-1/12$. 
\end{lemma}

\begin{proof}
Let $\delta=1/12$ be fixed and assume there exists $s_1\in\intervaL{0,s_*}$ with $r(s_1)\geq 2r_0$. 
Towards a contradiction, suppose $\theta(s_1)\geq\pi/4-\delta$. 
By monotonicity of $\theta$, we then have $\theta(s)\geq\pi/4-\delta$ for all $s\in[0,s_1]$. Hence, $0\leq\cot(2\theta)\leq\tan(2\delta)=\vcentcolon b$  
for all $s\in[0,s_1]$. 
Equations \eqref{eqn:dt/ds} and \eqref{eqn:da/ds} then imply  
\begin{align}\label{eqn:da/dt3}
\frac{d\alpha}{ds}
&=\Bigl(2\cot(2\theta)\cot(\alpha)
-\frac{2+4\cos(2r)}{\sin(2r)}\sin(r)
\Bigr)\frac{d\theta}{ds}
\leq3\Bigl(b\cot(\alpha)-1\Bigr)\frac{d\theta}{ds}.
\end{align}
Estimate \eqref{eqn:da/dt3} is identical to \eqref{eqn:da/dt} for $n=2$ in the proof of Lemma \ref{lem:dipping_down}. 
Analogously to \eqref{eqn:alpha(s1)}, 
\begin{align}\label{eqn:20210920}
\alpha(s_1)&\leq-b\log(b)+(b^2+1)3\delta-\frac{\pi}{2}<-\frac{\pi}{4}.
\end{align}
We then obtain the contradiction $\theta(s_1)-\pi/4<-\delta$ with the same argument as for \eqref{eqn:dt/dr}--\eqref{eqn:20210918-ts1}. 
\end{proof}

\begin{lemma}[see Figure \ref{fig:shooting3}\;(1)]\label{lem:roof-appendix}
There exists a constant $c>0$ such that if the initial data $r_0\in\interval{0,\pi/8}$ is sufficiently small, then the solution to system \eqref{eqn:dr/ds},\eqref{eqn:dt/ds},\eqref{eqn:da/ds3}  satisfies $\alpha(s_*)=0$ and $r(s_*)\leq c r_0$. 
\end{lemma}

\begin{proof}
We may argue exactly as in the proof of Lemma \ref{lem:roof} for $n=2$. 
By Lemma \ref{lem:dipping_down-appendix}, we obtain $d\alpha/ds\geq(2/r)\tan(1/6)dr/ds$ for all $s\in[s_1,s_*]$ as in \eqref{eqn:da/ds-estimate}. 
The claim then follows by integration. 
\end{proof}

\begin{figure}
\centering
\begin{tikzpicture}[line cap=round,line join=round,baseline={(0,pi/4*\textwidth/3.5)},scale=\textwidth/3.5cm] 
\draw[-latex](0,0)--(pi/2,0)node[right]{$r$};
\draw plot[plus]({rad(atan(sqrt(2)))},0)node[below]{$\varrho\mathrlap{{}\vcentcolon=\arctan\sqrt{2}}$};
\pgfresetboundingbox
\draw[thin](0,0)grid [xstep=pi/2,ystep=pi/4](pi/2,pi/2);
\draw[-latex](0,0)--(0,pi/2)node[below left]{$\theta$};
\draw plot[plus](0,0)node[below]{$0$};
\draw plot[plus](pi/4,0)node[below]{$\frac{\pi}{4}$};
\draw plot[plus](0,pi/4)node[left]{$\frac{\pi}{4}$};
\draw[very thin,dashed]({rad(atan(sqrt(2)))},0)--++(0,pi/2);
\draw[domain=0:pi/2,variable=\t,thin]plot({rad(atan(1/cos(deg(\t))))},{pi/2-\t}); 
\draw[domain=0:pi/2,variable=\t,thin]plot({rad(atan(1/cos(deg(\t))))},{\t});
\draw[domain=0:pi/4,variable=\t,color={cmyk,1:cyan,1;yellow,0.5},thick,fill=black!5]plot({rad(atan(1/cos(deg(\t))))},{\t})-|(0,0)--cycle;
%
\draw[smooth,torustrajectory,dotted]plot coordinates {(0.8138,0.3513) (0.8244,0.3478) (0.8351,0.3449) (0.8459,0.3425) (0.8568,0.3407) (0.8677,0.3395) (0.8786,0.3388) (0.8896,0.3386) (0.9006,0.3389) (0.9116,0.3397) (0.9225,0.3410) (0.9334,0.3426) (0.9442,0.3447) (0.9549,0.3472) (0.9656,0.3500) (0.9761,0.3532) (0.9866,0.3567) (0.9971,0.3606) (1.0074,0.3647) (1.0176,0.3691) (1.0277,0.3738) (1.0378,0.3788) (1.0477,0.3840) (1.0575,0.3895) (1.0673,0.3952) (1.0769,0.4011) (1.0863,0.4073) (1.0957,0.4136) (1.1049,0.4201) (1.1141,0.4268) (1.1230,0.4337) (1.1319,0.4408) (1.1406,0.4480) (1.1492,0.4554) (1.1576,0.4630) (1.1659,0.4707) (1.1741,0.4786) (1.1821,0.4866) (1.1899,0.4948) (1.1976,0.5031) (1.2051,0.5116) (1.2124,0.5202) (1.2196,0.5289) (1.2266,0.5379) (1.2334,0.5469) (1.2400,0.5561) (1.2464,0.5654) (1.2525,0.5749) (1.2585,0.5845) (1.2642,0.5943) (1.2697,0.6041) (1.2750,0.6141) (1.2800,0.6243) (1.2847,0.6345) (1.2892,0.6449) (1.2933,0.6554) (1.2972,0.6660) (1.3008,0.6767) (1.3040,0.6875) (1.3070,0.6984) (1.3096,0.7093) (1.3118,0.7204) (1.3137,0.7315) (1.3153,0.7426) (1.3165,0.7538) (1.3173,0.7650) (1.3178,0.7763) (1.3179,0.7875) (1.3176,0.7988) (1.3170,0.8101) (1.3160,0.8213) (1.3146,0.8325) (1.3129,0.8436) (1.3108,0.8547) (1.3084,0.8657) (1.3057,0.8766) (1.3026,0.8874) (1.2992,0.8982) (1.2955,0.9088) (1.2916,0.9193) (1.2873,0.9298) (1.2827,0.9401) (1.2778,0.9503) (1.2727,0.9604) (1.2673,0.9704) (1.2617,0.9802) (1.2559,0.9899) (1.2498,0.9995) (1.2434,1.0090) (1.2369,1.0183) (1.2302,1.0274) (1.2233,1.0364) (1.2162,1.0453) (1.2090,1.0540) (1.2015,1.0625) (1.1940,1.0709) (1.1862,1.0792) (1.1783,1.0873) (1.1702,1.0953) (1.1620,1.1031) (1.1537,1.1107) (1.1452,1.1182) (1.1365,1.1255) (1.1278,1.1327) (1.1189,1.1397) (1.1098,1.1465) (1.1006,1.1531) (1.0914,1.1595) (1.0820,1.1658) (1.0724,1.1718) (1.0628,1.1776) (1.0530,1.1832) (1.0432,1.1886) (1.0332,1.1937) (1.0231,1.1986) (1.0129,1.2032) (1.0027,1.2075) (0.9923,1.2115) (0.9818,1.2153) (0.9713,1.2186) (0.9607,1.2217) (0.9500,1.2244) (0.9392,1.2267) (0.9284,1.2286) (0.9176,1.2300) (0.9067,1.2311) (0.8958,1.2317) (0.8849,1.2318) (0.8741,1.2314) (0.8632,1.2305) (0.8523,1.2291) (0.8415,1.2271) (0.8308,1.2246) (0.8201,1.2215) (0.8096,1.2178) (0.7991,1.2135) (0.7888,1.2085) (0.7786,1.2030) (0.7686,1.1968) (0.7588,1.1899) (0.7492,1.1824) (0.7398,1.1742) (0.7306,1.1654) (0.7217,1.1559) (0.7130,1.1458) (0.7046,1.1350) (0.6965,1.1236) (0.6887,1.1116) (0.6812,1.0990) (0.6741,1.0858) (0.6672,1.0720) (0.6608,1.0577) (0.6547,1.0429) (0.6489,1.0276) (0.6436,1.0118) (0.6386,0.9956) (0.6340,0.9789) (0.6298,0.9619) (0.6260,0.9445) (0.6226,0.9268) (0.6196,0.9087) (0.6171,0.8905) (0.6149,0.8720) (0.6132,0.8533) (0.6119,0.8344) (0.6110,0.8155) (0.6106,0.7965) 
};
\draw[smooth,torustrajectory,thick]plot coordinates {(0.6107,0.7854) (0.6109,0.7664) (0.6116,0.7474) (0.6126,0.7285) (0.6141,0.7098) (0.6160,0.6912) (0.6184,0.6728) (0.6211,0.6546) (0.6243,0.6367) (0.6278,0.6191) (0.6318,0.6018) (0.6361,0.5849) (0.6408,0.5685) (0.6460,0.5524) (0.6515,0.5368) (0.6574,0.5217) (0.6636,0.5072) (0.6702,0.4931) (0.6772,0.4796) (0.6844,0.4667) (0.6921,0.4544) (0.7000,0.4426) (0.7082,0.4315) (0.7167,0.4210) (0.7255,0.4112) (0.7345,0.4020) (0.7437,0.3934) (0.7532,0.3855) (0.7629,0.3782) (0.7727,0.3715) (0.7828,0.3656) (0.7930,0.3602) (0.8033,0.3555) (0.8138,0.3513)  
};
%
\draw[smooth,dotted]plot coordinates {(0.2389,0.4289) (0.2470,0.4310) (0.2551,0.4346) (0.2632,0.4394) (0.2712,0.4452) (0.2792,0.4518) (0.2871,0.4589) (0.2950,0.4664) (0.3028,0.4742) (0.3106,0.4823) (0.3183,0.4905) (0.3261,0.4987) (0.3338,0.5070) (0.3415,0.5152) (0.3492,0.5234) (0.3568,0.5315) (0.3645,0.5395) (0.3721,0.5474) (0.3798,0.5551) (0.3874,0.5627) (0.3951,0.5702) (0.4027,0.5775) (0.4104,0.5846) (0.4181,0.5916) (0.4257,0.5985) (0.4334,0.6051) (0.4411,0.6116) (0.4488,0.6179) (0.4565,0.6241) (0.4642,0.6301) (0.4719,0.6360) (0.4797,0.6417) (0.4874,0.6473) (0.4952,0.6527) (0.5029,0.6580) (0.5107,0.6632) 
};
\draw[smooth,trajectory]plot coordinates {(0.1571,0.7854) (0.1577,0.7335) (0.1596,0.6832) (0.1626,0.6359) (0.1667,0.5928) (0.1717,0.5547) (0.1776,0.5221) (0.1841,0.4950) (0.1912,0.4732) (0.1987,0.4565) (0.2065,0.4442) (0.2145,0.4359) (0.2225,0.4309) (0.2307,0.4287) (0.2389,0.4289)};
%
\draw[smooth,dotted]plot coordinates { (0.9089,0.6819) (0.9110,0.6561) (0.9135,0.6305) (0.9164,0.6050) (0.9198,0.5797) (0.9237,0.5546) (0.9281,0.5296) (0.9330,0.5049) (0.9383,0.4804) (0.9442,0.4563) (0.9506,0.4324) (0.9575,0.4089) (0.9651,0.3857) (0.9732,0.3630) (0.9819,0.3408) (0.9913,0.3191) (1.0014,0.2979) (1.0123,0.2775) (1.0239,0.2578) (1.0365,0.2390) (1.0499,0.2213) (1.0644,0.2048) (1.0799,0.1897) (1.0966,0.1764) (1.1143,0.1651) (1.1331,0.1563) (1.1527,0.1502) (1.1729,0.1471) (1.1933,0.1472) (1.2135,0.1504) (1.2330,0.1567) (1.2515,0.1658) (1.2688,0.1772) (1.2846,0.1906) (1.2989,0.2058) (1.3116,0.2224) (1.3225,0.2402) (1.3318,0.2589) (1.3393,0.2784) (1.3451,0.2985) (1.3492,0.3190) (1.3515,0.3397) (1.3522,0.3606) (1.3513,0.3815) (1.3489,0.4023) (1.3451,0.4228) 
};
\draw[smooth,trajectory]plot coordinates {(0.9053,0.7854) (0.9055,0.7595) (0.9062,0.7335) (0.9073,0.7077) (0.9089,0.6819)  
};
\draw(0.2389,0.4289)node[below]{(1)};
\draw(0.9089,0.6819)node[left]{(2)};
\draw[torustrajectory](0.73,0.4289)node[below left]{(3)};
\draw(0.66*pi/4,0)node[above=2ex]{$B'$};
\end{tikzpicture}
\hfill
\tdplotsetmaincoords{60}{45}
\begin{tikzpicture}[baseline={(0,0)},scale=\textwidth/4.33cm,line cap=round,tdplot_main_coords] 
\path[tdplot_screen_coords,ball color=white] (0,0) circle (1);
\foreach\y in {0,9,...,90}{
\tdplotsetrotatedcoords{0}{-\y}{0}
\tdplotdrawarc[tdplot_rotated_coords,very thin,black!57]{(0,0,0)}{1}{-90}{0}{}{}
}
\tdplotsetthetaplanecoords{0}
\foreach\x in {1,...,9}{ 
\tdplotdrawarc[tdplot_rotated_coords ,very thin,black!57]
{(0,0,{sin(-9*\x)})} {{cos(-9*\x)}}{0}{90}{}{}
}
\tdplotdrawarc[tdplot_rotated_coords,-stealth,semithick]{(0,0,0)}{1}{90}{0}{right}{$\theta$}
\tdplotsetrotatedcoords{0}{-0}{0}
\tdplotdrawarc[tdplot_rotated_coords,very thin,dashed]{(0,0,0)}{1}{0}{360}{}{}
\tdplotsetrotatedcoords{0}{-90}{0}
\tdplotdrawarc[tdplot_rotated_coords]{(0,0,0)}{1}{-90}{0}{}{}
\tdplotsetrotatedcoords{0}{-0}{0}
\tdplotdrawarc[tdplot_rotated_coords,-stealth,semithick]{(0,0,0)}{1}{-90}{0}{below}{$r$}
%
\begin{scope}[ ]
\tdplotsetrotatedcoords{0}{-45}{0}
\tdplotdrawarc[tdplot_rotated_coords]{(0,0,0)}{1}{-90}{0}{}{}
\tdplotsetthetaplanecoords{-45}
\tdplotdrawarc[tdplot_rotated_coords]{(0,0,0)}{1}{90}{0}{}{}
\tdplotsetrotatedcoords{0}{-45}{0}
\tdplotdrawarc[tdplot_rotated_coords]{(0,0,0)}{1}{-90}{0}{}{}
\tdplotsetrotatedcoords{90}{45}{0}
\tdplotdrawarc[tdplot_rotated_coords]{(0,0,0)}{1}{180}{270}{}{}
\end{scope}
%
\begin{scope}[color={cmyk,1:cyan,1;yellow,0.5},thick]
\tdplotsetrotatedcoords{0}{-0}{0}
\tdplotdrawarc[tdplot_rotated_coords]
{(0,0,0)}{1}{-90}{-45}{}{}
\tdplotsetthetaplanecoords{-45}
\tdplotdrawarc[tdplot_rotated_coords]{(0,0,0)}{1}{90}{55}{}{}
\tdplotsetrotatedcoords{0}{-45}{0}
\tdplotdrawarc[tdplot_rotated_coords]{(0,0,0)}{1}{-90}{-35.33}{}{}
\end{scope}
%
\begin{scope}[dotted]
\draw(0.1106,-0.9877,0.1106)--(0.1183,-0.9875,0.1039)--(0.1267,-0.9871,0.0981)--(0.1357,-0.9863,0.0935)--(0.1452,-0.9853,0.0900)--(0.1551,-0.9840,0.0877)--(0.1652,-0.9825,0.0867)--(0.1752,-0.9807,0.0868)--(0.1851,-0.9788,0.0881)--(0.1949,-0.9767,0.0904)--(0.2043,-0.9744,0.0935)--(0.2134,-0.9721,0.0975)--(0.2222,-0.9696,0.1021)--(0.2306,-0.9671,0.1073)--(0.2387,-0.9645,0.1130)--(0.2464,-0.9618,0.1190)--(0.2539,-0.9591,0.1254)--(0.2611,-0.9562,0.1321)--(0.2681,-0.9533,0.1389)--(0.2748,-0.9504,0.1460)--(0.2813,-0.9473,0.1532)--(0.2877,-0.9442,0.1605)--(0.2939,-0.9410,0.1680)--(0.2999,-0.9377,0.1755)--(0.3059,-0.9343,0.1831)--(0.3117,-0.9308,0.1908)--(0.3173,-0.9273,0.1985)--(0.3229,-0.9237,0.2062)--(0.3284,-0.9200,0.2140)--(0.3338,-0.9162,0.2218)--(0.3391,-0.9123,0.2295)--(0.3444,-0.9083,0.2373)--(0.3496,-0.9043,0.2451)--(0.3547,-0.9001,0.2529)--(0.3597,-0.8959,0.2607)--(0.3648,-0.8916,0.2685)--(0.3697,-0.8871,0.2762)--(0.3746,-0.8826,0.2840)--(0.3794,-0.8780,0.2917)--(0.3842,-0.8733,0.2994)--(0.3890,-0.8686,0.3070)--(0.3937,-0.8637,0.3147)--(0.3984,-0.8587,0.3223)--(0.4030,-0.8537,0.3298)--(0.4076,-0.8486,0.3373) ;
\draw(0.5562,-0.6174,0.5562)--(0.5706,-0.6173,0.5417)--(0.5847,-0.6167,0.5270)--(0.5987,-0.6158,0.5122)--(0.6125,-0.6146,0.4972)--(0.6260,-0.6130,0.4820)--(0.6394,-0.6110,0.4667)--(0.6526,-0.6087,0.4513)--(0.6655,-0.6060,0.4358)--(0.6783,-0.6029,0.4201)--(0.6908,-0.5994,0.4044)--(0.7031,-0.5955,0.3886)--(0.7153,-0.5911,0.3728)--(0.7272,-0.5864,0.3569)--(0.7389,-0.5812,0.3410)--(0.7504,-0.5755,0.3251)--(0.7617,-0.5694,0.3093)--(0.7728,-0.5627,0.2936)--(0.7837,-0.5554,0.2779)--(0.7945,-0.5476,0.2625)--(0.8051,-0.5391,0.2472)--(0.8156,-0.5299,0.2323)--(0.8259,-0.5200,0.2178)--(0.8361,-0.5093,0.2038)--(0.8462,-0.4976,0.1904)--(0.8562,-0.4850,0.1778)--(0.8661,-0.4714,0.1663)--(0.8758,-0.4566,0.1561)--(0.8854,-0.4408,0.1476)--(0.8947,-0.4239,0.1410)--(0.9036,-0.4060,0.1367)--(0.9119,-0.3875,0.1351)--(0.9195,-0.3686,0.1363)--(0.9262,-0.3498,0.1404)--(0.9319,-0.3314,0.1473)--(0.9364,-0.3139,0.1567)--(0.9398,-0.2974,0.1683)--(0.9420,-0.2823,0.1818)--(0.9429,-0.2686,0.1968)--(0.9428,-0.2563,0.2132)--(0.9415,-0.2457,0.2306)--(0.9392,-0.2367,0.2487)--(0.9359,-0.2294,0.2675)--(0.9316,-0.2238,0.2866)--(0.9263,-0.2198,0.3059)--(0.9203,-0.2175,0.3253)--(0.9134,-0.2168,0.3445)--(0.9058,-0.2177,0.3634)--(0.8976,-0.2200,0.3819)--(0.8888,-0.2237,0.3999);
\end{scope}
%
\begin{scope} 
\draw(0.1106,-0.9877,0.1106)--(0.1183,-0.9875,0.1039)--(0.1267,-0.9871,0.0981)--(0.1357,-0.9863,0.0935)--(0.1452,-0.9853,0.0900)--(0.1551,-0.9840,0.0877)--(0.1652,-0.9825,0.0867)--(0.1752,-0.9807,0.0868)--(0.1851,-0.9788,0.0881)--(0.1949,-0.9767,0.0904)--(0.2043,-0.9744,0.0935)--(0.2134,-0.9721,0.0975);
\draw(0.5562,-0.6174,0.5562)--(0.5706,-0.6173,0.5417)--(0.5847,-0.6167,0.5270)--(0.5987,-0.6158,0.5122)--(0.6125,-0.6146,0.4972) ;
\end{scope}
%
\begin{scope}[torustrajectory,thick]
\draw(0.4055,-0.8192,0.4055)--(0.4282,-0.8182,0.3837)--(0.4516,-0.8151,0.3628)--(0.4756,-0.8100,0.3430)--(0.5001,-0.8028,0.3246)--(0.5250,-0.7935,0.3076)--(0.5502,-0.7822,0.2924)--(0.5754,-0.7688,0.2791)--(0.6005,-0.7533,0.2681)--(0.6254,-0.7359,0.2595)--(0.6497,-0.7167,0.2536)--(0.6731,-0.6958,0.2506)--(0.6955,-0.6734,0.2505)--(0.7164,-0.6500,0.2535)--(0.7357,-0.6258,0.2592)--(0.7531,-0.6010,0.2677)--(0.7686,-0.5759,0.2786)--(0.7821,-0.5505,0.2918)--(0.7936,-0.5253,0.3071)--(0.8029,-0.5002,0.3241)--(0.8102,-0.4757,0.3426)--(0.8153,-0.4516,0.3623)--(0.8184,-0.4282,0.3832)--(0.8195,-0.4055,0.4050)--(0.8185,-0.3836,0.4277)--(0.8154,-0.3628,0.4511)--(0.8103,-0.3430,0.4751)--(0.8032,-0.3245,0.4996)--(0.7939,-0.3075,0.5245)--(0.7826,-0.2923,0.5497)--(0.7691,-0.2791,0.5749)--(0.7537,-0.2680,0.6001)--(0.7363,-0.2595,0.6249)--(0.7171,-0.2536,0.6492)--(0.6963,-0.2506,0.6726)--(0.6740,-0.2505,0.6949)--(0.6507,-0.2534,0.7158)--(0.6264,-0.2592,0.7352)--(0.6015,-0.2677,0.7527)--(0.5763,-0.2788,0.7682)--(0.5510,-0.2920,0.7817)--(0.5259,-0.3072,0.7931)--(0.5010,-0.3241,0.8025)--(0.4765,-0.3425,0.8097)--(0.4524,-0.3622,0.8149)--(0.4290,-0.3831,0.8180)--(0.4062,-0.4050,0.8191)--(0.3844,-0.4276,0.8182)--(0.3635,-0.4510,0.8152)--(0.3437,-0.4750,0.8101)--(0.3252,-0.4995,0.8029)--(0.3082,-0.5244,0.7937)--(0.2929,-0.5496,0.7824)--(0.2796,-0.5748,0.7691)--(0.2685,-0.5999,0.7537)--(0.2598,-0.6248,0.7363)--(0.2538,-0.6491,0.7171)--(0.2507,-0.6725,0.6963)--(0.2506,-0.6948,0.6741)--(0.2534,-0.7158,0.6507)--(0.2591,-0.7351,0.6265)--(0.2674,-0.7526,0.6017)--(0.2782,-0.7682,0.5766)--(0.2913,-0.7818,0.5513)--(0.3064,-0.7933,0.5261)--(0.3234,-0.8027,0.5011)--(0.3418,-0.8100,0.4764)--(0.3616,-0.8153,0.4523)--(0.3824,-0.8184,0.4289)--(0.4042,-0.8195,0.4062)--cycle;
\end{scope}
\pgfresetboundingbox\useasboundingbox[tdplot_screen_coords] (0,0) circle (1);
\end{tikzpicture}
\caption{Left image: Implementing the shooting method. 
Trajectory (1) exists $B$ through the ``roof'' $\{\alpha=0\}$. 
Trajectory (2) exits $B$ through the ``side'' $\{u(r,\theta)=1\}$.  
Trajectory (3) is the desired curve. 
Right image: Visualisation of the reflection symmetries in the quotient $\Sp^{3n-1}/H$. }
\label{fig:shooting3}
\end{figure}
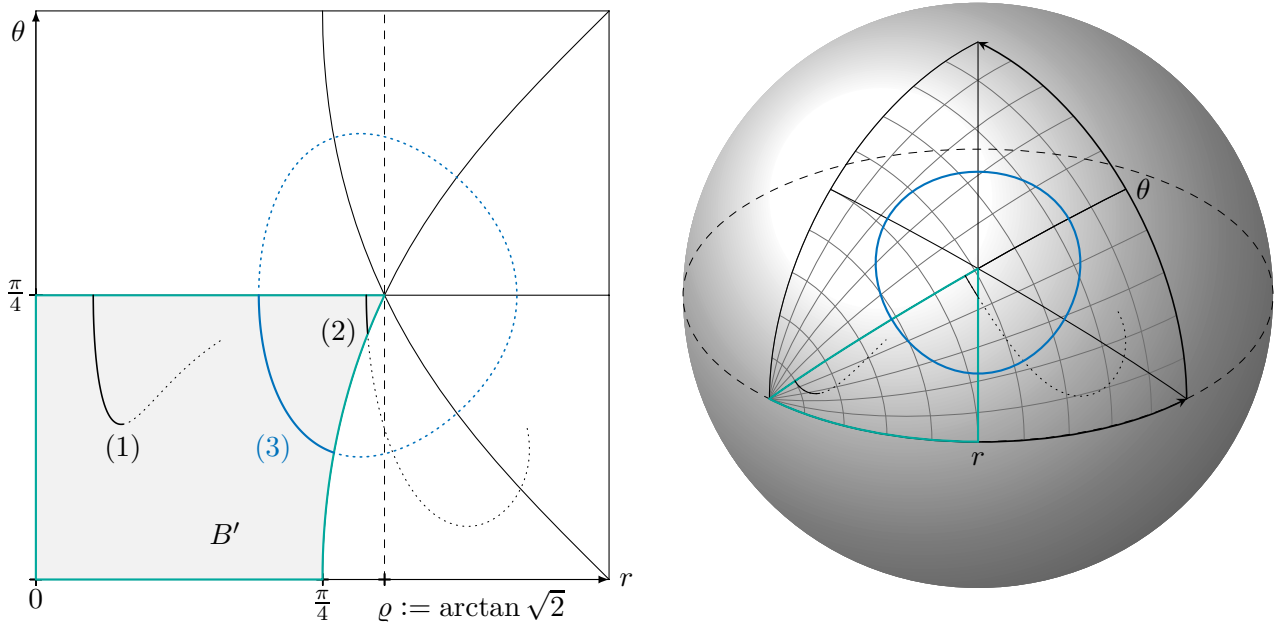

\begin{lemma}[see Figure \ref{fig:shooting3}\;(2)]\label{lem:side-appendix}
For every $\varepsilon>0$ there exists $c>0$ such that if $r_0\in\interval{\varrho-c,\varrho}$ then $\alpha(s)<\varepsilon-\pi/2$ for all $s\in[0,s_*]$ and $\tan(r(s_*))\cos(\theta(s_*))=1$. 
\end{lemma}
 
\begin{proof}
The monotonicity of $r$ and $\theta$ and the definition \eqref{eqn:domainB'} 
imply that for every $\delta>0$ there exists $c>0$ such that if $r_0\geq\varrho-c$ then, $\theta\geq\pi/4-\delta$ for all $s\in\IntervaL{0,s_*}$. 
In fact, solving the equation $\tan(r)\cos(\theta)=1$ for $r=\arctan\bigl(1/\cos(\theta)\bigr)$ we obtain a smooth, strictly increasing function of $\theta$  which attains the value $r=\varrho$ at $\theta=\pi/4$. 
Hence, $0\leq\cot(2\theta)\leq\tan(2\delta)$ 
as in the proof of Lemma~\ref{lem:dipping_down-appendix} but now the estimate holds for all $s\in[0,s_*]$. 
Thus, we may argue as we did for \eqref{eqn:20210920} to conclude that $\alpha(s)$ is arbitrarily close to $-\pi/2$ for all $s\in[0,s_*]$ provided that $\delta>0$ is chosen sufficiently small. 
In particular, this implies that the trajectory must exit the domain $B$ via $\tan(r(s_*))\cos(\theta(s_*))=1$. 
\end{proof}

Now, given these ancillary lemmata, the proof of Theorem \ref{thm:HsiangClaim} is easily obtained. Indeed, we have already shown the existence of trajectories emanating from initial data of the form $(r_0, \pi/4, -\pi/2)$ that exit from the domain $B$ at points in the interior of the sets
\[
(B\cap \left\{\alpha=0 \right\})\cup (B \cap \left\{u(r,\theta)=1, \ \alpha\geq v(r,\theta)\right\}), \ \text{and} \ (B \cap \left\{u(r,\theta)=1, \ \alpha\leq v(r,\theta)\right\})
\]
respectively for $r_0$ sufficiently small, and sufficiently close to the threshold $\varrho$. At this stage, arguing exactly as we did for the proof of Theorem \ref{thm:main} we obtain a trajectory that exits the domain $B$ at the intersection of the two sets in question, which implies the desired conclusion.

\begin{remark}\label{rem:Completeness}
We explicitly note that, as can be readily checked, apart from the Clifford torus in $\Sp^3$, the minimal embeddings of $T^3$ in $\Sp^4$ and of $T^4$ in $\Sp^5$ are the only ones that can be obtained via equivariant methods relying on the outer direct sum of orthogonal groups determining a cohomogeneity two action.
\end{remark}




\begin{bibdiv}
\begin{biblist}

\bib{Alm66}{article}{
      author={Almgren, F.~J., Jr.},
       title={Some interior regularity theorems for minimal surfaces and an
  extension of {B}ernstein's theorem},
        date={1966},
     journal={Ann. of Math. (2)},
      volume={84},
       pages={277\ndash 292},
}

\bib{Angenent1992}{incollection}{
      author={Angenent, Sigurd~B.},
       title={Shrinking doughnuts},
        date={1992},
   booktitle={Nonlinear diffusion equations and their equilibrium states, 3
  ({G}regynog, 1989)},
      series={Progr. Nonlinear Differential Equations Appl.},
      volume={7},
   publisher={Birkh\"{a}user Boston, Boston, MA},
       pages={21\ndash 38},
}

\bib{BdGG69}{article}{
      author={Bombieri, E.},
      author={De~Giorgi, E.},
      author={Giusti, E.},
       title={Minimal cones and the {B}ernstein problem},
        date={1969},
     journal={Invent. Math.},
      volume={7},
       pages={243\ndash 268},
}

\bib{Brendle2013}{article}{
      author={Brendle, Simon},
       title={Embedded minimal tori in {$S^3$} and the {L}awson conjecture},
        date={2013},
     journal={Acta Math.},
      volume={211},
      number={2},
       pages={177\ndash 190},
}

\bib{Bryant1982}{article}{
      author={Bryant, Robert~L.},
       title={Conformal and minimal immersions of compact surfaces into the
  {$4$}-sphere},
        date={1982},
     journal={J. Differential Geometry},
      volume={17},
      number={3},
       pages={455\ndash 473},
}

\bib{Calabi1967}{article}{
      author={Calabi, Eugenio},
       title={Minimal immersions of surfaces in {E}uclidean spheres},
        date={1967},
     journal={J. Differential Geom.},
      volume={1},
       pages={111\ndash 125},
}

\bib{ChoeFra18}{article}{
      author={Choe, Jaigyoung},
      author={Fraser, Ailana},
       title={Mean curvature in manifolds with {R}icci curvature bounded from
  below},
        date={2018},
     journal={Comment. Math. Helv.},
      volume={93},
      number={1},
       pages={55\ndash 69},
}

\bib{Drugan2018}{article}{
      author={Drugan, Gregory},
      author={Lee, Hojoo},
      author={Nguyen, Xuan~Hien},
       title={A survey of closed self-shrinkers with symmetry},
        date={2018},
        ISSN={1422-6383},
     journal={Results Math.},
      volume={73},
      number={1},
       pages={Paper No. 32, 32},
}

\bib{HsiHsi80}{article}{
      author={Hsiang, Wu-Teh},
      author={Hsiang, Wu-Yi},
       title={Examples of codimension-one closed minimal submanifolds in some
  symmetric spaces. {I}},
        date={1980},
     journal={J. Differential Geometry},
      volume={15},
      number={4},
       pages={543\ndash 551},
}

\bib{HsiHsi82}{article}{
      author={Hsiang, Wu-Teh},
      author={Hsiang, Wu-Yi},
       title={On the existence of codimension-one minimal spheres in compact
  symmetric spaces of rank {$2$}. {II}},
        date={1982},
     journal={J. Differential Geometry},
      volume={17},
      number={4},
       pages={583\ndash 594 (1983)},
}

\bib{Hsiang1983}{article}{
      author={Hsiang, Wu-Yi},
       title={Minimal cones and the spherical {B}ernstein problem. {I}},
        date={1983},
     journal={Ann. of Math. (2)},
      volume={118},
      number={1},
       pages={61\ndash 73},
}

\bib{Hsiang1987}{article}{
      author={Hsiang, Wu-Yi},
       title={On the construction of infinitely many congruence classes of
  imbedded closed minimal hypersurfaces in {$S^n(1)$} for all {$n\geq 3$}},
        date={1987},
     journal={Duke Math. J.},
      volume={55},
      number={2},
       pages={361\ndash 367},
}

\bib{HsiLaw71}{article}{
      author={Hsiang, Wu-Yi},
      author={Lawson, H.~Blaine, Jr.},
       title={Minimal submanifolds of low cohomogeneity},
        date={1971},
     journal={J. Differential Geometry},
      volume={5},
       pages={1\ndash 38},
}

\bib{HsiSte86}{article}{
      author={Hsiang, Wu-Yi},
      author={Sterling, Ivan},
       title={Minimal cones and the spherical {B}ernstein problem. {III}},
        date={1986},
     journal={Invent. Math.},
      volume={85},
      number={2},
       pages={223\ndash 247},
}

\bib{IriMarNev18}{article}{
      author={Irie, Kei},
      author={Marques, Fernando~C.},
      author={Neves, Andr\'{e}},
       title={Density of minimal hypersurfaces for generic metrics},
        date={2018},
     journal={Ann. of Math. (2)},
      volume={187},
      number={3},
       pages={963\ndash 972},
}

\bib{KiNak87}{article}{
      author={Ki, U-Hang},
      author={Nakagawa, Hisao},
       title={A characterization of the {C}artan hypersurface in a sphere},
        date={1987},
     journal={Tohoku Math. J. (2)},
      volume={39},
      number={1},
       pages={27\ndash 40},
}

\bib{Lawson1970}{article}{
      author={Lawson, H.~Blaine},
       title={Complete minimal surfaces in {$S^3$}},
        date={1970},
     journal={Ann. of Math. (2)},
      volume={92},
       pages={335\ndash 374},
}

\bib{LioMarNev18}{article}{
      author={Liokumovich, Yevgeny},
      author={Marques, Fernando~C.},
      author={Neves, Andr\'{e}},
       title={Weyl law for the volume spectrum},
        date={2018},
     journal={Ann. of Math. (2)},
      volume={187},
      number={3},
       pages={933\ndash 961},
}

\bib{MarNev17}{article}{
      author={Marques, Fernando~C.},
      author={Neves, Andr\'{e}},
       title={Existence of infinitely many minimal hypersurfaces in positive
  {R}icci curvature},
        date={2017},
     journal={Invent. Math.},
      volume={209},
      number={2},
       pages={577\ndash 616},
}

\bib{MarNevSon19}{article}{
      author={Marques, Fernando~C.},
      author={Neves, Andr\'{e}},
      author={Song, Antoine},
       title={Equidistribution of minimal hypersurfaces for generic metrics},
        date={2019},
     journal={Invent. Math.},
      volume={216},
      number={2},
       pages={421\ndash 443},
}

\bib{Sha17}{article}{
      author={Sharp, Ben},
       title={Compactness of minimal hypersurfaces with bounded index},
        date={2017},
     journal={J. Differential Geom.},
      volume={106},
      number={2},
       pages={317\ndash 339},
}

\bib{Sol90}{article}{
      author={Solomon, Bruce},
       title={The harmonic analysis of cubic isoparametric minimal
  hypersurfaces. {I}. {D}imensions {$3$} and {$6$}},
        date={1990},
     journal={Amer. J. Math.},
      volume={112},
      number={2},
       pages={157\ndash 203},
}

\bib{Son18}{article}{
      author={Song, Antoine},
       title={Existence of infinitely many minimal hypersurfaces in closed
  manifolds},
        date={2023},
     journal={Ann. of Math. (2)},
      volume={197},
      number={3},
       pages={859\ndash 895},
}

\bib{Zhou20}{article}{
      author={Zhou, Xin},
       title={On the multiplicity one conjecture in min-max theory},
        date={2020},
     journal={Ann. of Math. (2)},
      volume={192},
      number={3},
       pages={767\ndash 820},
}

\end{biblist}
\end{bibdiv}
\end{document}